\documentclass[a4paper, 11pt, dvipdfmx]{article}
\usepackage[top=30truemm, bottom=30truemm, left=20truemm, right=20truemm]{geometry}
\usepackage{amsthm}
\usepackage{amsmath}
\usepackage{amssymb}
\usepackage{amsfonts}
\usepackage{mathtools}
\usepackage{bm}
\usepackage{color}
\usepackage[hidelinks]{hyperref}
\newtheoremstyle{mystyle}
	{\baselineskip}
	{\baselineskip}
	{\itshape}
	{}
	{\bfseries}
	{.}
	{1em}
	{}

\theoremstyle{mystyle}
	\newtheorem{lem}{Lemma}[section]
	
	\newtheorem{prop}[lem]{Proposition}
	\newtheorem{thm}[lem]{Theorem}

\newtheoremstyle{mystyle2}
	{\baselineskip}
	{\baselineskip}
	{\normalfont}
	{}
	{\bfseries}
	{.}
	{1em}
	{}

\theoremstyle{mystyle2}
	\newtheorem{rem}[lem]{Remark}
	\newtheorem{definition}[lem]{Definition}
\newcommand{\R}{\mathbb{R}}
\newcommand{\Z}{\mathbb{Z}}
\renewcommand{\AA}{\mathcal{A}}
\newcommand{\MM}{\mathcal{M}}
\newcommand{\RR}{\mathcal{R}}
\newcommand{\abs}[1]{\left\lvert #1 \right\rvert}
\newcommand{\norm}[1]{\left\lVert #1 \right\rVert}
\numberwithin{equation}{section}
\allowdisplaybreaks[4]
\begin{document}
%%%%%%%%%%%%%%%%%%%%%%%%%
\title{
	Higher order asymptotic expansions for the convection-diffusion equation in the Fujita-subcritical case
}

\author{
	Ryunosuke Kusaba\footnote{e-mail: ryu2411501@akane.waseda.jp} \bigskip \\
	Department of Pure and Applied Physics, \\
	Graduate School of Advanced Science and Engineering, \\
	Waseda University, 3-4-1 Okubo, Shinjuku-ku, Tokyo 169-8555, JAPAN
}

\date{\color{red} \textbf{November 2, 2024}}
\date{}

\maketitle
%%%%%%%%%%%%%%%%%%%%%%%%%
\begin{abstract}
	This paper is devoted to the asymptotic behavior of global solutions to the convection-diffusion equation in the Fujita-subcritical case.
	We improve the result by Zuazua (1993) and establish higher order asymptotic expansions with decay estimates of the remainders.
	We also discuss the optimality for the decay rates of the remainders.
\end{abstract}


%35Q56: Semilinear parabolic equations
%35B40: Asymptotic behavior of solutions to PDEs
%35C20: Asymptotic expansions of solutions to PDEs

\tableofcontents
%%%%%%%%%%%%%%%%%%%%%%%%%
\newpage
\section{Introduction}
%%%%%%%%%%%%%%%%%%%%%%%%%

We consider the large time behavior of global solutions to the following Cauchy problem for the convection-diffusion equation:
\begin{align*}
	\tag{P} \label{P}
	\begin{cases}
		\partial_{t} u- \Delta u=a \cdot \nabla f \left( u \right), &\qquad \left( t, x \right) \in \left( 0, + \infty \right) \times \R^{n}, \\
		u \left( 0 \right) =u_{0}, &\qquad x \in \R^{n},
	\end{cases}
\end{align*}
where $u \colon \left[ 0, + \infty \right) \times \R^{n} \to \R$ is an unknown function, $u_{0} \colon \R^{n} \to \R$ is a given data at $t=0$, $a= \left( a_{1}, \ldots, a_{n} \right) \in \R^{n} \setminus \left\{ 0 \right\}$ is a given constant vector, and $f \in C^{1} \left( \R \right) \setminus \left\{ 0 \right\}$ is a nonlinear function satisfying $f \left( 0 \right) =0$ and
\begin{align*}
	\abs{f \left( \xi \right) -f \left( \eta \right)} \leq C \left( \abs{\xi}^{p-1} + \abs{\eta}^{p-1} \right) \abs{\xi - \eta}, \qquad \forall \xi, \eta \in \R
\end{align*}
for some $C>0$ and $p \in \left( 1, + \infty \right)$.
Typical examples are given by homogeneous functions of order $p$ such as
\begin{align*}
	f \left( \xi \right) = \xi^{p}, {\ } \abs{\xi}^{p}, {\ } \abs{\xi}^{p-1} \xi.
\end{align*}
We remark that the first example above makes sense if $p$ is an integer or if $u_{0}$ is non-negative and has a sufficient regularity by virtue of the comparison principle.
The convection-diffusion equation is known as a model of fluid dynamics with the conservation law of the mass:
\begin{align*}
	\int_{\R^{n}} u \left( t, x \right) dx= \int_{\R^{n}} u_{0} \left( x \right) dx, \qquad \forall t>0.
\end{align*}

To review the previous works, we introduce some notation.
For each $q \in \left[ 1, + \infty \right]$, let $L^{q} \left( \R^{n} \right)$ denote the standard Lebesgue space with the norm denoted by $\norm{\,\cdot\,}_{q}$.
We also use the weighted $L^{1}$-space defined by
\begin{align*}
	L_{m}^{1} \left( \R^{n} \right) \coloneqq \left\{ \varphi \in L^{1} \left( \R^{n} \right); {\,} \text{$x^{\alpha} \varphi \in L^{1} \left( \R^{n} \right)$ for all $\alpha \in \Z_{\geq 0}^{n}$ with $\abs{\alpha} \leq m$} \right\}
\end{align*}
for each $m \in \Z_{>0}$, where $x^{\alpha} \varphi$ means the function $\R^{n} \ni x \mapsto x^{\alpha} \varphi \left( x \right) \in \R$.
We define the heat semigroup $\left( e^{t \Delta}; t \geq 0 \right)$ by
\begin{align*}
	e^{t \Delta} \varphi \coloneqq \begin{cases}
		G_{t} \ast \varphi, &\qquad t>0, \\
		\varphi, &\qquad t=0
	\end{cases}
\end{align*}
for $\varphi \in L^{q} \left( \R^{n} \right)$ with $q \in \left[ 1, + \infty \right]$, where $G_{t} \colon \R^{n} \to \R$ is the Gauss kernel given by
\begin{align*}
	G_{t} \left( x \right) = \left( 4 \pi t \right)^{- \frac{n}{2}} \exp \left( - \frac{\abs{x}^{2}}{4t} \right), \qquad x \in \R^{n}
\end{align*}
and $\ast$ is the convolution in $\R^{n}$.

We first give the definition of global solutions to \eqref{P}.

\begin{definition}
	Let $u_{0} \in \left( L^{1} \cap L^{\infty} \right) \left( \R^{n} \right)$.
	A function
	\begin{align*}
		u \in X \coloneqq \left( C \cap L^{\infty} \right) \left( \left[ 0, + \infty \right); L^{1} \left( \R^{n} \right) \right) \cap \left( C \cap L^{\infty} \right) \left( \left( 0, + \infty \right); L^{\infty} \left( \R^{n} \right) \right)
	\end{align*}
	is said to be a global solution to \eqref{P} if the integral equation
	\begin{align*}
		\tag{I} \label{I}
		u \left( t \right) =e^{t \Delta} u_{0} + \int_{0}^{t} a \cdot \nabla e^{\left( t-s \right) \Delta} f \left( u \left( s \right) \right) ds
	\end{align*}
	holds in $\left( L^{1} \cap L^{\infty} \right) \left( \R^{n} \right)$ for all $t>0$.
\end{definition}

The global well-posedness of \eqref{P} and the decay estimate of the global solution were established by Escobedo--Zuazua \cite{Escobedo-Zuazua}.
See also the lecture note by Zuazua \cite{Zuazua2020}.

\begin{prop}[{\cite[Proposition 1]{Escobedo-Zuazua}, \cite[Theorems 4.1 and 4.2]{Zuazua2020}}] \label{pro:P_global}
	Let $p \in \left( 1, + \infty \right)$ and let $u_{0} \in \left( L^{1} \cap L^{\infty} \right) \left( \R^{n} \right)$.
	Then, \eqref{P} has a unique global solution $u \in X$, which satisfies
	\begin{align}
		\label{eq:P_decay}
		\sup_{q \in \left[ 1, + \infty \right]} \sup_{t>0} t^{\frac{n}{2} \left( 1- \frac{1}{q} \right)} \norm{u \left( t \right)}_{q} \leq C \norm{u_{0}}_{1}
	\end{align}
	for some $C>0$ independent of $u_{0}$.
	Moreover, we have
	\begin{align}
		\label{eq:P_regular}
		u \in \bigcap_{q \in \left( 1, + \infty \right)} \left[ C \left( \left( 0, + \infty \right); W^{2, q} \left( \R^{n} \right) \right) \cap C^{1} \left( \left( 0, + \infty \right); L^{q} \left( \R^{n} \right) \right) \right].
	\end{align}
\end{prop}

We emphasize that for the global well-posedness of \eqref{P}, the sizes of the exponent $p$ and the initial data $u_{0}$ are not of concern.
This is a crucial difference between the convection-diffusion equation and the Fujita equation:
\begin{align}
	\label{eq:Fujita}
	\begin{cases}
		\partial_{t} u- \Delta u=u^{p}, &\qquad \left( t, x \right) \in \left( 0, + \infty \right) \times \R^{n}, \\
		u \left( 0 \right) =u_{0} \geq 0, &\qquad x \in \R^{n}.
	\end{cases}
\end{align}
As for the Fujita equation, we have some thresholds for the sizes of the exponent $p$ and the initial data $u_{0}$ which separate the existence and nonexistence of global solutions to \eqref{eq:Fujita} (cf. \cite{Fujita, Kawanago1996}).

In this paper, we are interested in the asymptotic behavior of the global solution to \eqref{P} given by Proposition \ref{pro:P_global}.
The starting point of our study is the result by Escobedo--Zuazua \cite{Escobedo-Zuazua}, which describes the 0th order asymptotic expansion with decay estimates of the remainder.

\begin{prop}[{\cite[Theorem 2]{Escobedo-Zuazua}, \cite[Theorem 5.3]{Zuazua2020}}] \label{pro:P_asymp-0}
	Let $p>1+1/n$.
	Let $u_{0} \in \left( L^{1}_{1} \cap L^{\infty} \right) \left( \R^{n} \right)$ and let $u \in X$ be the global solution to \eqref{P} given in Proposition \ref{pro:P_global}.
	Then, for any $q \in \left[ 1, + \infty \right]$, there exists $C>0$ such that the estimates
	\begin{align*}
		t^{\frac{n}{2} \left( 1- \frac{1}{q} \right)} \norm{u \left( t \right) - \AA_{0} \left( t \right)}_{q} \leq \begin{dcases}
			Ct^{- \sigma} &\qquad \text{if} \quad 1+ \frac{1}{n} <p<1+ \frac{2}{n}, \\
			Ct^{- \frac{1}{2}} \log \left( 2+t \right) &\qquad \text{if} \quad p=1+ \frac{2}{n}, \\
			Ct^{- \frac{1}{2}} &\qquad \text{if} \quad p>1+ \frac{2}{n}
		\end{dcases}
	\end{align*}
	hold for all $t>1$, where
	\begin{align*}
		\sigma &\coloneqq \frac{n}{2} \left( p-1 \right) - \frac{1}{2} >0, \\
		\AA_{0} \left( t \right) &\coloneqq \Lambda_{0, 0} \left( t; u_{0} \right) = \MM_{0} \left( u_{0} \right) \delta_{t} G_{1}, \\
		\MM_{0} \left( u_{0} \right) &\coloneqq \int_{\R^{n}} u_{0} \left( x \right) dx,
	\end{align*}
	and $\delta_{t}$ is the dilation acting on functions $\varphi$ on $\R^{n}$ as
	\begin{align*}
		\left( \delta_{t} \varphi \right) \left( x \right) =t^{- \frac{n}{2}} \varphi \left( t^{- \frac{1}{2}} x \right), \qquad x \in \R^{n}.
	\end{align*}
\end{prop}

We remark that the exponent $1+1/n$ is the critical exponent which gives a threshold for the asymptotic behavior of the global solution to \eqref{P}.
In fact, as for the case where $f \left( \xi \right) = \abs{\xi}^{p-1} \xi$, we have the following classification:
\begin{itemize}
	\item[(1)]
		If $p<1+1/n$, then the global solution to \eqref{P} behaves in the large time like the self-similar solution to the following reduced equation:
		\begin{align*}
			\partial_{t} v- \Delta_{\perp a} v=a \cdot \nabla \left( \abs{v}^{p-1} v \right), \qquad \left( t, x \right) \in \left( 0, + \infty \right) \times \R^{n}
		\end{align*}
		with $\MM_{0} \left( v \left( t \right) \right) = \MM_{0} \left( u_{0} \right)$ for all $t>0$, where $\Delta_{\perp a}$ is the Laplacian on the hyperplane orthogonal to the constant vector $a$ (see \cite{Escobedo-Vazquez-Zuazua_1, Escobedo-Vazquez-Zuazua_2, Carpio}).
	\item[(2)]
		If $p=1+1/n$, then the global solution to \eqref{P} behaves in the large time like the self-similar solution to the (nonlinear) convection-diffusion equation
		\begin{align*}
			\partial_{t} v- \Delta v=a \cdot \nabla \left( \abs{v}^{\frac{1}{n}} v \right), \qquad \left( t, x \right) \in \left( 0, + \infty \right) \times \R^{n}
		\end{align*}
		with $\MM_{0} \left( v \left( t \right) \right) = \MM_{0} \left( u_{0} \right)$ for all $t>0$ (see \cite{Escobedo-Zuazua, Zuazua2020}).
	\item[(3)]
		If $p>1+1/n$, then the global solution to \eqref{P} behaves in the large time like the self-similar solution to the linear heat equation
		\begin{align*}
			\partial_{t} v- \Delta v=0, \qquad \left( t, x \right) \in \left( 0, + \infty \right) \times \R^{n}
		\end{align*}
		with $\MM_{0} \left( v \left( t \right) \right) = \MM_{0} \left( u_{0} \right)$ for all $t>0$ (see \cite{Escobedo-Zuazua, Zuazua2020}).
\end{itemize}
We can also determine the exponent $1+1/n$ using the rescaling argument which leaves both the equation in \eqref{P} and the $L^{1}$-norm of the solution at $t=0$ invariant (see \cite{Escobedo-Zuazua}).
Proposition \ref{pro:P_asymp-0} implies (3) by taking into account the fact that $\AA_{0} \left( t \right)$ has the self-similarity described as $\AA_{0} \left( t \right) = \delta_{t} \AA_{0} \left( 1 \right)$.
We can understand (2) more explicitly in the context of the viscous Burgers equation:
\begin{align}
	\label{eq:Burgers}
	\begin{cases}
		\partial_{t} u- \partial_{x}^{2} u=a \partial_{x} \left( u^{2} \right), &\qquad \left( t, x \right) \in \left( 0, + \infty \right) \times \R, \\
		u \left( 0 \right) =u_{0}, &\qquad x \in \R.
	\end{cases}
\end{align}
It is known that the global solution to \eqref{eq:Burgers} converges to the nonlinear diffusion wave $\chi_{t}$ given by
\begin{align*}
	\chi_{t} \left( x \right) =- \frac{1}{a} \frac{\left( \exp \left( -a \MM_{0} \left( u_{0} \right) \right) -1 \right) G_{t} \left( x \right)}{1+ \left( \exp \left( -a \MM_{0} \left( u_{0} \right) \right) -1 \right) \int_{x}^{+ \infty} G_{t} \left( y \right) dy},
\end{align*}
which has the self-similarity described as $\chi_{t} = \delta_{t} \chi_{1}$ (cf. \cite{Hopf, Cole}).

From Proposition \ref{pro:P_asymp-0}, we see that the decay rate of the remainder $u \left( t \right) - \AA_{0} \left( t \right)$ may change depending on the size of the exponent $p$ relative to the Fujita exponent $1+2/n$.
In particular, since the asymptotic profile $\AA_{0} \left( t \right)$ is completely determined by the moment of the initial data, the effect of the nonlinear convection may appear only in the decay rate of the remainder.
We remark that $p<1+2/n$ if and only if $\sigma <1/2$.
To be certain that this observation is indeed true, we have to prove the optimality for the decay rates of the remainder given in Proposition \ref{pro:P_asymp-0}.
Here, the decay rate of the remainder is said to be optimal if we obtain a lower estimate of the remainder with the same decay rate for some initial data.

After the result by \cite{Escobedo-Zuazua}, Zuazua \cite{Zuazua1993} derived the first order asymptotics for each case: $1+1/n<p<1+2/n$; $p=1+2/n$; $p>1+2/n$.
Similar results can be seen in \cite{Karch, Duro-Carpio, Fukuda} for related equations.

\begin{prop}[{\cite[Theorem 1]{Zuazua1993}}] \label{pro:P_asymp_F-sub}
	Let $1+1/n<p<1+2/n$.
	Let $u_{0} \in \left( L^{1}_{1} \cap L^{\infty} \right) \left( \R^{n} \right)$ and let $u \in X$ be the global solution to \eqref{P} given in Proposition \ref{pro:P_global}.
	Then,
	\begin{align*}
		\lim_{t \to + \infty} t^{\frac{n}{2} \left( 1- \frac{1}{q} \right) + \sigma} \norm{u \left( t \right) - \AA_{0, 1} \left( t \right)}_{q} =0
	\end{align*}
	holds for any $q \in \left[ 1, + \infty \right]$, where
	\begin{align*}
		\AA_{0, 1} \left( t \right) \coloneqq \AA_{0} \left( t \right) + \int_{0}^{t} a \cdot \nabla e^{\left( t-s \right) \Delta} f \left( \AA_{0} \left( s \right) \right) ds.
	\end{align*}
\end{prop}

\begin{prop}[{\cite[Theorem 2]{Zuazua1993}}] \label{pro:P_asymp_F-critical}
	Let $p=1+2/n$.
	Let $u_{0} \in \left( L^{1}_{1} \cap L^{\infty} \right) \left( \R^{n} \right)$ and let $u \in X$ be the global solution to \eqref{P} given in Proposition \ref{pro:P_global}.
	Then, for any $q \in \left[ 1, + \infty \right]$, there exists $C>0$ such that the estimate
	\begin{align*}
		t^{\frac{n}{2} \left( 1- \frac{1}{q} \right)} \norm{u \left( t \right) - \widetilde{\AA}_{0, 1} \left( t \right)}_{q} \leq Ct^{- \frac{1}{2}}
	\end{align*}
	holds for all $t>2$, where
	\begin{align*}
		\widetilde{\AA}_{0, 1} \left( t \right) &\coloneqq \AA_{0} \left( t \right) + \sum_{j=1}^{n} a_{j} \int_{1}^{t} \Lambda_{e_{j}, 0} \left( t; f \left( \AA_{0} \left( s \right) \right) \right) ds, \\
		\Lambda_{e_{j}, 0} \left( t; \varphi \right) &\coloneqq - \frac{1}{2} t^{- \frac{1}{2}} \MM_{0} \left( \varphi \right) \delta_{t} \left( x_{j} G_{1} \right).
	\end{align*}
\end{prop}

\begin{prop}[{\cite[Theorem 3]{Zuazua1993}}] \label{pro:P_asymp_F-super}
	Let $p>1+2/n$.
	Let $u_{0} \in \left( L^{1}_{1} \cap L^{\infty} \right) \left( \R^{n} \right)$ and let $u \in X$ be the global solution to \eqref{P} given in Proposition \ref{pro:P_global}.
	Then,
	\begin{align*}
		\lim_{t \to + \infty} t^{\frac{n}{2} \left( 1- \frac{1}{q} \right) + \frac{1}{2}} \norm{u \left( t \right) - \AA_{1, 0} \left( t \right)}_{q} =0
	\end{align*}
	holds for any $q \in \left[ 1, + \infty \right]$, where
	\begin{align*}
		\AA_{1, 0} \left( t \right) &\coloneqq \Lambda_{0, 1} \left( t; u_{0} \right) + \sum_{j=1}^{n} a_{j} \Lambda_{e_{j}, 0} \left( t; \psi_{0, 0} \right), \\
		\Lambda_{0, 1} \left( t; u_{0} \right) &\coloneqq \Lambda_{0, 0} \left( t; u_{0} \right) + \frac{1}{2} t^{- \frac{1}{2}} \sum_{j=1}^{n} \MM_{e_{j}} \left( u_{0} \right) \delta_{t} \left( x_{j} G_{1} \right), \\
		\MM_{e_{j}} \left( u_{0} \right) &\coloneqq \int_{\R^{n}} x_{j} u_{0} \left( x \right) dx, \\
		\psi_{0, 0} &\coloneqq \int_{0}^{+ \infty} f \left( u \left( s \right) \right) ds.
	\end{align*}
\end{prop}

\begin{rem}
	Propositions \ref{pro:P_asymp_F-critical} and \ref{pro:P_asymp_F-super} are improved versions of the results in \cite{Zuazua1993}.
	The author in \cite{Zuazua1993} obtained the corresponding results under a stronger assumption for the initial data, namely, $u_{0} \in L^{2} {\,} ( \R^{n}; e^{\abs{x}^{2} /4} dx ) \cap L^{\infty} \left( \R^{n} \right)$.
	In particular, for the Fujita-critical case, the author showed only that
	\begin{align*}
		\lim_{t \to + \infty} \frac{t^{\frac{n}{2} \left( 1- \frac{1}{q} \right) + \frac{1}{2}}}{\log t} \norm{u \left( t \right) - \widetilde{\AA}_{0, 1} \left( t \right)}_{q} =0
	\end{align*}
	holds for any $q \in \left[ 1, + \infty \right]$.
	In Appendix \ref{app:Zuazua}, we give the proofs of Propositions \ref{pro:P_asymp_F-critical} and \ref{pro:P_asymp_F-super} by more direct approaches than those in \cite{Zuazua1993}.
\end{rem}

By using not only Propositions \ref{pro:P_asymp_F-sub}, \ref{pro:P_asymp_F-critical}, and \ref{pro:P_asymp_F-super} but also the self-similar structures in the asymptotic profiles, we can show that the decay rates of the remainder $u \left( t \right) - \AA_{0} \left( t \right)$ given in Proposition \ref{pro:P_asymp-0} are optimal under a suitable assumption for $f$.
For the reader's convenience, we give the proof of the optimality in Appendix \ref{app:optimal}.
As a result, we conclude that the Fujita exponent $1+2/n$ is the critical exponent for the asymptotic behavior of the remainder $u \left( t \right) - \AA_{0} \left( t \right)$.

The above discussion leads us to the next question on the optimality for the decay rates of the remainders for the first order asymptotic expansions given in Propositions \ref{pro:P_asymp_F-sub}, \ref{pro:P_asymp_F-critical}, and \ref{pro:P_asymp_F-super}.
In the Fujita-supercritical case, the decay rate of the remainder $u \left( t \right) - \AA_{1, 0} \left( t \right)$ given in Proposition \ref{pro:P_asymp_F-super} seems optimal since it is the same as that of the first order asymptotic expansion for the heat semigroup.
In fact, if $\varphi \in L^{1}_{1} \left( \R^{n} \right)$, then
\begin{align*}
	\lim_{t \to + \infty} t^{\frac{n}{2} \left( 1- \frac{1}{q} \right) + \frac{1}{2}} \norm{e^{t \Delta} \varphi - \Lambda_{0, 1} \left( t; \varphi \right)}_{q} =0
\end{align*}
holds for any $q \in \left[ 1, + \infty \right]$ (see Proposition \ref{pro:heat_asymp_lim} in Section \ref{sec:heat}).
In the Fujita-critical case, Fukuda--Sato \cite{Fukuda-Sato} established the second order asymptotic expansion whose remainder vanishes as fast as that of the first order asymptotic expansion for the heat semigroup.
In addition, they showed that the decay rate of the remainder $u \left( t \right) - \widetilde{\AA}_{0, 1} \left( t \right)$ given in Proposition \ref{pro:P_asymp_F-critical} is optimal provided that $f$ is homogeneous of order $p$.
By considering the Fujita-subcritical case similarly, the decay rate of the remainder $u \left( t \right) - \AA_{0, 1} \left( t \right)$ given in Proposition \ref{pro:P_asymp_F-sub} does not seem optimal.
In fact, Ishige--Kawakami \cite{Ishige-Kawakami2013} obtained the following proposition.

\begin{prop}[{\cite[Theorem 6.1]{Ishige-Kawakami2013}}] \label{pro:Ishige-Kawakami}
	Let $p>1+1/n$.
	Let $u_{0} \in \left( L^{1}_{1} \cap L^{\infty} \right) \left( \R^{n} \right)$ and let $u \in X$ be the global solution to \eqref{P} given in Proposition \ref{pro:P_global}.
	Then,
	\begin{align*}
		t^{\frac{n}{2} \left( 1- \frac{1}{q} \right)} \norm{u \left( t \right) - \AA \left( t \right)}_{q} = \begin{dcases}
			O \left( t^{-2 \sigma} \right) &\qquad \text{if} \quad 2 \sigma < \frac{1}{2}, \\
			O \left( t^{- \frac{1}{2}} \log t \right) &\qquad \text{if} \quad 2 \sigma = \frac{1}{2}, \\
			o \left( t^{- \frac{1}{2}} \right) &\qquad \text{if} \quad 2 \sigma > \frac{1}{2}
		\end{dcases}
	\end{align*}
	hold for any $q \in \left[ 1, + \infty \right]$ as $t \to + \infty$, where
	\begin{align*}
		\AA \left( t \right) &\coloneqq \AA_{0} \left( 1+t \right) + \int_{0}^{t} a \cdot \nabla e^{\left( t-s \right) \Delta} f \left( \AA_{0} \left( 1+s \right) \right) ds+ \frac{1}{2} \left( 1+t \right)^{- \frac{1}{2}} \sum_{j=1}^{n} \MM_{e_{j}} \left( \Psi_{0} \left( t \right) \right) \delta_{1+t} \left( x_{j} G_{1} \right), \\
		\Psi_{0} \left( t \right) &\coloneqq u \left( t \right) - \int_{0}^{t} a \cdot \nabla f \left( \AA_{0} \left( 1+s \right) \right) ds.
	\end{align*}
\end{prop}

This proposition suggests that there may be room for improvement in the decay rate of the remainder $u \left( t \right) - \AA_{0, 1} \left( t \right)$ given by Proposition \ref{pro:P_asymp_F-sub}.
However, since there is no information in \cite{Ishige-Kawakami2013} on the relation between $\AA_{0, 1} \left( t \right)$ and $\AA \left( t \right)$, we do not know whether it is possible or not.
Even if we can do it, it seems impossible to show the optimality of the decay rate in any case, due to the lack of the self-similarity of the asymptotic profile $\AA \left( t \right)$.
For these reasons, we focus on the Fujita-subcritical case, namely, $1+1/n<p<1+2/n$.
In this paper, we improve the decay rate of the remainder $u \left( t \right) - \AA_{0, 1} \left( t \right)$ given by Proposition \ref{pro:P_asymp_F-sub} without reconstructing the asymptotic profile as in \cite{Ishige-Kawakami2013}.
We also establish higher order asymptotic expansions and discuss the optimality and non-optimality for the decay rates of the remainders in some cases.

The rest of this paper is organized as follows.
In the next section, we state our main results: Theorems \ref{th:P_asymp_sub_gene}, \ref{th:P_asymp_critical_gene}, \ref{th:P_asymp_super_gene}, \ref{th:P_asymp_super_optimal}, \ref{th:P_asymp_sub-1_optimal}, and \ref{th:P_asymp_critical-1_optimal} below.
In Section \ref{sec:heat}, we introduce basic estimates and asymptotic expansions of the heat semigroup.
In Sections \ref{sec:P_asymp_sub_gene}, \ref{sec:P_asymp_critical_gene}, \ref{sec:P_asymp_super_gene}, and \ref{sec:P_asymp_super_optimal}, we give the proofs of Theorems \ref{th:P_asymp_sub_gene}, \ref{th:P_asymp_critical_gene}, \ref{th:P_asymp_super_gene}, and \ref{th:P_asymp_super_optimal}, respectively.
In Section \ref{sec:P_asymp-1_optimal}, we present the proofs of Theorems \ref{th:P_asymp_sub-1_optimal} and \ref{th:P_asymp_critical-1_optimal}.

%%%%%%%%%%%%%%%%%%%%%%%%%
\section{Main results}
%%%%%%%%%%%%%%%%%%%%%%%%%

The following three theorems are improvements and generalizations of Propositions \ref{pro:P_asymp_F-sub}, \ref{pro:P_asymp_F-critical}, and \ref{pro:P_asymp_F-super}, respectively.

\begin{thm} \label{th:P_asymp_sub_gene}
	Let $k \in \Z_{>0}$ and let $1+ \frac{1}{n} <p<1+ \frac{k+1}{kn}$.
	Let $u_{0} \in \left( L^{1}_{1} \cap L^{\infty} \right) \left( \R^{n} \right)$ and let $u \in X$ be the global solution to \eqref{P} given in Proposition \ref{pro:P_global}.
	Then, for any $q \in \left[ 1, + \infty \right]$, there exists $C>0$ such that the estimates
	\begin{align*}
		t^{\frac{n}{2} \left( 1- \frac{1}{q} \right)} \norm{u \left( t \right) - \AA_{0, k} \left( t \right)}_{q} \leq \begin{dcases}
			Ct^{- \left( k+1 \right) \sigma} &\qquad \text{if} \quad \left( k+1 \right) \sigma < \frac{1}{2}, \\
			Ct^{- \frac{1}{2}} \log \left( 2+t \right) &\qquad \text{if} \quad \left( k+1 \right) \sigma = \frac{1}{2}, \\
			Ct^{- \frac{1}{2}} &\qquad \text{if} \quad \left( k+1 \right) \sigma > \frac{1}{2}
		\end{dcases}
	\end{align*}
	hold for all $t>2^{k}$, where
	\begin{align*}
		\sigma &\coloneqq \frac{n}{2} \left( p-1 \right) - \frac{1}{2} \in \left( 0, \frac{1}{2k} \right), \\
		\AA_{0, k} \left( t \right) &\coloneqq \begin{dcases}
			\AA_{0} \left( t \right), &\quad k=0, \\
			\AA_{0} \left( t \right) + \int_{0}^{1} a \cdot \nabla e^{\left( t-s \right) \Delta} f \left( \AA_{0} \left( s \right) \right) ds+ \int_{1}^{t} a \cdot \nabla e^{\left( t-s \right) \Delta} f \left( \AA_{0, k-1} \left( s \right) \right) ds, &\quad k \geq 1.
		\end{dcases}
	\end{align*}
\end{thm}

\begin{thm} \label{th:P_asymp_critical_gene}
	Let $k \in \Z_{>0}$ and let $p=1+ \frac{k+2}{\left( k+1 \right) n}$.
	Let $u_{0} \in \left( L^{1}_{1} \cap L^{\infty} \right) \left( \R^{n} \right)$ and let $u \in X$ be the global solution to \eqref{P} given in Proposition \ref{pro:P_global}.
	Then, for any $q \in \left[ 1, + \infty \right]$, there exists $C>0$ such that the estimate
	\begin{align*}
		t^{\frac{n}{2} \left( 1- \frac{1}{q} \right)} \norm{u \left( t \right) - \widetilde{\AA}_{0, k+1} \left( t \right)}_{q} \leq Ct^{- \frac{1}{2}}
	\end{align*}
	holds for all $t>2^{k+1}$, where
	\begin{align*}
		\widetilde{\AA}_{0, k+1} \left( t \right) &\coloneqq \AA_{0, k} \left( t \right) + \sum_{j=1}^{n} a_{j} \int_{1}^{t} \Lambda_{e_{j}, 0} \left( t; f \left( \AA_{0, k} \left( s \right) \right) -f \left( \AA_{0, k-1} \left( s \right) \right) \right) ds.
	\end{align*}
\end{thm}

\begin{thm} \label{th:P_asymp_super_gene}
	Let $k \in \Z_{>0}$ and let $1+ \frac{k+2}{\left( k+1 \right) n} <p<1+ \frac{k+1}{kn}$.
	Let $u_{0} \in \left( L^{1}_{1} \cap L^{\infty} \right) \left( \R^{n} \right)$ and let $u \in X$ be the global solution to \eqref{P} given in Proposition \ref{pro:P_global}.
	Then,
	\begin{align*}
		\lim_{t \to + \infty} t^{\frac{n}{2} \left( 1- \frac{1}{q} \right) + \frac{1}{2}} \norm{u \left( t \right) - \AA_{1, k} \left( t \right)}_{q} =0
	\end{align*}
	holds for any $q \in \left[ 1, + \infty \right]$, where
	\begin{align*}
		\AA_{1, k} \left( t \right) &\coloneqq \AA_{0, k} \left( t \right) + \frac{1}{2} t^{- \frac{1}{2}} \sum_{j=1}^{n} \MM_{e_{j}} \left( u_{0} \right) \delta_{t} \left( x_{j} G_{1} \right) + \sum_{j=1}^{n} a_{j} \Lambda_{e_{j}, 0} \left( t; \psi_{0, k} \right) \\
		&= \Lambda_{0, 1} \left( t; u_{0} \right) + \left( \AA_{0, k} \left( t \right) - \AA_{0} \left( t \right) \right) + \sum_{j=1}^{n} a_{j} \Lambda_{e_{j}, 0} \left( t; \psi_{0, k} \right), \\
		\psi_{0, k} &\coloneqq \int_{0}^{1} \left( f \left( u \left( s \right) \right) -f \left( \AA_{0} \left( s \right) \right) \right) ds+ \int_{1}^{+ \infty} \left( f \left( u \left( s \right) \right) -f \left( \AA_{0, k-1} \left( s \right) \right) \right) ds.
	\end{align*}
\end{thm}

Theorem \ref{th:P_asymp_sub_gene} with $k=1$ means that the decay rate of the remainder $u \left( t \right) - \AA_{0, 1} \left( t \right)$ given in Proposition \ref{pro:P_asymp_F-sub} can be improved and that it may change depending on the size of $2 \sigma$ relative to $1/2$, namely, the size of the exponent $p$ relative to $1+3/ \left( 2n \right)$.
This result leads us to the natural question whether the decay rates of the remainder $u \left( t \right) - \AA_{0, 1} \left( t \right)$ given in Theorem \ref{th:P_asymp_sub_gene} with $k=1$ are optimal or not.
To answer this question, we establish higher order asymptotic expansions for each case, which are described in the above theorems for arbitrary $k \in \Z_{>0}$.
In particular, we find that there may be many thresholds for the exponent $p$ determined by the size relations between an integer multiple of $\sigma$ and $1/2$, which separate the asymptotic behavior of the remainders.
Here, we note that
\begin{align*}
	p<1+ \frac{k+2}{\left( k+1 \right) n} \quad \Longleftrightarrow \quad \left( k+1 \right) \sigma < \frac{1}{2}.
\end{align*}
We also emphasize that we do not need to assume the exponential decay at the far field for the initial data as in \cite{Zuazua1993} and to use the weighted estimates of the global solution to \eqref{P} developed in \cite{Ishige-Kawakami2013}.

Now, we explain how to determine the asymptotic profiles $\AA_{0, k} \left( t \right)$, $\widetilde{\AA}_{0, k+1} \left( t \right)$, and $\AA_{1, k} \left( t \right)$.
For this purpose, we start with the integral equation \eqref{I} associated with \eqref{P} and focus on the Duhamel term in \eqref{I}:
\begin{align*}
	\int_{0}^{t} a \cdot \nabla e^{\left( t-s \right) \Delta} f \left( u \left( s \right) \right) ds.
\end{align*}
Roughly speaking, we have two choices to extract the leading term of the Duhamel term: The first is the approximation of the heat semigroup and the second is the approximation of the nonlinear term.
The asymptotic profile $\AA_{0, k} \left( t \right)$ given in Theorem \ref{th:P_asymp_sub_gene} comes from the approximation of the nonlinear term recursively.
We assume that the asymptotic profile $\AA_{0, k-1} \left( t \right)$ is determined for some $k \in \Z_{>0}$.
Since the difference $u \left( t \right) - \AA_{0, k-1} \left( t \right)$ decays faster than both $u \left( t \right)$ and $\AA_{0, k-1} \left( t \right)$, we determine the next leading term of the Duhamel term to make the difference $f \left( u \left( t \right) \right) -f \left( \AA_{0, k-1} \left( t \right) \right)$, which yields the difference $u \left( t \right) - \AA_{0, k-1} \left( t \right)$, in the representation of the remainder.
One of the candidates for the leading term is
\begin{align*}
	\int_{0}^{t} a \cdot \nabla e^{\left( t-s \right) \Delta} f \left( \AA_{0, k-1} \left( s \right) \right) ds,
\end{align*}
whose remainder is written as
\begin{align*}
	&\int_{0}^{t} a \cdot \nabla e^{\left( t-s \right) \Delta} f \left( u \left( s \right) \right) ds- \int_{0}^{t} a \cdot \nabla e^{\left( t-s \right) \Delta} f \left( \AA_{0, k-1} \left( s \right) \right) ds \\
	&\hspace{1cm} = \int_{0}^{t} a \cdot \nabla e^{\left( t-s \right) \Delta} \left( f \left( u \left( s \right) \right) -f \left( \AA_{0, k-1} \left( s \right) \right) \right) ds.
\end{align*}
On the other hand, the stronger decay in time makes the stronger singularity near $t=0$.
Moreover, for $k \geq 2$, the singularity of $\AA_{0, k-1} \left( t \right)$ near $t=0$ is too strong to ensure the integrability of
\begin{align*}
	\int_{0}^{1} a \cdot \nabla e^{\left( t-s \right) \Delta} f \left( \AA_{0, k-1} \left( s \right) \right) ds
\end{align*}
for $t>1$.
Therefore, we need to modify the leading term as
\begin{align*}
	\int_{0}^{1} a \cdot \nabla e^{\left( t-s \right) \Delta} f \left( \AA_{0} \left( s \right) \right) ds + \int_{1}^{t} a \cdot \nabla e^{\left( t-s \right) \Delta} f \left( \AA_{0, k-1} \left( s \right) \right) ds,
\end{align*}
which is just our asymptotic profile of the Duhamel term given in Theorem \ref{th:P_asymp_sub_gene}.
By this modification, the remainder is represented as
\begin{align*}
	&\int_{0}^{t} a \cdot \nabla e^{\left( t-s \right) \Delta} f \left( u \left( s \right) \right) ds- \left( \int_{0}^{1} a \cdot \nabla e^{\left( t-s \right) \Delta} f \left( \AA_{0} \left( s \right) \right) ds + \int_{1}^{t} a \cdot \nabla e^{\left( t-s \right) \Delta} f \left( \AA_{0, k-1} \left( s \right) \right) ds \right) \\
	&\hspace{1cm} = \int_{0}^{1} a \cdot \nabla e^{\left( t-s \right) \Delta} f \left( u \left( s \right) \right) ds- \int_{0}^{1} a \cdot \nabla e^{\left( t-s \right) \Delta} f \left( \AA_{0} \left( s \right) \right) ds \\
	&\hspace{3cm} + \int_{1}^{t} a \cdot \nabla e^{\left( t-s \right) \Delta} \left( f \left( u \left( s \right) \right) -f \left( \AA_{0, k-1} \left( s \right) \right) \right) ds.
\end{align*}
We emphasize that the first and second terms on the right hand side of the above identity should be estimated separately without using the decay estimate of the difference $u \left( t \right) - \AA_{0} \left( t \right)$.
This is the crucial point to improve the decay rate of the remainder $u \left( t \right) - \AA_{0, 1} \left( t \right)$ given by Proposition \ref{pro:P_asymp_F-sub} in particular.
Based on the above explicit formula of the remainder, we determine the asymptotic profiles $\widetilde{\AA}_{0, k+1} \left( t \right)$ and $\AA_{1, k} \left( t \right)$ given in Theorems \ref{th:P_asymp_critical_gene} and \ref{th:P_asymp_super_gene} by the approximation of the heat semigroup.
Concretely, we use the identity
\begin{align*}
	e^{\left( t-s \right) \Delta} =e^{t \Delta} + \int_{0}^{1} \left( -s \Delta \right) e^{\left( t-s \theta \right) \Delta} d \theta
\end{align*}
in the sense of operators.
To derive the profile $\widetilde{\AA}_{0, k+1} \left( t \right)$, we need to modify the leading term taking into account the fact that the difference $\AA_{k} \left( t \right) - \AA_{0, k-1} \left( t \right)$ decays as fast as the difference $u \left( t \right) - \AA_{0, k-1} \left( t \right)$.
We remark that approximations similar to those employed to determine the profiles $\AA_{0, k} \left( t \right)$ and $\AA_{1, k} \left( t \right)$ can be seen in \cite{Ishige-Kawakami2012, Ishige-Kawakami2013, Ishige-Kawakami-Kobayashi, Ishige-Kawakami2024}, while the approximation used to derive the profile $\widetilde{\AA}_{0, k+1} \left( t \right)$ cannot.

Once we succeed in constructing the higher order asymptotic expansions, we get a chance to show the optimality for the decay rates of the remainder $u \left( t \right) - \AA_{0, k} \left( t \right)$ given in Theorem \ref{th:P_asymp_sub_gene}.
Under suitable assumptions, we can derive
\begin{align*}
	t^{\frac{n}{2} \left( 1- \frac{1}{q} \right)} \norm{\AA_{0, k+1} \left( t \right) - \AA_{0, k} \left( t \right)}_{q} &\leq Ct^{- \left( k+1 \right) \sigma}, \\
	t^{\frac{n}{2} \left( 1- \frac{1}{q} \right)} \norm{\widetilde{\AA}_{0, k+1} \left( t \right) - \AA_{0, k} \left( t \right)}_{q} &\leq Ct^{- \frac{1}{2}} \log t, \\
	t^{\frac{n}{2} \left( 1- \frac{1}{q} \right)} \norm{\AA_{1, k} \left( t \right) - \AA_{0, k} \left( t \right)}_{q} &\leq Ct^{- \frac{1}{2}}
\end{align*}
for any $q \in \left[ 1, + \infty \right]$ and $t>1$.
For details, see Lemma \ref{lem:P_asymp_sub_gene}, Remarks \ref{rem:P_asymp_critical_gene} and \ref{rem:P_asymp_super_gene}, respectively.
From these observations, we infer that the decay rates of the remainder $u \left( t \right) - \AA_{0, k} \left( t \right)$ given in Theorem \ref{th:P_asymp_sub_gene} are optimal in any case.
We prove this conjecture in the case where $1+ \frac{k+2}{\left( k+1 \right) n} <p<1+ \frac{k+1}{kn}$.

\begin{thm} \label{th:P_asymp_super_optimal}
	Under the same assumption as in Theorem \ref{th:P_asymp_super_gene},
	\begin{align*}
		\lim_{t \to + \infty} t^{\frac{n}{2} \left( 1- \frac{1}{q} \right) + \frac{1}{2}} \norm{u \left( t \right) - \AA_{0, k} \left( t \right)}_{q} = \norm{\AA_{1, k} \left( 1 \right) - \AA_{0, k} \left( 1 \right)}_{q}
	\end{align*}
	holds for any $q \in \left[ 1, + \infty \right]$.
	In particular, if there exists $j \in \left\{ 1, \ldots, n \right\}$ such that
	\begin{align*}
		\MM_{e_{j}} \left( u_{0} \right) -a_{j} \MM_{0} \left( \psi_{0, k} \right) \neq 0,
	\end{align*}
	then $\AA_{1, k} \left( 1 \right) - \AA_{0, k} \left( 1 \right) \not\equiv 0$, and therefore
	\begin{align*}
		t^{\frac{n}{2} \left( 1- \frac{1}{q} \right)} \norm{u \left( t \right) - \AA_{0, k} \left( t \right)}_{q} =t^{- \frac{1}{2}} \norm{\AA_{1, k} \left( 1 \right) - \AA_{0, k} \left( 1 \right)}_{q} \left( 1+o \left( 1 \right) \right)
	\end{align*}
	for all $q \in \left[ 1, + \infty \right]$ as $t \to + \infty$.
\end{thm}

For $j \in \left\{ 1, \ldots, n \right\}$, the condition $\MM_{e_{j}} \left( u_{0} \right) -a_{j} \MM_{0} \left( \psi_{0, k} \right) \neq 0$ holds for all $u_{0} \in \left( L^{1}_{1} \cap L^{\infty} \right) \left( \R^{n} \right)$ satisfying $\MM_{e_{j}} \left( u_{0} \right) \neq 0$ if $a_{j} =0$ and for sufficiently small $u_{0} \in \left( L^{1}_{1} \cap L^{\infty} \right) \left( \R^{n} \right)$ with $\MM_{e_{j}} \left( u_{0} \right) \neq 0$ if $a_{j} \neq 0$ (cf. \cite[Appendix A]{Kusaba-Ozawa_CGL}).
Hence, Theorem \ref{th:P_asymp_super_optimal} implies that the decay rate of the remainder $u \left( t \right) - \AA_{0, k} \left( t \right)$ given in Theorem \ref{th:P_asymp_sub_gene} with $1+ \frac{k+2}{\left( k+1 \right) n} <p<1+ \frac{k+1}{kn}$ is optimal.

Theorem \ref{th:P_asymp_super_optimal} follows from the stratification structure of the asymptotic profile:
\begin{align*}
	\AA_{1, k} \left( t \right) - \AA_{0, k} \left( t \right) =t^{- \frac{1}{2}} \delta_{t} \left( \AA_{1, k} \left( 1 \right) - \AA_{0, k} \left( 1 \right) \right).
\end{align*}
In the above identity, the negative power $t^{-1/2}$ means the decay of the asymptotic amplitude, the dilation $\delta_{t}$ means the self-similarity of the asymptotic profile, and the function $\AA_{1, k} \left( 1 \right) - \AA_{0, k} \left( 1 \right)$ on $\R^{n}$ means the shape of the asymptotic profile, which is represented explicitly as
\begin{align*}
	\AA_{1, k} \left( 1 \right) - \AA_{0, k} \left( 1 \right) = \frac{1}{2} \sum_{j=1}^{n} \left( \MM_{e_{j}} \left( u_{0} \right) -a_{j} \MM_{0} \left( \psi_{0, k} \right) \right) x_{j} G_{1}.
\end{align*}
For details, see Lemma \ref{lem:P_asymp_super_structure}.

We also prove the optimality for the decay rate of the remainder $u \left( t \right) - \AA_{0, 1} \left( t \right)$ given in Theorem \ref{th:P_asymp_sub_gene} with $k=1$ and $1+1/n<p<1+3/ \left( 2n \right)$.

\begin{thm} \label{th:P_asymp_sub-1_optimal}
	Let $1+1/n<p<1+3/ \left( 2n \right)$.
	Let $u_{0} \in \left( L^{1}_{1} \cap L^{\infty} \right) \left( \R^{n} \right)$ and let $u \in X$ be the global solution to \eqref{P} given in Proposition \ref{pro:P_global}.
	Assume that $f$ is given by $f \left( \xi \right) = \abs{\xi}^{p-1} \xi$ for all $\xi \in \R$.
	Then,
	\begin{align*}
		\lim_{t \to + \infty} t^{\frac{n}{2} \left( 1- \frac{1}{q} \right) +2 \sigma} \norm{u \left( t \right) - \AA_{0, 1} \left( t \right)}_{q} = \norm{\RR_{0, 2}^{*}}_{q}
	\end{align*}
	holds for any $q \in \left[ 1, + \infty \right]$, where
	\begin{align*}
		\RR_{0, 2}^{*} &\coloneqq \int_{0}^{1} a \cdot \nabla e^{\left( 1- \theta \right) \Delta} \left( \left( \AA_{0, 1} \left( \theta \right) - \AA_{0} \left( \theta \right) \right) f' \left( \AA_{0} \left( \theta \right) \right) \right) d \theta \\
		&=f \left( \MM_{0} \left( u_{0} \right) \right) f' \left( \MM_{0} \left( u_{0} \right) \right) \int_{0}^{1} a \cdot \nabla e^{\left( 1- \theta \right) \Delta} \left( G_{\theta}^{p-1} \int_{0}^{\theta} a \cdot \nabla e^{\left( \theta - \tau \right) \Delta} G_{\tau}^{p} d \tau \right) d \theta.
	\end{align*}
	In particular, if $f \left( \MM_{0} \left( u_{0} \right) \right) f' \left( \MM_{0} \left( u_{0} \right) \right) \neq 0$, then $\RR_{0, 2}^{*} \not\equiv 0$, and hence
	\begin{align*}
		t^{\frac{n}{2} \left( 1- \frac{1}{q} \right)} \norm{u \left( t \right) - \AA_{0, 1} \left( t \right)}_{q} =t^{-2 \sigma} \norm{\RR_{0, 2}^{*}}_{q} \left( 1+o \left( 1 \right) \right)
	\end{align*}
	for all $q \in \left[ 1, + \infty \right]$ as $t \to + \infty$.
\end{thm}

The condition $f \left( \MM_{0} \left( u_{0} \right) \right) f' \left( \MM_{0} \left( u_{0} \right) \right) \neq 0$ holds for any $u_{0} \in \left( L^{1}_{1} \cap L^{\infty} \right) \left( \R^{n} \right)$ with $\MM_{0} \left( u_{0} \right) \neq 0$.
In practice, we have to show $S_{0, 2}^{*} \not\equiv 0$, where
\begin{align*}
	S_{0, 2}^{*} \left( x \right) \coloneqq \left( \int_{0}^{1} a \cdot \nabla e^{\left( 1- \theta \right) \Delta} \left( G_{\theta}^{p-1} \int_{0}^{\theta} a \cdot \nabla e^{\left( \theta - \tau \right) \Delta} G_{\tau}^{p} d \tau \right) d \theta \right) \left( x \right), \qquad x \in \R^{n}.
\end{align*}
This follows from the fact that $S_{0, 2}^{*} \left( 0 \right) <0$.
For details, see Remark \ref{rem:S_02}.
The function $t^{-2 \sigma} \delta_{t} \RR_{0, 2}^{*}$ is the asymptotic self-similar profile of the difference $\AA_{0, 2} \left( t \right) - \AA_{0, 1} \left( t \right)$.
More precisely, for any $q \in \left[ 1, + \infty \right]$, we obtain
\begin{align*}
	\AA_{0, 2} \left( t \right) - \AA_{0, 1} \left( t \right) =t^{-2 \sigma} \delta_{t} \RR_{0, 2}^{*} +o \left( t^{- \frac{n}{2} \left( 1- \frac{1}{q} \right) -2 \sigma} \right)
\end{align*}
in $L^{q} \left( \R^{n} \right)$ as $t \to + \infty$.
In the case where $f \left( \xi \right) = \abs{\xi}^{p-1} \xi$, we see that $f'$ is of class $C^{1}$ if $p>2$ and H\"{o}lder continuous of order $\left( p-1 \right)$ if $1<p \leq 2$, which plays an important role in the proof of Theorem \ref{th:P_asymp_sub-1_optimal}.
Conversely, if $f'$ has this property, we can derive the same result as in Theorem \ref{th:P_asymp_sub-1_optimal} (see Remark \ref{rem:th_optimal} below).

Under the same assumption for $f$, we have the following theorem.

\begin{thm} \label{th:P_asymp_critical-1_optimal}
	Let $p=1+3/ \left( 2n \right)$.
	Let $u_{0} \in \left( L^{1}_{1} \cap L^{\infty} \right) \left( \R^{n} \right)$ and let $u \in X$ be the global solution to \eqref{P} given in Proposition \ref{pro:P_global}.
	Assume that $f$ is given by $f \left( \xi \right) = \abs{\xi}^{p-1} \xi$ for all $\xi \in \R$.
	Then, for any $q \in \left[ 1, + \infty \right]$, there exists $C>0$ such that the estimate
	\begin{align*}
		t^{\frac{n}{2} \left( 1- \frac{1}{q} \right)} \norm{u \left( t \right) - \AA_{0, 1} \left( t \right)}_{q} \leq Ct^{- \frac{1}{2}}
	\end{align*}
	holds for all $t>4$.
\end{thm}

Since there is no restriction on the initial data, Theorem \ref{th:P_asymp_critical-1_optimal} means the non-optimality for the decay rate of the remainder $u \left( t \right) - \AA_{0, 1} \left( t \right)$ given in Theorem \ref{th:P_asymp_sub_gene} with $k=1$ and $p=1+3/ \left( 2n \right)$.
This is a consequence of the facts that for any $q \in \left[ 1, + \infty \right]$,
\begin{align*}
	\widetilde{\AA}_{0, 2} \left( t \right) - \AA_{0, 1} \left( t \right) =t^{- \frac{1}{2}} \left( \log t \right) \delta_{t} \widetilde{\RR}_{0, 2}^{*} +O \left( t^{- \frac{n}{2} \left( 1- \frac{1}{q} \right) - \frac{1}{2}} \right)
\end{align*}
holds in $L^{q} \left( \R^{n} \right)$ as $t \to + \infty$ and that $\widetilde{\RR}_{0, 2}^{*} \equiv 0$, where
\begin{align*}
	\widetilde{\RR}_{0, 2}^{*} \coloneqq - \frac{1}{2} \MM_{0} \left( \left( \AA_{0, 1} \left( 1 \right) - \AA_{0} \left( 1 \right) \right) f' \left( \AA_{0} \left( 1 \right) \right) \right) \sum_{j=1}^{n} a_{j} x_{j} G_{1}.
\end{align*}
We can regard Theorem \ref{th:P_asymp_critical-1_optimal} as an improvement of the result in Proposition \ref{pro:Ishige-Kawakami} with $p=1+3/ \left( 2n \right)$ $\left( \Leftrightarrow 2 \sigma =1/2 \right)$, although there is no clear relation between $\AA_{0, 1} \left( t \right)$ and $\AA \left( t \right)$.
From Theorem \ref{th:P_asymp_super_optimal} with $k=1$ and from Theorem \ref{th:P_asymp_sub-1_optimal}, we see that the exponent $1+3/ \left( 2n \right)$ is the critical exponent for the asymptotic behavior of the remainder $u \left( t \right) - \AA_{0, 1} \left( t \right)$.
In addition, Theorem \ref{th:P_asymp_critical-1_optimal} suggests that it is not sufficient to check only the upper estimate of the remainder for the optimality.

It is still open whether the decay rates of the remainder $u \left( t \right) - \AA_{0, k} \left( t \right)$ given in Theorem \ref{th:P_asymp_sub_gene} with $k \geq 2$ and $1+ \frac{1}{n} <p \leq 1+ \frac{k+2}{\left( k+1 \right) n}$ are optimal or not, due to the lack of fine structures in the differences $\AA_{0, k+1} \left( t \right) - \AA_{0, k} \left( t \right)$ and $\widetilde{\AA}_{0, k+1} \left( t \right) - \AA_{0, k} \left( t \right)$.
In addition, the optimality for the decay rate of the remainder $u \left( t \right) - \widetilde{\AA}_{0, k+1} \left( t \right)$ given by Theorem \ref{th:P_asymp_critical_gene} and that of the remainder $u \left( t \right) - \AA_{0, 1} \left( t \right)$ given by Theorem \ref{th:P_asymp_critical-1_optimal} is also an open problem.
To solve this problem, we need to establish the asymptotic expansion of $u \left( t \right) - \widetilde{\AA}_{0, k+1} \left( t \right)$ as in \cite{Fukuda-Sato}.
We note that the decay rate of the remainder $u \left( t \right) - \AA_{1, k} \left( t \right)$ given in Theorem \ref{th:P_asymp_super_gene} seems optimal from the viewpoint of the first order asymptotic expansion for the heat semigroup.

\begin{rem}
	Even if we assume only $u_{0} \in L^{1} \left( \R^{n} \right)$, we can construct a unique global solution
	\begin{align*}
		u \in \left( C \cap L^{\infty} \right) \left( \left[ 0, + \infty \right); L^{1} \left( \R^{n} \right) \right) \cap C \left( \left( 0, + \infty \right); L^{\infty} \left( \R^{n} \right) \right)
	\end{align*}
	to \eqref{P}, which satisfies \eqref{eq:P_decay} and \eqref{eq:P_regular} (cf. \cite{Escobedo-Zuazua, Zuazua2020}).
	Therefore, in Proposition \ref{pro:P_asymp-0} with $1+1/n<p<1+2/n$ and in Theorems \ref{th:P_asymp_sub_gene}, \ref{th:P_asymp_critical_gene}, \ref{th:P_asymp_super_gene}, \ref{th:P_asymp_super_optimal}, \ref{th:P_asymp_sub-1_optimal}, and \ref{th:P_asymp_critical-1_optimal}, we can relax the assumption for the initial data from $u_{0} \in \left( L^{1}_{1} \cap L^{\infty} \right) \left( \R^{n} \right)$ to $u_{0} \in L^{1}_{1} \left( \R^{n} \right)$.
	However, to obtain the results in Proposition \ref{pro:P_asymp-0} with $p \geq 1+2/n$, Propositions \ref{pro:P_asymp_F-critical} and \ref{pro:P_asymp_F-super}, we need to assume $u_{0} \in \left( L^{1}_{1} \cap L^{\infty} \right) \left( \R^{n} \right)$ for example in order to remove the singularity in time of the global solution near $t=0$.
\end{rem}

\begin{rem}
	Under the same assumption as in Theorem \ref{th:P_asymp_sub_gene}, for any $q \in \left[ 1, + \infty \right]$, there exists $C>0$ such that the estimate
	\begin{align*}
		t^{\frac{n}{2} \left( 1- \frac{1}{q} \right)} \norm{\int_{0}^{1} a \cdot \nabla e^{\left( t-s \right) \Delta} f \left( \AA_{0} \left( s \right) \right) ds}_{q} \leq Ct^{- \frac{1}{2}}
	\end{align*}
	holds for all $t>2$.
	Hence, Theorems \ref{th:P_asymp_sub_gene}, \ref{th:P_asymp_critical_gene}, and \ref{th:P_asymp_super_gene} are true even if we define $\AA_{0, k} \left( t \right)$ as
	\begin{align*}
		\AA_{0, k} \left( t \right) \coloneqq \AA_{0} \left( t \right) + \int_{1}^{t} a \cdot \nabla e^{\left( t-s \right) \Delta} f \left( \AA_{0, k-1} \left( s \right) \right) ds
	\end{align*}
	for $k \geq 1$.
	In this case, $\psi_{0, k}$ in Theorem \ref{th:P_asymp_super_gene} is given by
	\begin{align*}
		\psi_{0, k} \coloneqq \int_{0}^{1} f \left( u \left( s \right) \right) ds+ \int_{1}^{+ \infty} \left( f \left( u \left( s \right) \right) -f \left( \AA_{0, k-1} \left( s \right) \right) \right) ds.
	\end{align*}
	On the other hand, by using $\AA_{0, k} \left( t \right)$ defined in Theorem \ref{th:P_asymp_sub_gene}, $\AA_{0, 1} \left( t \right)$ makes sense for all $t>0$.
	This gives a good structure of the difference $\AA_{0, 1} \left( t \right) - \AA_{0} \left( t \right)$, which plays a crucial role in the proofs of Theorems \ref{th:P_asymp_sub-1_optimal} and \ref{th:P_asymp_critical-1_optimal}.
\end{rem}

\begin{rem} \label{rem:th_optimal}
	Theorems \ref{th:P_asymp_sub-1_optimal} and \ref{th:P_asymp_critical-1_optimal} are valid for $f \in C^{1} \left( \R \right) \setminus \left\{ 0 \right\}$ satisfying the following two conditions:
	\begin{itemize}
		\item[(1)]
			$f$ is homogeneous of order $p$, namely,
			\begin{align*}
				f \left( \lambda \xi \right) = \lambda^{p} f \left( \xi \right)
			\end{align*}
			for any $\lambda >0$ and $\xi \in \R$.
		\item[(2)]
			There exist $C>0$ and $p \in \left( 1, + \infty \right)$ such that the estimates
			\begin{align*}
				\abs{f \left( \xi \right) -f \left( \eta \right)} &\leq C \left( \abs{\xi}^{p-1} + \abs{\eta}^{p-1} \right) \abs{\xi - \eta}, \\
				\abs{f' \left( \xi \right) -f' \left( \eta \right)} &\leq \begin{cases}
					C \left( \abs{\xi}^{p-2} + \abs{\eta}^{p-2} \right) \abs{\xi - \eta} &\qquad \text{if} \quad p>2, \\
					C \abs{\xi - \eta}^{p-1} &\qquad \text{if} \quad 1<p \leq 2
				\end{cases}
			\end{align*}
			hold for all $\xi, \eta \in \R$.
	\end{itemize}
	In particular, $f$ given by $f \left( \xi \right) = \abs{\xi}^{p}$ satisfies the above conditions.
\end{rem}

%%%%%%%%%%%%%%%%%%%%%%%%%
\section{Preliminaries} \label{sec:heat}
%%%%%%%%%%%%%%%%%%%%%%%%%

In this section, we introduce basic properties of the heat semigroup.
For details, we refer the readers to \cite{Kusaba-Ozawa2023, Kusaba-Ozawa_CGL} and the references therein.
Let $\Z_{>0}$ be the set of positive integers and let $\Z_{\geq 0} \coloneqq \Z_{>0} \cup \left\{ 0 \right\}$.
For $\alpha = \left( \alpha_{1}, \ldots, \alpha_{n} \right) \in \Z_{\geq 0}^{n}$ and $x= \left( x_{1}, \ldots, x_{n} \right) \in \R^{n}$, we define
\begin{align*}
	\abs{\alpha} \coloneqq \sum_{j=1}^{n} \alpha_{j}, \qquad \alpha ! \coloneqq \prod_{j=1}^{n} \alpha_{j} !, \qquad x^{\alpha} \coloneqq \prod_{j=1}^{n} x_{j}^{\alpha_{j}}, \qquad \partial^{\alpha} = \partial_{x}^{\alpha} \coloneqq \prod_{j=1}^{n} \partial_{j}^{\alpha_{j}}, \qquad \partial_{j} \coloneqq \frac{\partial}{\partial x_{j}}.
\end{align*}
For $\alpha = \left( \alpha_{1}, \ldots, \alpha_{n} \right)$, $\beta = \left( \beta_{1}, \ldots, \beta_{n} \right) \in \Z_{\geq 0}^{n}$, $\alpha \leq \beta$ means that $\alpha_{j} \leq \beta_{j}$ holds for any $j \in \left\{ 1, \ldots, n \right\}$.

The Gauss kernel $G_{t}$ has the self-similarity of the form
\begin{align*}
	G_{t} \left( x \right) =t^{- \frac{n}{2}} G_{1} \left( t^{- \frac{1}{2}} x \right).
\end{align*}
Based on this structure, we define the dilation $\delta_{t}$ by
\begin{align*}
	\left( \delta_{t} \varphi \right) \left( x \right) =t^{- \frac{n}{2}} \varphi \left( t^{- \frac{1}{2}} x \right), \qquad \varphi \in L_{\mathrm{loc}}^{1} \left( \R^{n} \right), {\ } x \in \R^{n}
\end{align*}
for each $t>0$.
Then, the family of the dilations $\left( \delta_{t}; t>0 \right)$ has the following properties:
\begin{itemize}
	\item[(1)]
		$\delta_{t} \delta_{s} = \delta_{ts}$ for any $t, s>0$.
	\item[(2)]
		$\norm{\delta_{t} \varphi}_{q} =t^{- \frac{n}{2} \left( 1- \frac{1}{q} \right)} \norm{\varphi}_{q}$ for any $t>0$, $q \in \left[ 1, + \infty \right]$, and $\varphi \in L^{q} \left( \R^{n} \right)$.
	\item[(3)]
		$\MM_{0} \left( \delta_{t} \varphi \right) = \MM_{0} \left( \varphi \right)$ for any $t>0$ and $\varphi \in L^{1} \left( \R^{n} \right)$.
\end{itemize}
By using the dilation $\delta_{t}$, the self-similarity of $G_{t}$ is described as
\begin{align*}
	G_{t} = \delta_{t} G_{1}.
\end{align*}
In addition, we have
\begin{align*}
	\partial^{\alpha} e^{t \Delta} \varphi = \left( \partial^{\alpha} G_{t} \right) \ast \varphi =t^{- \frac{\abs{\alpha}}{2}} \left( \delta_{t} \left( \partial^{\alpha} G_{1} \right) \right) \ast \varphi
\end{align*}
for any $t>0$ and $\alpha \in \Z_{\geq 0}^{n}$.
From the above identity, we obtain the following $L^{p}$-$L^{q}$ estimate of the heat semigroup.

\begin{lem}[{cf. \cite[Lemma 2.1]{Kusaba-Ozawa_CGL}}] \label{lem:heat_LpLq}
	Let $1 \leq q \leq p \leq + \infty$, $\alpha \in \Z_{\geq 0}^{n}$, and $\varphi \in L^{q} \left( \R^{n} \right)$.
	Then, the estimate
	\begin{align*}
		\norm{\partial^{\alpha} e^{t \Delta} \varphi}_{p} \leq t^{- \frac{n}{2} \left( \frac{1}{q} - \frac{1}{p} \right) - \frac{\abs{\alpha}}{2}} \norm{\partial^{\alpha} G_{1}}_{r} \norm{\varphi}_{q}
	\end{align*}
	holds for any $t>0$, where $r \in \left[ 1, + \infty \right]$ with $1/p+1=1/r+1/q$.
\end{lem}

We also need the following weighted $L^{1}$-estimate for the derivative of the heat semigroup, which can be shown by the same argument as in \cite[Lemma 2.1]{Ishige-Ishiwata-Kawakami}.

\begin{lem} \label{lem:heat_weight}
	There exists $C>0$ such that the estimate
	\begin{align*}
		\sum_{j=1}^{n} \norm{\abs{x} \partial_{j} e^{t \Delta} \varphi}_{1} \leq C \left( t^{- \frac{1}{2}} \norm{\abs{x} \varphi}_{1} + \norm{\varphi}_{1} \right)
	\end{align*}
	holds for any $\varphi \in L^{1}_{1} \left( \R^{n} \right)$ and $t>0$.
\end{lem}

We next consider asymptotic expansions of the heat semigroup.
We can derive them from the Taylor expansions of the Gauss kernel appearing in the integral representation of the heat semigroup.
To write the asymptotic profiles explicitly, we introduce the multi-variable Hermite polynomial
\begin{align*}
	\bm{h}_{\alpha} \left( x \right) \coloneqq \sum_{2 \beta \leq \alpha} \frac{\left( -1 \right)^{\abs{\beta}} \alpha !}{\beta ! \left( \alpha -2 \beta \right) !} x^{\alpha -2 \beta}, \qquad x \in \R^{n}
\end{align*}
for each $\alpha \in \Z_{\geq 0}^{n}$.
We recall that the Hermite polynomials are orthogonal in the following sense:
\begin{align*}
	\int_{\R^{n}} \bm{h}_{\alpha} \left( x \right) \bm{h}_{\beta} \left( x \right) G_{1} \left( x \right) dx= \begin{cases}
		2^{\abs{\alpha}} \alpha ! &\qquad \text{if} \quad \alpha = \beta, \\
		0 &\qquad \text{if} \quad \alpha \neq \beta.
	\end{cases}
\end{align*}
By using the Hermite polynomial $\bm{h}_{\alpha}$, the derivative of the Gauss kernel $\partial^{\alpha} G_{t}$ is represented as
\begin{align*}
	\partial^{\alpha} G_{t} =t^{- \frac{\abs{\alpha}}{2}} \delta_{t} \left( \partial^{\alpha} G_{1} \right) = \left( -2 \right)^{- \abs{\alpha}} t^{- \frac{\abs{\alpha}}{2}} \delta_{t} \left( \bm{h}_{\alpha} G_{1} \right).
\end{align*}
Combining the above identity with the Taylor expansions of the Gauss kernel yields the following two propositions which describe the asymptotic expansions of the heat semigroup.

\begin{prop}[{\cite[Proposition 2.3]{Kusaba-Ozawa_CGL}}] \label{pro:heat_asymp}
	Let $m \in \Z_{\geq 0}$, $\varphi \in L^{1}_{m+1} \left( \R^{n} \right)$, and $\alpha \in \Z_{\geq 0}^{n}$.
	Then, the estimate
	\begin{align*}
		t^{\frac{n}{2} \left( 1- \frac{1}{q} \right) + \frac{\abs{\alpha} +m}{2}} \norm{\partial^{\alpha} e^{t \Delta} \varphi - \Lambda_{\alpha, m} \left( t; \varphi \right)}_{q} \leq 2^{- \left( \abs{\alpha} +m+1 \right)} t^{- \frac{1}{2}} \sum_{\abs{\beta} =m+1} \frac{1}{\beta !} \norm{\bm{h}_{\alpha + \beta} G_{1}}_{q} \lVert x^{\beta} \varphi \rVert_{1}
	\end{align*}
	holds for any $q \in \left[ 1, + \infty \right]$ and $t>0$, where
	\begin{align*}
		\Lambda_{\alpha, m} \left( t; \varphi \right) &\coloneqq \left( -2 \right)^{- \abs{\alpha}} t^{- \frac{\abs{\alpha}}{2}} \sum_{k=0}^{m} 2^{-k} t^{- \frac{k}{2}} \sum_{\abs{\beta} =k} \MM_{\beta} \left( \varphi \right) \delta_{t} \left( \bm{h}_{\alpha + \beta} G_{1} \right), \\
		\MM_{\beta} \left( \varphi \right) &\coloneqq \frac{1}{\beta !} \int_{\R^{n}} x^{\beta} \varphi \left( x \right) dx.
	\end{align*}
\end{prop}

\begin{prop}[{\cite[Proposition 2.4]{Kusaba-Ozawa_CGL}}] \label{pro:heat_asymp_lim}
	Let $m \in \Z_{\geq 0}$, $\varphi \in L_{m}^{1} \left( \R^{n} \right)$, and $\alpha \in \Z_{\geq 0}^{n}$.
	Then,
	\begin{align*}
		\lim_{t \to + \infty} t^{\frac{n}{2} \left( 1- \frac{1}{q} \right) + \frac{\abs{\alpha} +m}{2}} \norm{\partial^{\alpha} e^{t \Delta} \varphi - \Lambda_{\alpha, m} \left( t; \varphi \right)}_{q} =0
	\end{align*}
	holds for any $q \in \left[ 1, + \infty \right]$.
\end{prop}

\begin{rem}
	We easily see that $\bm{h}_{0} \equiv 1$ and $\bm{h}_{e_{j}} \equiv x_{j}$, where $e_{j}$ is a multi-index with $\abs{e_{j}} =1$ whose components are $0$ except for the $j$th coordinate.
	Therefore, the asymptotic profiles $\Lambda_{\alpha, m} \left( t; \varphi \right)$ with $\left( \alpha, m \right) = \left( 0, 0 \right), \left( e_{j}, 0 \right), \left( 0, 1 \right)$ are represented as
	\begin{align*}
		\Lambda_{0, 0} \left( t; \varphi \right) &= \MM_{0} \left( \varphi \right) \delta_{t} G_{1}, \\
		\Lambda_{e_{j}, 0} \left( t; \varphi \right) &=- \frac{1}{2} t^{- \frac{1}{2}} \MM_{0} \left( \varphi \right) \delta_{t} \left( x_{j} G_{1} \right), \\
		\Lambda_{0, 1} \left( t; \varphi \right) &= \MM_{0} \left( \varphi \right) \delta_{t} G_{1} + \frac{1}{2} t^{- \frac{1}{2}} \sum_{j=1}^{n} \MM_{e_{j}} \left( \varphi \right) \delta_{t} \left( x_{j} G_{1} \right),
	\end{align*}
	respectively.
\end{rem}

%%%%%%%%%%%%%%%%%%%%%%%%%
\section{Proofs of main results}
%%%%%%%%%%%%%%%%%%%%%%%%%
\subsection{Proof of Theorem \ref{th:P_asymp_sub_gene}} \label{sec:P_asymp_sub_gene}
%%%%%%%%%%%%%%%%%%%%%%%%%

We first prove the following lemma which gives the decay estimates of the asymptotic profiles.

\begin{lem} \label{lem:P_asymp_sub_gene}
	Let $k \in \Z_{>0}$, $1+ \frac{1}{n} <p<1+ \frac{k+1}{kn}$, and $u_{0} \in L^{1} \left( \R^{n} \right)$.
	Then, for any $q \in \left[ 1, + \infty \right]$, there exists $C>0$ such that the estimates
	\begin{align*}
		t^{\frac{n}{2} \left( 1- \frac{1}{q} \right)} \norm{\AA_{0, k} \left( t \right)}_{q} &\leq C, \\
		t^{\frac{n}{2} \left( 1- \frac{1}{q} \right)} \norm{\AA_{0, k} \left( t \right) - \AA_{0, k-1} \left( t \right)}_{q} &\leq Ct^{-k \sigma}
	\end{align*}
	hold for all $t>1$.
\end{lem}

\begin{rem}
	The second estimate in Lemma \ref{lem:P_asymp_sub_gene} with $k=1$ holds for any $t>0$.
\end{rem}

\begin{proof}[Proof of Lemma \ref{lem:P_asymp_sub_gene}]
	We show the assertion by induction on $k \in \Z_{>0}$.
	First, we consider the case where $k=1$.
	Let $1+1/n<p<1+2/n$, $q \in \left[ 1, + \infty \right]$, and $t>0$.
	Since
	\begin{align*}
		\AA_{0, 1} \left( t \right) - \AA_{0, 0} \left( t \right) = \int_{0}^{t} a \cdot \nabla e^{\left( t-s \right) \Delta} f \left( \AA_{0} \left( s \right) \right) ds,
	\end{align*}
	Lemma \ref{lem:heat_LpLq} implies
	\begin{align*}
		\norm{\AA_{0, 1} \left( t \right) - \AA_{0, 0} \left( t \right)}_{q}
		&\leq C \left( \int_{0}^{t/2} + \int_{t/2}^{t} \right) \norm{\nabla e^{\left( t-s \right) \Delta} f \left( \AA_{0} \left( s \right) \right)}_{q} ds \\
		&\leq C \int_{0}^{t/2} \left( t-s \right)^{- \frac{n}{2} \left( 1- \frac{1}{q} \right) - \frac{1}{2}} \norm{\AA_{0} \left( s \right)}_{p}^{p} ds+C \int_{t/2}^{t} \left( t-s \right)^{- \frac{1}{2}} \norm{\AA_{0} \left( s \right)}_{pq}^{p} ds \\
		&\leq C \int_{0}^{t/2} \left( t-s \right)^{- \frac{n}{2} \left( 1- \frac{1}{q} \right) - \frac{1}{2}} s^{- \frac{n}{2} \left( p-1 \right)} ds+C \int_{t/2}^{t} \left( t-s \right)^{- \frac{1}{2}} s^{- \frac{n}{2} \left( 1- \frac{1}{q} \right) - \frac{n}{2} \left( p-1 \right)} ds \\
		&=C \int_{0}^{t/2} \left( t-s \right)^{- \frac{n}{2} \left( 1- \frac{1}{q} \right) - \frac{1}{2}} s^{- \frac{1}{2} - \sigma} ds+C \int_{t/2}^{t} \left( t-s \right)^{- \frac{1}{2}} s^{- \frac{n}{2} \left( 1- \frac{1}{q} \right) - \frac{1}{2} - \sigma} ds \\
		&\leq Ct^{- \frac{n}{2} \left( 1- \frac{1}{q} \right) - \frac{1}{2}} \int_{0}^{t/2} s^{- \frac{1}{2} - \sigma} ds+Ct^{- \frac{n}{2} \left( 1- \frac{1}{q} \right) - \frac{1}{2} - \sigma} \int_{t/2}^{t} \left( t-s \right)^{- \frac{1}{2}} ds \\
		&\leq Ct^{- \frac{n}{2} \left( 1- \frac{1}{q} \right) - \sigma}.
	\end{align*}
	Here, we have used the fact that the estimates
	\begin{align*}
		\norm{\AA_{0} \left( s \right)}_{r} = \abs{\MM_{0} \left( u_{0} \right)} \norm{\delta_{s} G_{1}}_{r} \leq \norm{u_{0}}_{1} \norm{G_{1}}_{r} s^{- \frac{n}{2} \left( 1- \frac{1}{r} \right)}
	\end{align*}
	hold for any $r \in \left[ 1, + \infty \right]$ and $s>0$.
	Furthermore, we obtain
	\begin{align*}
		\norm{\AA_{0, 1} \left( t \right)}_{q} &\leq \norm{\AA_{0, 0} \left( t \right)}_{q} + \norm{\AA_{0, 1} \left( t \right) - \AA_{0, 0} \left( t \right)}_{q} \\
		&\leq Ct^{- \frac{n}{2} \left( 1- \frac{1}{q} \right)} +Ct^{- \frac{n}{2} \left( 1- \frac{1}{q} \right) - \sigma} \\
		&\leq Ct^{- \frac{n}{2} \left( 1- \frac{1}{q} \right)}
	\end{align*}
	for $t>1$.

	Next, we assume that Lemma \ref{lem:P_asymp_sub_gene} is true for some $k \in \Z_{>0}$.
	Let $1+ \frac{1}{n} <p<1+ \frac{k+2}{\left( k+1 \right) n}$, $q \in \left[ 1, + \infty \right]$, and $t>1$.
	We note that
	\begin{align*}
		\AA_{0, k+1} \left( t \right) - \AA_{0, k} \left( t \right) = \int_{1}^{t} a \cdot \nabla e^{\left( t-s \right) \Delta} \left( f \left( \AA_{0, k} \left( s \right) \right) -f \left( \AA_{0, k-1} \left( s \right) \right) \right) ds.
	\end{align*}
	If $t>2$, then it follows from Lemma \ref{lem:heat_LpLq} and the inductive hypothesis that
	\begin{align*}
		&\norm{\AA_{0, k+1} \left( t \right) - \AA_{0, k} \left( t \right)}_{q} \\
		&\hspace{1cm} \leq C \left( \int_{1}^{t/2} + \int_{t/2}^{t} \right) \norm{\nabla e^{\left( t-s \right) \Delta} \left( f \left( \AA_{0, k} \left( s \right) \right) -f \left( \AA_{0, k-1} \left( s \right) \right) \right)}_{q} ds \\
		&\hspace{1cm} \leq C \int_{1}^{t/2} \left( t-s \right)^{- \frac{n}{2} \left( 1- \frac{1}{q} \right) - \frac{1}{2}} \norm{f \left( \AA_{0, k} \left( s \right) \right) -f \left( \AA_{0, k-1} \left( s \right) \right)}_{1} ds \\
		&\hspace{2cm} +C \int_{t/2}^{t} \left( t-s \right)^{- \frac{1}{2}} \norm{f \left( \AA_{0, k} \left( s \right) \right) -f \left( \AA_{0, k-1} \left( s \right) \right)}_{q} ds \\
		&\hspace{1cm} \leq C \int_{1}^{t/2} \left( t-s \right)^{- \frac{n}{2} \left( 1- \frac{1}{q} \right) - \frac{1}{2}} \left( \norm{\AA_{0, k} \left( s \right)}_{\infty}^{p-1} + \norm{\AA_{0, k-1} \left( s \right)}_{\infty}^{p-1} \right) \norm{\AA_{0, k} \left( s \right) - \AA_{0, k-1} \left( s \right)}_{1} ds \\
		&\hspace{2cm} +C \int_{t/2}^{t} \left( t-s \right)^{- \frac{1}{2}} \left( \norm{\AA_{0, k} \left( s \right)}_{\infty}^{p-1} + \norm{\AA_{0, k-1} \left( s \right)}_{\infty}^{p-1} \right) \norm{\AA_{0, k} \left( s \right) - \AA_{0, k-1} \left( s \right)}_{q} ds \\
		&\hspace{1cm} \leq C \int_{1}^{t/2} \left( t-s \right)^{- \frac{n}{2} \left( 1- \frac{1}{q} \right) - \frac{1}{2}} s^{- \frac{n}{2} \left( p-1 \right) -k \sigma} ds+C \int_{t/2}^{t} \left( t-s \right)^{- \frac{1}{2}} s^{- \frac{n}{2} \left( p-1 \right) - \frac{n}{2} \left( 1- \frac{1}{q} \right) -k \sigma} ds \\
		&\hspace{1cm} =C \int_{1}^{t/2} \left( t-s \right)^{- \frac{n}{2} \left( 1- \frac{1}{q} \right) - \frac{1}{2}} s^{- \frac{1}{2} - \left( k+1 \right) \sigma} ds+C \int_{t/2}^{t} \left( t-s \right)^{- \frac{1}{2}} s^{- \frac{n}{2} \left( 1- \frac{1}{q} \right) - \frac{1}{2} - \left( k+1 \right) \sigma} ds \\
		&\hspace{1cm} \leq Ct^{- \frac{n}{2} \left( 1- \frac{1}{q} \right) - \frac{1}{2}} \int_{1}^{t/2} s^{- \frac{1}{2} - \left( k+1 \right) \sigma} ds+Ct^{- \frac{n}{2} \left( 1- \frac{1}{q} \right) - \frac{1}{2} - \left( k+1 \right) \sigma} \int_{t/2}^{t} \left( t-s \right)^{- \frac{1}{2}} ds \\
		&\hspace{1cm} \leq Ct^{- \frac{n}{2} \left( 1- \frac{1}{q} \right) - \left( k+1 \right) \sigma}.
	\end{align*}
	Similarly, if $1<t \leq 2$, then we have
	\begin{align*}
		\norm{\AA_{0, k+1} \left( t \right) - \AA_{0, k} \left( t \right)}_{q}
		&\leq C \int_{1}^{t} \norm{\nabla e^{\left( t-s \right) \Delta} \left( f \left( \AA_{0, k} \left( s \right) \right) -f \left( \AA_{0, k-1} \left( s \right) \right) \right)}_{q} ds \\
		&\leq C \int_{1}^{t} \left( t-s \right)^{- \frac{1}{2}} \left( \norm{\AA_{0, k} \left( s \right)}_{pq}^{p} + \norm{\AA_{0, k-1} \left( s \right)}_{pq}^{p} \right) ds \\
		&\leq C \int_{1}^{t} \left( t-s \right)^{- \frac{1}{2}} s^{- \frac{n}{2} \left( 1- \frac{1}{q} \right) - \frac{n}{2} \left( p-1 \right)} ds \\
		&\leq C \int_{1}^{t} \left( t-s \right)^{- \frac{1}{2}} ds \\
		&\leq C \left( t-1 \right)^{\frac{1}{2}} \\
		&\leq C \\
		&\leq Ct^{- \frac{n}{2} \left( 1- \frac{1}{q} \right) - \left( k+1 \right) \sigma}.
	\end{align*}
	Moreover, we obtain
	\begin{align*}
		\norm{\AA_{0, k+1} \left( t \right)}_{q} &\leq \norm{\AA_{0, k} \left( t \right)}_{q} + \norm{\AA_{0, k+1} \left( t \right) - \AA_{0, k} \left( t \right)}_{q} \\
		&\leq Ct^{- \frac{n}{2} \left( 1- \frac{1}{q} \right)} +Ct^{- \frac{n}{2} \left( 1- \frac{1}{q} \right) - \left( k+1 \right) \sigma} \\
		&\leq Ct^{- \frac{n}{2} \left( 1- \frac{1}{q} \right)}.
	\end{align*}
	This completes the inductive argument.
\end{proof}

Now, we are ready to show Theorem \ref{th:P_asymp_sub_gene}.

\begin{proof}[Proof of Theorem \ref{th:P_asymp_sub_gene}]
	We show the assertion by induction on $k \in \Z_{>0}$.
	First, we assume that Theorem \ref{th:P_asymp_sub_gene} is true for some $k \in \Z_{>0}$.
	Let $1+ \frac{1}{n} <p<1+ \frac{k+2}{\left( k+1 \right) n}$, $q \in \left[ 1, + \infty \right]$, and $t>2^{k+1}$.
	By using \eqref{I}, we decompose the remainder $u \left( t \right) - \AA_{0, k+1} \left( t \right)$ into five parts:
	\begin{align*}
		u \left( t \right) - \AA_{0, k+1} \left( t \right)
		&= \left( e^{t \Delta} u_{0} - \Lambda_{0, 0} \left( t; u_{0} \right) \right) + \int_{0}^{1} a \cdot \nabla e^{\left( t-s \right) \Delta} \left( f \left( u \left( s \right) \right) -f \left( \AA_{0} \left( s \right) \right) \right) ds \\
		&\hspace{1cm} + \int_{1}^{2^{k}} a \cdot \nabla e^{\left( t-s \right) \Delta} \left( f \left( u \left( s \right) \right) -f \left( \AA_{0, k} \left( s \right) \right) \right) ds \\
		&\hspace{1cm} + \int_{2^{k}}^{t/2} a \cdot \nabla e^{\left( t-s \right) \Delta} \left( f \left( u \left( s \right) \right) -f \left( \AA_{0, k} \left( s \right) \right) \right) ds \\
		&\hspace{1cm} + \int_{t/2}^{t} a \cdot \nabla e^{\left( t-s \right) \Delta} \left( f \left( u \left( s \right) \right) -f \left( \AA_{0, k} \left( s \right) \right) \right) ds \\
		&\eqqcolon \sum_{\ell =1}^{5} R_{0, k+1, \ell} \left( t \right).
	\end{align*}
	From \eqref{eq:P_decay}, Proposition \ref{pro:heat_asymp}, Lemmas \ref{lem:heat_LpLq} and \ref{lem:P_asymp_sub_gene}, and the fact that
	\begin{align*}
		p<1+ \frac{k+2}{\left( k+1 \right) n} <1+ \frac{2}{n},
	\end{align*}
	we obtain
	\begin{align*}
		\norm{R_{0, k+1, 1} \left( t \right)}_{q}
		&\leq Ct^{- \frac{n}{2} \left( 1- \frac{1}{q} \right) - \frac{1}{2}} \sum_{\abs{\alpha} =1} \norm{x^{\alpha} u_{0}}_{1}, \\
		\norm{R_{0, k+1, 2} \left( t \right)}_{q}
		&\leq C \int_{0}^{1} \norm{\nabla e^{\left( t-s \right) \Delta} \left( f \left( u \left( s \right) \right) -f \left( \AA_{0} \left( s \right) \right) \right)}_{q} ds \\
		&\leq C \int_{0}^{1} \left( t-s \right)^{- \frac{n}{2} \left( 1- \frac{1}{q} \right) - \frac{1}{2}} \left( \norm{u \left( s \right)}_{p}^{p} + \norm{\AA_{0} \left( s \right)}_{p}^{p} \right) ds \\
		&\leq C \left( t-1 \right)^{- \frac{n}{2} \left( 1- \frac{1}{q} \right) - \frac{1}{2}} \int_{0}^{1} s^{- \frac{n}{2} \left( p-1 \right)} ds \\
		&\leq Ct^{- \frac{n}{2} \left( 1- \frac{1}{q} \right) - \frac{1}{2}}, \\
		\norm{R_{0, k+1, 3} \left( t \right)}_{q}
		&\leq C \int_{1}^{2^{k}} \norm{\nabla e^{\left( t-s \right) \Delta} \left( f \left( u \left( s \right) \right) -f \left( \AA_{0, k} \left( s \right) \right) \right)}_{q} ds \\
		&\leq C \int_{1}^{2^{k}} \left( t-s \right)^{- \frac{n}{2} \left( 1- \frac{1}{q} \right) - \frac{1}{2}} \left( \norm{u \left( s \right)}_{p}^{p} + \norm{\AA_{0, k} \left( s \right)}_{p}^{p} \right) ds \\
		&\leq C \left( t-2^{k} \right)^{- \frac{n}{2} \left( 1- \frac{1}{q} \right) - \frac{1}{2}} \int_{1}^{2^{k}} s^{- \frac{n}{2} \left( p-1 \right)} ds \\
		&\leq Ct^{- \frac{n}{2} \left( 1- \frac{1}{q} \right) - \frac{1}{2}}.
	\end{align*}
	By \eqref{eq:P_decay}, Lemmas \ref{lem:heat_LpLq} and \ref{lem:P_asymp_sub_gene}, and the inductive hypothesis, we have
	\begin{align*}
		\norm{R_{0, k+1, 4} \left( t \right)}_{q}
		&\leq C \int_{2^{k}}^{t/2} \norm{\nabla e^{\left( t-s \right) \Delta} \left( f \left( u \left( s \right) \right) -f \left( \AA_{0, k} \left( s \right) \right) \right)}_{q} ds \\
		&\leq C \int_{2^{k}}^{t/2} \left( t-s \right)^{- \frac{n}{2} \left( 1- \frac{1}{q} \right) - \frac{1}{2}} \norm{f \left( u \left( s \right) \right) -f \left( \AA_{0, k} \left( s \right) \right)}_{1} ds \\
		&\leq Ct^{- \frac{n}{2} \left( 1- \frac{1}{q} \right) - \frac{1}{2}} \int_{2^{k}}^{t/2} \left( \norm{u \left( s \right)}_{\infty}^{p-1} + \norm{\AA_{0, k} \left( s \right)}_{\infty}^{p-1} \right) \norm{u \left( s \right) - \AA_{0, k} \left( s \right)}_{1} ds \\
		&\leq Ct^{- \frac{n}{2} \left( 1- \frac{1}{q} \right) - \frac{1}{2}} \int_{2^{k}}^{t/2} s^{- \frac{n}{2} \left( p-1 \right) - \left( k+1 \right) \sigma} ds \\
		&=Ct^{- \frac{n}{2} \left( 1- \frac{1}{q} \right) - \frac{1}{2}} \int_{2^{k}}^{t/2} s^{- \frac{1}{2} - \left( k+2 \right) \sigma} ds \\
		&\leq Ct^{- \frac{n}{2} \left( 1- \frac{1}{q} \right) - \frac{1}{2}} \times \begin{dcases}
			t^{\frac{1}{2} - \left( k+2 \right) \sigma}, &\qquad \left( k+2 \right) \sigma < \frac{1}{2}, \\
			\log \left( 2+t \right), &\qquad \left( k+2 \right) \sigma = \frac{1}{2}, \\
			1, &\qquad \left( k+2 \right) \sigma > \frac{1}{2},
		\end{dcases} \\
		\norm{R_{0, k+1, 5} \left( t \right)}_{q}
		&\leq C \int_{t/2}^{t} \norm{\nabla e^{\left( t-s \right) \Delta} \left( f \left( u \left( s \right) \right) -f \left( \AA_{0, k} \left( s \right) \right) \right)}_{q} ds \\
		&\leq C \int_{t/2}^{t} \left( t-s \right)^{- \frac{1}{2}} \norm{f \left( u \left( s \right) \right) -f \left( \AA_{0, k} \left( s \right) \right)}_{q} ds \\
		&\leq C \int_{t/2}^{t} \left( t-s \right)^{- \frac{1}{2}} \left( \norm{u \left( s \right)}_{\infty}^{p-1} + \norm{\AA_{0, k} \left( s \right)}_{\infty}^{p-1} \right) \norm{u \left( s \right) - \AA_{0, k} \left( s \right)}_{q} ds \\
		&\leq C \int_{t/2}^{t} \left( t-s \right)^{- \frac{1}{2}} s^{- \frac{n}{2} \left( p-1 \right) - \frac{n}{2} \left( 1- \frac{1}{q} \right) - \left( k+1 \right) \sigma} ds \\
		&\leq Ct^{- \frac{n}{2} \left( 1- \frac{1}{q} \right) - \frac{n}{2} \left( p-1 \right) - \left( k+1 \right) \sigma} \int_{t/2}^{t} \left( t-s \right)^{- \frac{1}{2}} ds \\
		&\leq Ct^{- \frac{n}{2} \left( 1- \frac{1}{q} \right) - \frac{n}{2} \left( p-1 \right) - \left( k+1 \right) \sigma + \frac{1}{2}} \\
		&=Ct^{- \frac{n}{2} \left( 1- \frac{1}{q} \right) - \left( k+2 \right) \sigma}.
	\end{align*}
	Combining these estimates yields the desired result.
	The case where $k=1$ can be shown by the same argument with the aid of Proposition \ref{pro:P_asymp-0}.
\end{proof}

%%%%%%%%%%%%%%%%%%%%%%%%%
\subsection{Proof of Theorem \ref{th:P_asymp_critical_gene}} \label{sec:P_asymp_critical_gene}
%%%%%%%%%%%%%%%%%%%%%%%%%

To show Theorem \ref{th:P_asymp_critical_gene}, we need the weighted $L^{1}$-estimate of the difference between the asymptotic profiles given in Theorem \ref{th:P_asymp_sub_gene}.

\begin{lem} \label{lem:P_asymp_critical_gene}
	Let $k \in \Z_{>0}$, $1+ \frac{1}{n} <p<1+ \frac{k+1}{kn}$, and $u_{0} \in L^{1} \left( \R^{n} \right)$.
	Then, there exists $C>0$ such that the estimate
	\begin{align*}
		\norm{\abs{x} \left( \AA_{0, k} \left( t \right) - \AA_{0, k-1} \left( t \right) \right)}_{1} &\leq Ct^{\frac{1}{2} -k \sigma}
	\end{align*}
	holds for any $t>1$.
\end{lem}

\begin{proof}
	We prove the assertion by induction on $k \in \Z_{>0}$.
	First, we consider the case where $k=1$.
	Let $1+1/n<p<1+2/n$ and let $t>0$.
	By Lemma \ref{lem:heat_weight}, we have
	\begin{align*}
		\norm{\abs{x} \left( \AA_{0, 1} \left( t \right) - \AA_{0, 0} \left( t \right) \right)}_{1}
		&\leq C \sum_{j=1}^{n} \int_{0}^{t} \norm{\abs{x} \partial_{j} e^{\left( t-s \right) \Delta} f \left( \AA_{0} \left( s \right) \right)}_{1} ds \\
		&\leq C \int_{0}^{t} \left( t-s \right)^{- \frac{1}{2}} \norm{\AA_{0} \left( s \right)}_{\infty}^{p-1} \norm{\abs{x} \AA_{0} \left( s \right)}_{1} ds+C \int_{0}^{t} \norm{\AA_{0} \left( s \right)}_{p}^{p} ds \\
		&\leq C \int_{0}^{t} \left( t-s \right)^{- \frac{1}{2}} s^{\frac{1}{2} - \frac{n}{2} \left( p-1 \right)} ds+C \int_{0}^{t} s^{- \frac{n}{2} \left( p-1 \right)} ds \\
		&=Ct^{\frac{1}{2} - \sigma} \int_{0}^{1} \left( 1- \theta \right)^{- \frac{1}{2}} \theta^{- \sigma} d \theta +Ct^{\frac{1}{2} - \sigma} \int_{0}^{1} \theta^{- \frac{1}{2} - \sigma} d \theta \\
		&\leq Ct^{\frac{1}{2} - \sigma}.
	\end{align*}
	Here, we have used the fact that the estimates
	\begin{align*}
		\norm{\abs{x} \AA_{0} \left( s \right)}_{1} = \abs{\MM_{0} \left( u_{0} \right)} \norm{\abs{x} \delta_{s} G_{1}}_{1} \leq \norm{u_{0}}_{1} \norm{\abs{x} G_{1}}_{1} s^{\frac{1}{2}}
	\end{align*}
	hold for any $s>0$.

	Next, we assume that Lemma \ref{lem:P_asymp_critical_gene} is true for some $k \in \Z_{>0}$.
	Let $1+ \frac{1}{n} <p<1+ \frac{k+2}{\left( k+1 \right) n}$ and let $t>1$.
	Then, it follows from the inductive hypothesis, Lemmas \ref{lem:heat_weight} and \ref{lem:P_asymp_sub_gene} that
	\begin{align*}
		&\norm{\abs{x} \left( \AA_{0, k+1} \left( t \right) - \AA_{0, k} \left( t \right) \right)}_{1} \\
		&\hspace{1cm} \leq C \sum_{j=1}^{n} \int_{1}^{t} \norm{\abs{x} \partial_{j} e^{\left( t-s \right) \Delta} \left( f \left( \AA_{0, k} \left( s \right) \right) -f \left( \AA_{0, k-1} \left( s \right) \right) \right)}_{1} ds \\
		&\hspace{1cm} \leq C \int_{1}^{t} \left( t-s \right)^{- \frac{1}{2}} \norm{\abs{x} \left( f \left( \AA_{0, k} \left( s \right) \right) -f \left( \AA_{0, k-1} \left( s \right) \right) \right)}_{1} ds \\
		&\hspace{2cm} +C \int_{1}^{t} \norm{f \left( \AA_{0, k} \left( s \right) \right) -f \left( \AA_{0, k-1} \left( s \right) \right)}_{1} ds \\
		&\hspace{1cm} \leq C \int_{1}^{t} \left( t-s \right)^{- \frac{1}{2}} \left( \norm{\AA_{0, k} \left( s \right)}_{\infty}^{p-1} + \norm{\AA_{0, k-1} \left( s \right)}_{\infty}^{p-1} \right) \norm{\abs{x} \left( \AA_{0, k} \left( s \right) - \AA_{0, k-1} \left( s \right) \right)}_{1} ds \\
		&\hspace{2cm} +C \int_{1}^{t} \left( \norm{\AA_{0, k} \left( s \right)}_{\infty}^{p-1} + \norm{\AA_{0, k-1} \left( s \right)}_{\infty}^{p-1} \right) \norm{\AA_{0, k} \left( s \right) - \AA_{0, k-1} \left( s \right)}_{1} ds \\
		&\hspace{1cm} \leq C \int_{1}^{t} \left( t-s \right)^{- \frac{1}{2}} s^{- \frac{n}{2} \left( p-1 \right) + \frac{1}{2} -k \sigma} ds+C \int_{1}^{t} s^{- \frac{n}{2} \left( p-1 \right) -k \sigma} ds \\
		&\hspace{1cm} =Ct^{\frac{1}{2} - \left( k+1 \right) \sigma} \int_{1/t}^{1} \left( 1- \theta \right)^{- \frac{1}{2}} \theta^{- \left( k+1 \right) \sigma} d \theta +Ct^{\frac{1}{2} - \left( k+1 \right) \sigma} \int_{1/t}^{1} \theta^{- \frac{1}{2} - \left( k+1 \right) \sigma} d \theta \\
		&\hspace{1cm} \leq Ct^{\frac{1}{2} - \left( k+1 \right) \sigma} \int_{0}^{1} \left( 1- \theta \right)^{- \frac{1}{2}} \theta^{- \left( k+1 \right) \sigma} d \theta +Ct^{\frac{1}{2} - \left( k+1 \right) \sigma} \int_{0}^{1} \theta^{- \frac{1}{2} - \left( k+1 \right) \sigma} d \theta \\
		&\hspace{1cm} \leq Ct^{\frac{1}{2} - \left( k+1 \right) \sigma}.
	\end{align*}
	This completes the inductive argument.
\end{proof}

We next give the proof of Theorem \ref{th:P_asymp_critical_gene}.

\begin{proof}[Proof of Theorem \ref{th:P_asymp_critical_gene}]
	Let $q \in \left[ 1, + \infty \right]$ and let $t>2^{k+1}$.
	Using \eqref{I}, we decompose the remainder $u \left( t \right) - \widetilde{\AA}_{0, k+1} \left( t \right)$ into nine parts:
	\begin{align*}
		&u \left( t \right) - \widetilde{\AA}_{0, k+1} \left( t \right) \\
		&\hspace{0.5cm} = \left( e^{t \Delta} u_{0} - \Lambda_{0, 0} \left( t; u_{0} \right) \right) + \int_{0}^{1} a \cdot \nabla e^{\left( t-s \right) \Delta} \left( f \left( u \left( s \right) \right) -f \left( \AA_{0} \left( s \right) \right) \right) ds \\
		&\hspace{1cm} + \int_{1}^{2^{k}} a \cdot \nabla e^{\left( t-s \right) \Delta} \left( f \left( u \left( s \right) \right) -f \left( \AA_{0, k-1} \left( s \right) \right) \right) ds \\
		&\hspace{1cm} + \int_{t/2}^{t} a \cdot \nabla e^{\left( t-s \right) \Delta} \left( f \left( u \left( s \right) \right) -f \left( \AA_{0, k-1} \left( s \right) \right) \right) ds \\
		&\hspace{1cm} + \sum_{j=1}^{n} a_{j} \int_{2^{k}}^{t/2} \int_{0}^{1} \left( -s \Delta \right) \partial_{j} e^{\left( t-s \theta \right) \Delta} \left( f \left( u \left( s \right) \right) -f \left( \AA_{0, k-1} \left( s \right) \right) \right) d \theta ds \\
		&\hspace{1cm} + \int_{2^{k}}^{t/2} a \cdot \nabla e^{t \Delta} \left( f \left( u \left( s \right) \right) -f \left( \AA_{0, k} \left( s \right) \right) \right) ds \\
		&\hspace{1cm} + \sum_{j=1}^{n} a_{j} \int_{2^{k}}^{t/2} \left( \partial_{j} e^{t \Delta} \left( f \left( \AA_{0, k} \left( s \right) \right) -f \left( \AA_{0, k-1} \left( s \right) \right) \right) - \Lambda_{e_{j}, 0} \left( t; f \left( \AA_{0, k} \left( s \right) \right) -f \left( \AA_{0, k-1} \left( s \right) \right) \right) \right) ds \\
		&\hspace{1cm} - \sum_{j=1}^{n} a_{j} \int_{1}^{2^{k}} \Lambda_{e_{j}, 0} \left( t; f \left( \AA_{0, k} \left( s \right) \right) -f \left( \AA_{0, k-1} \left( s \right) \right) \right) ds \\
		&\hspace{1cm} - \sum_{j=1}^{n} a_{j} \int_{t/2}^{t} \Lambda_{e_{j}, 0} \left( t; f \left( \AA_{0, k} \left( s \right) \right) -f \left( \AA_{0, k-1} \left( s \right) \right) \right) ds \\
		&\hspace{0.5cm} \eqqcolon \sum_{\ell =1}^{9} \widetilde{R}_{0, k+1, \ell} \left( t \right).
	\end{align*}
	It follows from \eqref{eq:P_decay}, Lemmas \ref{lem:heat_LpLq} and \ref{lem:P_asymp_sub_gene} that
	\begin{align*}
		\norm{\widetilde{R}_{0, k+1, 2} \left( t \right)}_{q}
		&\leq C \int_{0}^{1} \norm{\nabla e^{\left( t-s \right) \Delta} \left( f \left( u \left( s \right) \right) -f \left( \AA_{0} \left( s \right) \right) \right)}_{q} ds \\
		&\leq C \int_{0}^{1} \left( t-s \right)^{- \frac{n}{2} \left( 1- \frac{1}{q} \right) - \frac{1}{2}} \left( \norm{u \left( s \right)}_{p}^{p} + \norm{\AA_{0} \left( s \right)}_{p}^{p} \right) ds \\
		&\leq C \left( t-1 \right)^{- \frac{n}{2} \left( 1- \frac{1}{q} \right) - \frac{1}{2}} \int_{0}^{1} s^{- \frac{n}{2} \left( p-1 \right)} ds \\
		&\leq Ct^{- \frac{n}{2} \left( 1- \frac{1}{q} \right) - \frac{1}{2}}, \\
		\norm{\widetilde{R}_{0, k+1, 3} \left( t \right)}_{q}
		&\leq C \int_{1}^{2^{k}} \norm{\nabla e^{\left( t-s \right) \Delta} \left( f \left( u \left( s \right) \right) -f \left( \AA_{0, k-1} \left( s \right) \right) \right)}_{q} ds \\
		&\leq C \int_{1}^{2^{k}} \left( t-s \right)^{- \frac{n}{2} \left( 1- \frac{1}{q} \right) - \frac{1}{2}} \left( \norm{u \left( s \right)}_{p}^{p} + \norm{\AA_{0, k-1} \left( s \right)}_{p}^{p} \right) ds \\
		&\leq C \left( t-2^{k} \right)^{- \frac{n}{2} \left( 1- \frac{1}{q} \right) - \frac{1}{2}} \int_{1}^{2^{k}} s^{- \frac{n}{2} \left( p-1 \right)} ds \\
		&\leq Ct^{- \frac{n}{2} \left( 1- \frac{1}{q} \right) - \frac{1}{2}}, \\
		\norm{\widetilde{R}_{0, k+1, 8} \left( t \right)}_{q}
		&\leq C \sum_{j=1}^{n} \int_{1}^{2^{k}} \norm{\Lambda_{e_{j}, 0} \left( t; f \left( \AA_{0, k} \left( s \right) \right) -f \left( \AA_{0, k-1} \left( s \right) \right) \right)}_{q} ds \\
		&\leq Ct^{- \frac{1}{2}} \sum_{j=1}^{n} \int_{1}^{2^{k}} \norm{f \left( \AA_{0, k} \left( s \right) \right) - f \left( \AA_{0, k-1} \left( s \right) \right)}_{1} \norm{\delta_{t} \left( x_{j} G_{1} \right)}_{q} ds \\
		&\leq Ct^{- \frac{n}{2} \left( 1- \frac{1}{q} \right) - \frac{1}{2}} \int_{1}^{2^{k}} \left( \norm{\AA_{0, k} \left( s \right)}_{p}^{p} + \norm{\AA_{0, k-1} \left( s \right)}_{p}^{p} \right) ds \\
		&\leq Ct^{- \frac{n}{2} \left( 1- \frac{1}{q} \right) - \frac{1}{2}} \int_{1}^{2^{k}} s^{- \frac{n}{2} \left( p-1 \right)} ds \\
		&\leq Ct^{- \frac{n}{2} \left( 1- \frac{1}{q} \right) - \frac{1}{2}}, \\
		\norm{\widetilde{R}_{0, k+1, 9} \left( t \right)}_{q}
		&\leq C \sum_{j=1}^{n} \int_{t/2}^{t} \norm{\Lambda_{e_{j}, 0} \left( t; f \left( \AA_{0, k} \left( s \right) \right) -f \left( \AA_{0, k-1} \left( s \right) \right) \right)}_{q} ds \\
		&\leq Ct^{- \frac{1}{2}} \sum_{j=1}^{n} \int_{t/2}^{t} \norm{f \left( \AA_{0, k} \left( s \right) \right) - f \left( \AA_{0, k-1} \left( s \right) \right)}_{1} \norm{\delta_{t} \left( x_{j} G_{1} \right)}_{q} ds \\
		&\leq Ct^{- \frac{n}{2} \left( 1- \frac{1}{q} \right) - \frac{1}{2}} \int_{t/2}^{t} \left( \norm{\AA_{0, k} \left( s \right)}_{\infty}^{p-1} + \norm{\AA_{0, k-1} \left( s \right)}_{\infty}^{p-1} \right) \norm{\AA_{0, k} \left( s \right) - \AA_{0, k-1} \left( s \right)}_{1} ds \\
		&\leq Ct^{- \frac{n}{2} \left( 1- \frac{1}{q} \right) - \frac{1}{2}} \int_{t/2}^{t} s^{- \frac{n}{2} \left( p-1 \right) -k \sigma} ds \\
		&=Ct^{- \frac{n}{2} \left( 1- \frac{1}{q} \right) - \frac{1}{2}} \int_{t/2}^{t} s^{-1} ds \\
		&\leq Ct^{- \frac{n}{2} \left( 1- \frac{1}{q} \right) - \frac{1}{2}}.
	\end{align*}
	By \eqref{eq:P_decay}, Theorem \ref{th:P_asymp_sub_gene}, and Lemmas \ref{lem:heat_LpLq} and \ref{lem:P_asymp_sub_gene}, we have
	\begin{align*}
		\norm{\widetilde{R}_{0, k+1, 4} \left( t \right)}_{q}
		&\leq C \int_{t/2}^{t} \norm{\nabla e^{\left( t-s \right) \Delta} \left( f \left( u \left( s \right) \right) -f \left( \AA_{0, k-1} \left( s \right) \right) \right)}_{q} ds \\
		&\leq C \int_{t/2}^{t} \left( t-s \right)^{- \frac{1}{2}} \norm{f \left( u \left( s \right) \right) -f \left( \AA_{0, k-1} \left( s \right) \right)}_{q} ds \\
		&\leq C \int_{t/2}^{t} \left( t-s \right)^{- \frac{1}{2}} \left( \norm{u \left( s \right)}_{\infty}^{p-1} + \norm{\AA_{0, k-1} \left( s \right)}_{\infty}^{p-1} \right) \norm{u \left( s \right) - \AA_{0, k-1} \left( s \right)}_{q} ds \\
		&\leq C \int_{t/2}^{t} \left( t-s \right)^{- \frac{1}{2}} s^{- \frac{n}{2} \left( p-1 \right) - \frac{n}{2} \left( 1- \frac{1}{q} \right) -k \sigma} ds \\
		&\leq Ct^{- \frac{n}{2} \left( 1- \frac{1}{q} \right) - \frac{n}{2} \left( p-1 \right) -k \sigma} \int_{t/2}^{t} \left( t-s \right)^{- \frac{1}{2}} ds \\
		&\leq Ct^{- \frac{n}{2} \left( 1- \frac{1}{q} \right) - \frac{n}{2} \left( p-1 \right) -k \sigma + \frac{1}{2}} \\
		&=Ct^{- \frac{n}{2} \left( 1- \frac{1}{q} \right) - \frac{1}{2}}, \\
		\norm{\widetilde{R}_{0, k+1, 5} \left( t \right)}_{q}
		&\leq C \sum_{j=1}^{n} \int_{2^{k}}^{t/2} \int_{0}^{1} s \norm{\partial_{j} \Delta e^{\left( t-s \theta \right) \Delta} \left( f \left( u \left( s \right) \right) -f \left( \AA_{0, k-1} \left( s \right) \right) \right)}_{q} d \theta ds \\
		&\leq C \int_{2^{k}}^{t/2} \int_{0}^{1} s \left( t-s \theta \right)^{- \frac{n}{2} \left( 1- \frac{1}{q} \right) - \frac{3}{2}} \norm{f \left( u \left( s \right) \right) -f \left( \AA_{0, k-1} \left( s \right) \right)}_{1} d \theta ds \\
		&\leq Ct^{- \frac{n}{2} \left( 1- \frac{1}{q} \right) - \frac{3}{2}} \int_{2^{k}}^{t/2} s \left( \norm{u \left( s \right)}_{\infty}^{p-1} + \norm{\AA_{0, k-1} \left( s \right)}_{\infty}^{p-1} \right) \norm{u \left( s \right) - \AA_{0, k-1} \left( s \right)}_{1} ds \\
		&\leq Ct^{- \frac{n}{2} \left( 1- \frac{1}{q} \right) - \frac{3}{2}} \int_{2^{k}}^{t/2} s^{1- \frac{n}{2} \left( p-1 \right) -k \sigma} ds \\
		&=Ct^{- \frac{n}{2} \left( 1- \frac{1}{q} \right) - \frac{3}{2}} \int_{2^{k}}^{t/2} ds \\
		&\leq Ct^{- \frac{n}{2} \left( 1- \frac{1}{q} \right) - \frac{1}{2}}, \\
		\norm{\widetilde{R}_{0, k+1, 6} \left( t \right)}_{q}
		&\leq C \int_{2^{k}}^{t/2} \norm{\nabla e^{t \Delta} \left( f \left( u \left( s \right) \right) -f \left( \AA_{0, k} \left( s \right) \right) \right)}_{q} ds \\
		&\leq Ct^{- \frac{n}{2} \left( 1- \frac{1}{q} \right) - \frac{1}{2}} \int_{2^{k}}^{t/2} \norm{f \left( u \left( s \right) \right) -f \left( \AA_{0, k} \left( s \right) \right)}_{1} ds \\
		&\leq Ct^{- \frac{n}{2} \left( 1- \frac{1}{q} \right) - \frac{1}{2}} \int_{2^{k}}^{t/2} \left( \norm{u \left( s \right)}_{\infty}^{p-1} + \norm{\AA_{0, k} \left( s \right)}_{\infty}^{p-1} \right) \norm{u \left( s \right) - \AA_{0, k} \left( s \right)}_{1} ds \\
		&\leq Ct^{- \frac{n}{2} \left( 1- \frac{1}{q} \right) - \frac{1}{2}} \int_{2^{k}}^{t/2} s^{- \frac{n}{2} \left( p-1 \right) - \frac{1}{2}} \log \left( 2+s \right) ds \\
		&=Ct^{- \frac{n}{2} \left( 1- \frac{1}{q} \right) - \frac{1}{2}} \int_{2^{k}}^{t/2} s^{-1- \sigma} \log \left( 2+s \right) ds \\
		&\leq Ct^{- \frac{n}{2} \left( 1- \frac{1}{q} \right) - \frac{1}{2}} \int_{2^{k}}^{t/2} s^{-1- \frac{\sigma}{2}} ds \\
		&\leq Ct^{- \frac{n}{2} \left( 1- \frac{1}{q} \right) - \frac{1}{2}}.
	\end{align*}
	Here, we have used the fact that
	\begin{align*}
		\sup_{s>1} s^{- \frac{\sigma}{2}} \log \left( 2+s \right) <+ \infty.
	\end{align*}
	In the case where $k=1$, we apply Proposition \ref{pro:P_asymp-0} to the estimates of $\widetilde{R}_{0, k+1, 4} \left( t \right)$ and $\widetilde{R}_{0, k+1, 5} \left( t \right)$ instead of Theorem \ref{th:P_asymp_sub_gene}.
	Due to Proposition \ref{pro:heat_asymp}, Lemmas \ref{lem:P_asymp_sub_gene} and \ref{lem:P_asymp_critical_gene}, we obtain
	\begin{align*}
		\norm{\widetilde{R}_{0, k+1, 1} \left( t \right)}_{q}
		&\leq Ct^{- \frac{n}{2} \left( 1- \frac{1}{q} \right) - \frac{1}{2}} \sum_{\abs{\alpha} =1} \norm{x^{\alpha} u_{0}}_{1}, \\
		\norm{\widetilde{R}_{0, k+1, 7} \left( t \right)}_{q}
		&\leq Ct^{- \frac{n}{2} \left( 1- \frac{1}{q} \right) -1} \sum_{\abs{\alpha} =1} \int_{2^{k}}^{t/2} \norm{x^{\alpha} \left( f \left( \AA_{0, k} \left( s \right) \right) -f \left( \AA_{0, k-1} \left( s \right) \right) \right)}_{1} ds \\
		&\leq Ct^{- \frac{n}{2} \left( 1- \frac{1}{q} \right) -1} \int_{2^{k}}^{t/2} \left( \norm{\AA_{0, k} \left( s \right)}_{\infty}^{p-1} + \norm{\AA_{0, k-1} \left( s \right)}_{\infty}^{p-1} \right) \\
		&\hspace{6cm} \times \norm{\abs{x} \left( \AA_{0, k} \left( s \right) - \AA_{0, k-1} \left( s \right) \right)}_{1} ds \\
		&\leq Ct^{- \frac{n}{2} \left( 1- \frac{1}{q} \right) -1} \int_{2^{k}}^{t/2} s^{- \frac{n}{2} \left( p-1 \right) + \frac{1}{2} -k \sigma} ds \\
		&=Ct^{- \frac{n}{2} \left( 1- \frac{1}{q} \right) -1} \int_{2^{k}}^{t/2} s^{- \frac{1}{2}} ds \\
		&\leq Ct^{- \frac{n}{2} \left( 1- \frac{1}{q} \right) - \frac{1}{2}}.
	\end{align*}
	This completes the proof of Theorem \ref{th:P_asymp_critical_gene}.
\end{proof}

\begin{rem} \label{rem:P_asymp_critical_gene}
	Under the same assumption as in Theorem \ref{th:P_asymp_critical_gene}, for any $q \in \left[ 1, + \infty \right]$, there exists $C>0$ such that the estimate
	\begin{align*}
		t^{\frac{n}{2} \left( 1- \frac{1}{q} \right)} \norm{\widetilde{\AA}_{0, k+1} \left( t \right) - \AA_{0, k} \left( t \right)}_{q} &\leq Ct^{- \frac{1}{2}} \log t
	\end{align*}
	holds for all $t>1$.
	In fact, by Lemma \ref{lem:P_asymp_sub_gene}, we have
	\begin{align*}
		&\norm{\widetilde{\AA}_{0, k+1} \left( t \right) - \AA_{0, k} \left( t \right)}_{q} \\
		&\hspace{1cm} \leq C \sum_{j=1}^{n} \int_{1}^{t} \norm{\Lambda_{e_{j}, 0} \left( t; f \left( \AA_{0, k} \left( s \right) \right) -f \left( \AA_{0, k-1} \left( s \right) \right) \right)}_{q} ds \\
		&\hspace{1cm} \leq Ct^{- \frac{1}{2}} \sum_{j=1}^{n} \int_{1}^{t} \norm{f \left( \AA_{0, k} \left( s \right) \right) -f \left( \AA_{0, k-1} \left( s \right) \right)}_{1} \norm{\delta_{t} \left( x_{j} G_{1} \right)}_{q} ds \\
		&\hspace{1cm} \leq Ct^{- \frac{n}{2} \left( 1- \frac{1}{q} \right) - \frac{1}{2}} \int_{1}^{t} \left( \norm{\AA_{0, k} \left( s \right)}_{\infty}^{p-1} + \norm{\AA_{0, k-1} \left( s \right)}_{\infty}^{p-1} \right) \norm{\AA_{0, k} \left( s \right) - \AA_{0, k-1} \left( s \right)}_{1} ds \\
		&\hspace{1cm} \leq Ct^{- \frac{n}{2} \left( 1- \frac{1}{q} \right) - \frac{1}{2}} \int_{1}^{t} s^{- \frac{n}{2} \left( p-1 \right) -k \sigma} ds \\
		&\hspace{1cm} =Ct^{- \frac{n}{2} \left( 1- \frac{1}{q} \right) - \frac{1}{2}} \int_{1}^{t} s^{-1} ds \\
		&\hspace{1cm} =Ct^{- \frac{n}{2} \left( 1- \frac{1}{q} \right) - \frac{1}{2}} \log t.
	\end{align*}
\end{rem}

%%%%%%%%%%%%%%%%%%%%%%%%%
\subsection{Proof of Theorem \ref{th:P_asymp_super_gene}} \label{sec:P_asymp_super_gene}
%%%%%%%%%%%%%%%%%%%%%%%%%

The following lemma gives the asymptotic behavior of the Duhamel term in \eqref{I}.

\begin{lem} \label{lem:P_asymp_super_gene}
	Under the same assumption as in Theorem \ref{th:P_asymp_super_gene}, for any $q \in \left[ 1, + \infty \right]$, there exists $C>0$ such that the estimate
	\begin{align*}
		t^{\frac{n}{2} \left( 1- \frac{1}{q} \right)} \norm{\int_{0}^{t} a \cdot \nabla e^{\left( t-s \right) \Delta} f \left( u \left( s \right) \right) ds- \left( \AA_{0, k} \left( t \right) - \AA_{0} \left( t \right) \right) - a \cdot \nabla e^{t \Delta} \psi_{0, k}}_{q} \leq Ct^{- \left( k+1 \right) \sigma}
	\end{align*}
	holds for all $t>2^{k}$.
\end{lem}

\begin{proof}
	By \eqref{eq:P_decay}, Theorem \ref{th:P_asymp_sub_gene}, and Lemma \ref{lem:P_asymp_sub_gene}, we have
	\begin{align*}
		\norm{\psi_{0, k}}_{1} &\leq \int_{0}^{1} \left( \norm{f \left( u \left( s \right) \right)}_{1} + \norm{f \left( \AA_{0} \left( s \right) \right)}_{1} \right) ds+ \int_{1}^{+ \infty} \norm{f \left( u \left( s \right) \right) -f \left( \AA_{0, k-1} \left( s \right) \right)}_{1} ds \\
		&\leq C \int_{0}^{1} \left( \norm{u \left( s \right)}_{p}^{p} + \norm{\AA_{0} \left( s \right)}_{p}^{p} \right) ds \\
		&\hspace{1cm} +C \int_{1}^{+ \infty} \left( \norm{u \left( s \right)}_{\infty}^{p-1} + \norm{\AA_{0, k-1} \left( s \right)}_{\infty}^{p-1} \right) \norm{u \left( s \right) - \AA_{0, k-1} \left( s \right)}_{1} ds \\
		&\leq C \int_{0}^{1} s^{- \frac{n}{2} \left( p-1 \right)} ds+C \int_{1}^{+ \infty} s^{- \frac{n}{2} \left( p-1 \right) -k \sigma} ds \\
		&=C \int_{0}^{1} s^{- \frac{1}{2} - \sigma} ds+C \int_{1}^{+ \infty} s^{- \frac{1}{2} - \left( k+1 \right) \sigma} ds \\
		&\leq C,
	\end{align*}
	which implies $\psi_{0, k} \in L^{1} \left( \R^{n} \right)$.
	Let $q \in \left[ 1, + \infty \right]$ and let $t>2^{k}$.
	We decompose the remainder into five parts:
	\begin{align*}
		&\int_{0}^{t} a \cdot \nabla e^{\left( t-s \right) \Delta} f \left( u \left( s \right) \right) ds- \left( \AA_{0, k} \left( t \right) - \AA_{0} \left( t \right) \right) - a \cdot \nabla e^{t \Delta} \psi_{0, k} \\
		&\hspace{1cm} = \sum_{j=1}^{n} a_{j} \int_{0}^{1} \int_{0}^{1} \left( -s \Delta \right) \partial_{j} e^{\left( t-s \theta \right) \Delta} \left( f \left( u \left( s \right) \right) -f \left( \AA_{0} \left( s \right) \right) \right) d \theta ds \\
		&\hspace{2cm} + \sum_{j=1}^{n} a_{j} \int_{1}^{2^{k-1}} \int_{0}^{1} \left( -s \Delta \right) \partial_{j} e^{\left( t-s \theta \right) \Delta} \left( f \left( u \left( s \right) \right) -f \left( \AA_{0, k-1} \left( s \right) \right) \right) d \theta ds \\
		&\hspace{2cm} + \sum_{j=1}^{n} a_{j} \int_{2^{k-1}}^{t/2} \int_{0}^{1} \left( -s \Delta \right) \partial_{j} e^{\left( t-s \theta \right) \Delta} \left( f \left( u \left( s \right) \right) -f \left( \AA_{0, k-1} \left( s \right) \right) \right) d \theta ds \\
		&\hspace{2cm} + \int_{t/2}^{t} a \cdot \nabla e^{\left( t-s \right) \Delta} \left( f \left( u \left( s \right) \right) -f \left( \AA_{0, k-1} \left( s \right) \right) \right) ds \\
		&\hspace{2cm} - \int_{t/2}^{+ \infty} a \cdot \nabla e^{t \Delta} \left( f \left( u \left( s \right) \right) -f \left( \AA_{0, k-1} \left( s \right) \right) \right) ds \\
		&\hspace{1cm} \eqqcolon \sum_{\ell =1}^{5} R_{1, k, \ell} \left( t \right).
	\end{align*}
	We note that
	\begin{align*}
		1+ \frac{k+2}{\left( k+1 \right) n} <p<1+ \frac{k+1}{kn} \quad \Longleftrightarrow \quad k \sigma < \frac{1}{2} < \left( k+1 \right) \sigma < \frac{3}{2}.
	\end{align*}
	By virtue of \eqref{eq:P_decay}, Theorem \ref{th:P_asymp_sub_gene}, and Lemmas \ref{lem:heat_LpLq} and \ref{lem:P_asymp_sub_gene}, each term is estimated as
	\begin{align*}
		\norm{R_{1, k, 1} \left( t \right)}_{q}
		&\leq C \sum_{j=1}^{n} \int_{0}^{1} \int_{0}^{1} s \norm{\partial_{j} \Delta e^{\left( t-s \theta \right) \Delta} \left( f \left( u \left( s \right) \right) -f \left( \AA_{0} \left( s \right) \right) \right)}_{q} d \theta ds \\
		&\leq C \int_{0}^{1} \int_{0}^{1} s \left( t-s \theta \right)^{- \frac{n}{2} \left( 1- \frac{1}{q} \right) - \frac{3}{2}} \left( \norm{u \left( s \right)}_{p}^{p} + \norm{\AA_{0} \left( s \right)}_{p}^{p} \right) d \theta ds \\
		&\leq C \left( t-1 \right)^{- \frac{n}{2} \left( 1- \frac{1}{q} \right) - \frac{3}{2}} \int_{0}^{1} s^{1- \frac{n}{2} \left( p-1 \right)} ds \\
		&\leq Ct^{- \frac{n}{2} \left( 1- \frac{1}{q} \right) - \frac{3}{2}}, \\
		\norm{R_{1, k, 2} \left( t \right)}_{q}
		&\leq C \sum_{j=1}^{n} \int_{1}^{2^{k-1}} \int_{0}^{1} s \norm{\partial_{j} \Delta e^{\left( t-s \theta \right) \Delta} \left( f \left( u \left( s \right) \right) -f \left( \AA_{0, k-1} \left( s \right) \right) \right)}_{q} d \theta ds \\
		&\leq C \int_{1}^{2^{k-1}} \int_{0}^{1} s \left( t-s \theta \right)^{- \frac{n}{2} \left( 1- \frac{1}{q} \right) - \frac{3}{2}} \left( \norm{u \left( s \right)}_{p}^{p} + \norm{\AA_{0, k-1} \left( s \right)}_{p}^{p} \right) d \theta ds \\
		&\leq C \left( t-2^{k-1} \right)^{- \frac{n}{2} \left( 1- \frac{1}{q} \right) - \frac{3}{2}} \int_{1}^{2^{k-1}} s^{1- \frac{n}{2} \left( p-1 \right)} ds \\
		&\leq Ct^{- \frac{n}{2} \left( 1- \frac{1}{q} \right) - \frac{3}{2}}, \\
		\norm{R_{1, k, 3} \left( t \right)}_{q}
		&\leq C \sum_{j=1}^{n} \int_{2^{k-1}}^{t/2} \int_{0}^{1} s \norm{\partial_{j} \Delta e^{\left( t-s \theta \right) \Delta} \left( f \left( u \left( s \right) \right) -f \left( \AA_{0, k-1} \left( s \right) \right) \right)}_{q} d \theta ds \\
		&\leq C \int_{2^{k-1}}^{t/2} \int_{0}^{1} s \left( t-s \right)^{- \frac{n}{2} \left( 1- \frac{1}{q} \right) - \frac{3}{2}} \norm{f \left( u \left( s \right) \right) -f \left( \AA_{0, k-1} \left( s \right) \right)}_{1} d \theta ds \\
		&\leq Ct^{- \frac{n}{2} \left( 1- \frac{1}{q} \right) - \frac{3}{2}} \int_{2^{k-1}}^{t/2} s \left( \norm{u \left( s \right)}_{\infty}^{p-1} + \norm{\AA_{0, k-1} \left( s \right)}_{\infty}^{p-1} \right) \norm{u \left( s \right) - \AA_{0, k-1} \left( s \right)}_{1} ds \\
		&\leq Ct^{- \frac{n}{2} \left( 1- \frac{1}{q} \right) - \frac{3}{2}} \int_{2^{k-1}}^{t/2} s^{1- \frac{n}{2} \left( p-1 \right) -k \sigma} ds \\
		&=Ct^{- \frac{n}{2} \left( 1- \frac{1}{q} \right) - \frac{3}{2}} \int_{2^{k-1}}^{t/2} s^{\frac{1}{2} - \left( k+1 \right) \sigma} ds \\
		&\leq Ct^{- \frac{n}{2} \left( 1- \frac{1}{q} \right) - \left( k+1 \right) \sigma}, \\
		\norm{R_{1, k, 4} \left( t \right)}_{q}
		&\leq C \int_{t/2}^{t} \norm{\nabla e^{\left( t-s \right) \Delta} \left( f \left( u \left( s \right) \right) -f \left( \AA_{0, k-1} \left( s \right) \right) \right)}_{q} ds \\
		&\leq C \int_{t/2}^{t} \left( t-s \right)^{- \frac{1}{2}} \norm{f \left( u \left( s \right) \right) -f \left( \AA_{0, k-1} \left( s \right) \right)}_{q} ds \\
		&\leq C \int_{t/2}^{t} \left( t-s \right)^{- \frac{1}{2}} \left( \norm{u \left( s \right)}_{\infty}^{p-1} + \norm{\AA_{0, k-1} \left( s \right)}_{\infty}^{p-1} \right) \norm{u \left( s \right) - \AA_{0, k-1} \left( s \right)}_{q} ds \\
		&\leq C \int_{t/2}^{t} \left( t-s \right)^{- \frac{1}{2}} s^{- \frac{n}{2} \left( p-1 \right) - \frac{n}{2} \left( 1- \frac{1}{q} \right) -k \sigma} ds \\
		&\leq Ct^{- \frac{n}{2} \left( 1- \frac{1}{q} \right) - \frac{n}{2} \left( p-1 \right) -k \sigma} \int_{t/2}^{t} \left( t-s \right)^{- \frac{1}{2}} ds \\
		&\leq Ct^{- \frac{n}{2} \left( 1- \frac{1}{q} \right) - \frac{n}{2} \left( p-1 \right) -k \sigma + \frac{1}{2}} \\
		&=Ct^{- \frac{n}{2} \left( 1- \frac{1}{q} \right) - \left( k+1 \right) \sigma}, \\
		\norm{R_{1, k, 5} \left( t \right)}_{q}
		&\leq C \int_{t/2}^{+ \infty} \norm{\nabla e^{t \Delta} \left( f \left( u \left( s \right) \right) -f \left( \AA_{0, k-1} \left( s \right) \right) \right)}_{q} ds \\
		&\leq Ct^{- \frac{n}{2} \left( 1- \frac{1}{q} \right) - \frac{1}{2}} \int_{t/2}^{+ \infty} \norm{f \left( u \left( s \right) \right) -f \left( \AA_{0, k-1} \left( s \right) \right)}_{1} ds \\
		&\leq Ct^{- \frac{n}{2} \left( 1- \frac{1}{q} \right) - \frac{1}{2}} \int_{t/2}^{+ \infty} \left( \norm{u \left( s \right)}_{\infty}^{p-1} + \norm{\AA_{0, k-1} \left( s \right)}_{\infty}^{p-1} \right) \norm{u \left( s \right) - \AA_{0, k-1} \left( s \right)}_{1} ds \\
		&\leq Ct^{- \frac{n}{2} \left( 1- \frac{1}{q} \right) - \frac{1}{2}} \int_{t/2}^{+ \infty} s^{- \frac{n}{2} \left( p-1 \right) -k \sigma} ds \\
		&=Ct^{- \frac{n}{2} \left( 1- \frac{1}{q} \right) - \frac{1}{2}} \int_{t/2}^{+ \infty} s^{- \frac{1}{2} - \left( k+1 \right) \sigma} ds \\
		&\leq Ct^{- \frac{n}{2} \left( 1- \frac{1}{q} \right) - \left( k+1 \right) \sigma}.
	\end{align*}
	In the case where $k=1$, we use Proposition \ref{pro:P_asymp-0} instead of Theorem \ref{th:P_asymp_sub_gene}.
	Combining these estimates, we arrive at the desired result.
\end{proof}

Theorem \ref{th:P_asymp_super_gene} follows from Proposition \ref{pro:heat_asymp_lim} and Lemma \ref{lem:P_asymp_super_gene}.

\begin{proof}[Proof of Theorem \ref{th:P_asymp_super_gene}]
	Let $q \in \left[ 1, + \infty \right]$.
	By taking into account the fact that
	\begin{align*}
		p>1+ \frac{k+2}{\left( k+1 \right) n} \quad \Longleftrightarrow \quad \frac{1}{2} < \left( k+1 \right) \sigma,
	\end{align*}
	Lemma \ref{lem:P_asymp_super_gene} implies
	\begin{align*}
		\lim_{t \to + \infty} t^{\frac{n}{2} \left( 1- \frac{1}{q} \right) + \frac{1}{2}} \norm{\int_{0}^{t} a \cdot \nabla e^{\left( t-s \right) \Delta} f \left( u \left( s \right) \right) ds- \left( \AA_{0, k} \left( t \right) - \AA_{0} \left( t \right) \right) - a \cdot \nabla e^{t \Delta} \psi_{0, k}}_{q} =0.
	\end{align*}
	Therefore, from \eqref{I} and Proposition \ref{pro:heat_asymp_lim}, we have
	\begin{align*}
		&t^{\frac{n}{2} \left( 1- \frac{1}{q} \right) + \frac{1}{2}} \norm{u \left( t \right) - \AA_{1, k} \left( t \right)}_{q} \\
		&\hspace{1cm} \leq t^{\frac{n}{2} \left( 1- \frac{1}{q} \right) + \frac{1}{2}} \norm{e^{t \Delta} u_{0} - \Lambda_{0, 1} \left( t; u_{0} \right)}_{q} \\
		&\hspace{2cm} + t^{\frac{n}{2} \left( 1- \frac{1}{q} \right) + \frac{1}{2}} \norm{\int_{0}^{t} a \cdot \nabla e^{\left( t-s \right) \Delta} f \left( u \left( s \right) \right) ds- \left( \AA_{0, k} \left( t \right) - \AA_{0} \left( t \right) \right) - a \cdot \nabla e^{t \Delta} \psi_{0, k}}_{q} \\
		&\hspace{2cm} + \sum_{j=1}^{n} \abs{a_{j}} t^{\frac{n}{2} \left( 1- \frac{1}{q} \right) + \frac{1}{2}} \norm{\partial_{j} e^{t \Delta} \psi_{0, k} - \Lambda_{e_{j}, 0} \left( t; \psi_{0, k} \right)}_{q} \\
		&\hspace{1cm} \to 0
	\end{align*}
	as $t \to + \infty$.
\end{proof}

%%%%%%%%%%%%%%%%%%%%%%%%%
\subsection{Proof of Theorem \ref{th:P_asymp_super_optimal}} \label{sec:P_asymp_super_optimal}
%%%%%%%%%%%%%%%%%%%%%%%%%

We start with the following lemma which describes the structure of the asymptotic profile given in Theorem \ref{th:P_asymp_super_gene}.

\begin{lem} \label{lem:P_asymp_super_structure}
	Under the same assumption as in Theorem \ref{th:P_asymp_super_gene}, the identity
	\begin{align*}
		\AA_{1, k} \left( t \right) - \AA_{0, k} \left( t \right) =t^{- \frac{1}{2}} \delta_{t} \left( \AA_{1, k} \left( 1 \right) - \AA_{0, k} \left( 1 \right) \right)
	\end{align*}
	holds for any $t \geq 1$.
	Moreover, $\AA_{1, k} \left( 1 \right) - \AA_{0, k} \left( 1 \right)$ is represented as
	\begin{align*}
		\AA_{1, k} \left( 1 \right) - \AA_{0, k} \left( 1 \right) = \frac{1}{2} \sum_{j=1}^{n} \left( \MM_{e_{j}} \left( u_{0} \right) -a_{j} \MM_{0} \left( \psi_{0, k} \right) \right) x_{j} G_{1}.
	\end{align*}
	In particular, $\AA_{1, k} \left( 1 \right) - \AA_{0, k} \left( 1 \right) \not\equiv 0$ if and only if there exists $j \in \left\{ 1, \ldots, n \right\}$ such that
	\begin{align*}
		\MM_{e_{j}} \left( u_{0} \right) -a_{j} \MM_{0} \left( \psi_{0, k} \right) \neq 0.
	\end{align*}
\end{lem}

\begin{rem} \label{rem:P_asymp_super_gene}
	Lemma \ref{lem:P_asymp_super_structure} implies that under the same assumption as in Theorem \ref{th:P_asymp_super_gene}, the estimate
	\begin{align*}
		t^{\frac{n}{2} \left( 1- \frac{1}{q} \right)} \norm{\AA_{1, k} \left( t \right) - \AA_{0, k} \left( t \right)}_{q} =t^{- \frac{1}{2}} \norm{\AA_{1, k} \left( 1 \right) - \AA_{0, k} \left( 1 \right)}_{q}
	\end{align*}
	holds for any $q \in \left[ 1, + \infty \right]$ and $t>1$.
\end{rem}

\begin{proof}[Proof of Lemma \ref{lem:P_asymp_super_structure}]
	For any $t \geq 1$, we have
	\begin{align*}
		\AA_{1, k} \left( t \right) - \AA_{0, k} \left( t \right) &= \frac{1}{2} t^{- \frac{1}{2}} \sum_{j=1}^{n} \MM_{e_{j}} \left( u_{0} \right) \delta_{t} \left( x_{j} G_{1} \right) + \sum_{j=1}^{n} a_{j} \Lambda_{e_{j}, 0} \left( t; \psi_{0, k} \right) \\
		&= \frac{1}{2} t^{- \frac{1}{2}} \sum_{j=1}^{n} \left( \MM_{e_{j}} \left( u_{0} \right) -a_{j} \MM_{0} \left( \psi_{0, k} \right) \right) \delta_{t} \left( x_{j} G_{1} \right) \\
		&=t^{- \frac{1}{2}} \delta_{t} \left[ \frac{1}{2} \sum_{j=1}^{n} \left( \MM_{e_{j}} \left( u_{0} \right) -a_{j} \MM_{0} \left( \psi_{0, k} \right) \right) x_{j} G_{1} \right].
	\end{align*}
	Substituting $t=1$, we obtain the explicit formula of $\AA_{1, k} \left( 1 \right) - \AA_{0, k} \left( 1 \right)$.
	The rest of the assertion in Lemma \ref{lem:P_asymp_super_structure} follows from the representation of $\AA_{1, k} \left( 1 \right) - \AA_{0, k} \left( 1 \right)$ and the orthogonality of the Hermite polynomials.
\end{proof}

By using Theorem \ref{th:P_asymp_super_gene} and Lemma \ref{lem:P_asymp_super_structure}, we can prove Theorem \ref{th:P_asymp_super_optimal}.

\begin{proof}[Proof of Theorem \ref{th:P_asymp_super_optimal}]
	By Lemma \ref{lem:P_asymp_super_structure}, we have
	\begin{align*}
		u \left( t \right) - \AA_{0, k} \left( t \right) &=u \left( t \right) - \AA_{1, k} \left( t \right) + \left( \AA_{1, k} \left( t \right) - \AA_{0, k} \left( t \right) \right) \\
		&=u \left( t \right) - \AA_{1, k} \left( t \right) +t^{- \frac{1}{2}} \delta_{t} \left( \AA_{1, k} \left( 1 \right) - \AA_{0, k} \left( 1 \right) \right)
	\end{align*}
	for any $t>1$.
	Theorem \ref{th:P_asymp_super_gene} implies
	\begin{align*}
		\limsup_{t \to + \infty} t^{\frac{n}{2} \left( 1- \frac{1}{q} \right) + \frac{1}{2}} \norm{u \left( t \right) - \AA_{0, k} \left( t \right)}_{q} &\leq \limsup_{t \to + \infty} t^{\frac{n}{2} \left( 1- \frac{1}{q} \right)} \norm{\delta_{t} \left( \AA_{1, k} \left( 1 \right) - \AA_{0, k} \left( 1 \right) \right)}_{q} \\
		&\hspace{1cm} + \limsup_{t \to + \infty} t^{\frac{n}{2} \left( 1- \frac{1}{q} \right) + \frac{1}{2}} \norm{u \left( t \right) - \AA_{1, k} \left( t \right)}_{q} \\
		&= \norm{\AA_{1, k} \left( 1 \right) - \AA_{0, k} \left( 1 \right)}_{q}, \\
		\liminf_{t \to + \infty} t^{\frac{n}{2} \left( 1- \frac{1}{q} \right) + \frac{1}{2}} \norm{u \left( t \right) - \AA_{0, k} \left( t \right)}_{q} &\geq \liminf_{t \to + \infty} t^{\frac{n}{2} \left( 1- \frac{1}{q} \right)} \norm{\delta_{t} \left( \AA_{1, k} \left( 1 \right) - \AA_{0, k} \left( 1 \right) \right)}_{q} \\
		&\hspace{1cm} - \limsup_{t \to + \infty} t^{\frac{n}{2} \left( 1- \frac{1}{q} \right) + \frac{1}{2}} \norm{u \left( t \right) - \AA_{1, k} \left( t \right)}_{q} \\
		&= \norm{\AA_{1, k} \left( 1 \right) - \AA_{0, k} \left( 1 \right)}_{q}
	\end{align*}
	for all $q \in \left[ 1, + \infty \right]$.
	Thus, we conclude that
	\begin{align*}
		\lim_{t \to + \infty} t^{\frac{n}{2} \left( 1- \frac{1}{q} \right) + \frac{1}{2}} \norm{u \left( t \right) - \AA_{0, k} \left( t \right)}_{q} = \norm{\AA_{1, k} \left( 1 \right) - \AA_{0, k} \left( 1 \right)}_{q}.
	\end{align*}
\end{proof}

%%%%%%%%%%%%%%%%%%%%%%%%%
\subsection{Proofs of Theorems \ref{th:P_asymp_sub-1_optimal} and \ref{th:P_asymp_critical-1_optimal}} \label{sec:P_asymp-1_optimal}
%%%%%%%%%%%%%%%%%%%%%%%%%

Under the same assumption as in Theorems \ref{th:P_asymp_sub-1_optimal} or \ref{th:P_asymp_critical-1_optimal}, we set
\begin{align*}
	\RR_{0, 1} \left( t \right) \coloneqq \AA_{0, 1} \left( t \right) - \AA_{0} \left( t \right)
\end{align*}
for each $t>0$.
Then, $\RR_{0, 1} \left( t \right)$ satisfies
\begin{align*}
	\RR_{0, 1} \left( t \right) =t^{- \sigma} \delta_{t} \RR_{0, 1} \left( 1 \right).
\end{align*}
Moreover, $\RR_{0, 1} \left( 1 \right)$ is represented as
\begin{align*}
	\RR_{0, 1} \left( 1 \right) =f \left( \MM_{0} \left( u_{0} \right) \right) \int_{0}^{1} a \cdot \nabla e^{\left( 1- \theta \right) \Delta} G_{\theta}^{p} d \theta,
\end{align*}
which is odd.
For the proof, see Lemma \ref{lem:P_asymp_F-sub_structure} in Appendix \ref{app:optimal}.

Since $f \left( \xi \right) = \abs{\xi}^{p-1} \xi$, we have $f' \left( \xi \right) =p \abs{\xi}^{p-1}$ and
\begin{align*}
	\abs{f' \left( \xi \right) -f' \left( \eta \right)} \leq \begin{cases}
		p \left( p-1 \right) \left( \abs{\xi}^{p-2} + \abs{\eta}^{p-2} \right) \abs{\xi - \eta} &\qquad \text{if} \quad p>2, \\
		p \abs{\xi - \eta}^{p-1} &\qquad \text{if} \quad 1<p \leq 2
	\end{cases}
\end{align*}
for all $\xi, \eta \in \R$.

\begin{proof}[Proof of Theorem \ref{th:P_asymp_sub-1_optimal}]
	Let $q \in \left[ 1, + \infty \right]$.
	We first show
	\begin{align}
		\label{eq:sub-1_1}
		\lim_{t \to + \infty} t^{\frac{n}{2} \left( 1- \frac{1}{q} \right) +2 \sigma} \norm{\AA_{0, 2} \left( t \right) - \AA_{0, 1} \left( t \right) - \int_{1}^{t} a \cdot \nabla e^{\left( t-s \right) \Delta} \left( \RR_{0, 1} \left( s \right) f' \left( \AA_{0} \left( s \right) \right) \right) ds}_{q} =0.
	\end{align}
	We note that
	\begin{align*}
		&\AA_{0, 2} \left( t \right) - \AA_{0, 1} \left( t \right) - \int_{1}^{t} a \cdot \nabla e^{\left( t-s \right) \Delta} \left( \RR_{0, 1} \left( s \right) f' \left( \AA_{0} \left( s \right) \right) \right) ds \\
		&\hspace{1cm} = \int_{1}^{t} a \cdot \nabla e^{\left( t-s \right) \Delta} \left[ f \left( \AA_{0, 1} \left( s \right) \right) -f \left( \AA_{0} \left( s \right) \right) - \RR_{0, 1} \left( s \right) f' \left( \AA_{0} \left( s \right) \right) \right] ds \\
		&\hspace{1cm} = \int_{1}^{t} a \cdot \nabla e^{\left( t-s \right) \Delta} \left[ \RR_{0, 1} \left( s \right) \int_{0}^{1} \left( f' \left( \AA_{0} \left( s \right) + \theta \RR_{0, 1} \left( s \right) \right) -f' \left( \AA_{0} \left( s \right) \right) \right) d \theta \right] ds.
	\end{align*}
	If $1<p \leq 2$, then it follows from Lemmas \ref{lem:heat_LpLq} and \ref{lem:P_asymp_sub_gene} that
	\begin{align*}
		&\norm{\AA_{0, 2} \left( t \right) - \AA_{0, 1} \left( t \right) - \int_{1}^{t} a \cdot \nabla e^{\left( t-s \right) \Delta} \left( \RR_{0, 1} \left( s \right) f' \left( \AA_{0} \left( s \right) \right) \right) ds}_{q} \\
		&\hspace{1cm} \leq C \left( \int_{1}^{t/2} + \int_{t/2}^{t} \right) \norm{\nabla e^{\left( t-s \right) \Delta} \left[ \RR_{0, 1} \left( s \right) \int_{0}^{1} \left( f' \left( \AA_{0} \left( s \right) + \theta \RR_{0, 1} \left( s \right) \right) -f' \left( \AA_{0} \left( s \right) \right) \right) d \theta \right]}_{q} ds \\
		&\hspace{1cm} \leq C \int_{1}^{t/2} \left( t-s \right)^{- \frac{n}{2} \left( 1- \frac{1}{q} \right) - \frac{1}{2}} \norm{\RR_{0, 1} \left( s \right) \int_{0}^{1} \left( f' \left( \AA_{0} \left( s \right) + \theta \RR_{0, 1} \left( s \right) \right) -f' \left( \AA_{0} \left( s \right) \right) \right) d \theta}_{1} ds \\
		&\hspace{2cm} +C \int_{t/2}^{t} \left( t-s \right)^{- \frac{1}{2}} \norm{\RR_{0, 1} \left( s \right) \int_{0}^{1} \left( f' \left( \AA_{0} \left( s \right) + \theta \RR_{0, 1} \left( s \right) \right) -f' \left( \AA_{0} \left( s \right) \right) \right) d \theta}_{q} ds \\
		&\hspace{1cm} \leq C \int_{1}^{t/2} \int_{0}^{1} \left( t-s \right)^{- \frac{n}{2} \left( 1- \frac{1}{q} \right) - \frac{1}{2}} \norm{\RR_{0, 1} \left( s \right)}_{1} \norm{f' \left( \AA_{0} \left( s \right) + \theta \RR_{0, 1} \left( s \right) \right) -f' \left( \AA_{0} \left( s \right) \right)}_{\infty} d \theta ds \\
		&\hspace{2cm} +C \int_{t/2}^{t} \int_{0}^{1} \left( t-s \right)^{- \frac{1}{2}} \norm{\RR_{0, 1} \left( s \right)}_{q} \norm{f' \left( \AA_{0} \left( s \right) + \theta \RR_{0, 1} \left( s \right) \right) -f' \left( \AA_{0} \left( s \right) \right)}_{\infty} d \theta ds \\
		&\hspace{1cm} \leq C \int_{1}^{t/2} \int_{0}^{1} \left( t-s \right)^{- \frac{n}{2} \left( 1- \frac{1}{q} \right) - \frac{1}{2}} \norm{\RR_{0, 1} \left( s \right)}_{1} \norm{\theta \RR_{0, 1} \left( s \right)}_{\infty}^{p-1} d \theta ds \\
		&\hspace{2cm} +C \int_{t/2}^{t} \int_{0}^{1} \left( t-s \right)^{- \frac{1}{2}} \norm{\RR_{0, 1} \left( s \right)}_{q} \norm{\theta \RR_{0, 1} \left( s \right)}_{\infty}^{p-1} d \theta ds \\
		&\hspace{1cm} \leq C \int_{1}^{t/2} \left( t-s \right)^{- \frac{n}{2} \left( 1- \frac{1}{q} \right) - \frac{1}{2}} s^{- \sigma - \frac{n}{2} \left( p-1 \right) - \sigma \left( p-1 \right)} ds \\
		&\hspace{2cm} +C \int_{t/2}^{t} \left( t-s \right)^{- \frac{1}{2}} s^{- \frac{n}{2} \left( 1- \frac{1}{q} \right) - \sigma - \frac{n}{2} \left( p-1 \right) - \sigma \left( p-1 \right)} ds \\
		&\hspace{1cm} \leq Ct^{- \frac{n}{2} \left( 1- \frac{1}{q} \right) - \frac{1}{2}} \int_{1}^{t/2} s^{- \frac{1}{2} - \sigma \left( p+1 \right)} ds+Ct^{- \frac{n}{2} \left( 1- \frac{1}{q} \right) - \frac{1}{2} - \sigma \left( p+1 \right)} \int_{t/2}^{t} \left( t-s \right)^{- \frac{1}{2}} ds \\
		&\hspace{1cm} \leq \begin{dcases}
			Ct^{- \frac{n}{2} \left( 1- \frac{1}{q} \right) - \sigma \left( p+1 \right)}, &\qquad \sigma \left( p+1 \right) < \frac{1}{2}, \\
			Ct^{- \frac{n}{2} \left( 1- \frac{1}{q} \right) - \frac{1}{2}} \log t+Ct^{- \frac{n}{2} \left( 1- \frac{1}{q} \right) - \sigma \left( p+1 \right)}, &\qquad \sigma \left( p+1 \right) = \frac{1}{2}, \\
			Ct^{- \frac{n}{2} \left( 1- \frac{1}{q} \right) - \frac{1}{2}} +Ct^{- \frac{n}{2} \left( 1- \frac{1}{q} \right) - \sigma \left( p+1 \right)}, &\qquad \sigma \left( p+1 \right) > \frac{1}{2}
		\end{dcases}
	\end{align*}
	for $t>2$.
	Similarly, if $p>2$, then we have
	\begin{align*}
		&\norm{\AA_{0, 2} \left( t \right) - \AA_{0, 1} \left( t \right) - \int_{1}^{t} a \cdot \nabla e^{\left( t-s \right) \Delta} \left( \RR_{0, 1} \left( s \right) f' \left( \AA_{0} \left( s \right) \right) \right) ds}_{q} \\
		&\hspace{1cm} \leq C \int_{1}^{t/2} \int_{0}^{1} \left( t-s \right)^{- \frac{n}{2} \left( 1- \frac{1}{q} \right) - \frac{1}{2}} \norm{\RR_{0, 1} \left( s \right)}_{1} \norm{f' \left( \AA_{0} \left( s \right) + \theta \RR_{0, 1} \left( s \right) \right) -f' \left( \AA_{0} \left( s \right) \right)}_{\infty} d \theta ds \\
		&\hspace{2cm} +C \int_{t/2}^{t} \int_{0}^{1} \left( t-s \right)^{- \frac{1}{2}} \norm{\RR_{0, 1} \left( s \right)}_{q} \norm{f' \left( \AA_{0} \left( s \right) + \theta \RR_{0, 1} \left( s \right) \right) -f' \left( \AA_{0} \left( s \right) \right)}_{\infty} d \theta ds \\
		&\hspace{1cm} \leq C \int_{1}^{t/2} \int_{0}^{1} \left( t-s \right)^{- \frac{n}{2} \left( 1- \frac{1}{q} \right) - \frac{1}{2}} \norm{\RR_{0, 1} \left( s \right)}_{1} \left( \norm{\AA_{0} \left( s \right)}_{\infty}^{p-2} + \norm{\theta \RR_{0, 1} \left( s \right)}_{\infty}^{p-2} \right) \norm{\theta \RR_{0, 1} \left( s \right)}_{\infty} d \theta ds \\
		&\hspace{2cm} +C \int_{t/2}^{t} \int_{0}^{1} \left( t-s \right)^{- \frac{1}{2}} \norm{\RR_{0, 1} \left( s \right)}_{q} \left( \norm{\AA_{0} \left( s \right)}_{\infty}^{p-2} + \norm{\theta \RR_{0, 1} \left( s \right)}_{\infty}^{p-2} \right) \norm{\theta \RR_{0, 1} \left( s \right)}_{\infty} d \theta ds \\
		&\hspace{1cm} \leq C \int_{1}^{t/2} \left( t-s \right)^{- \frac{n}{2} \left( 1- \frac{1}{q} \right) - \frac{1}{2}} s^{- \sigma - \frac{n}{2} \left( p-2 \right) - \frac{n}{2} - \sigma} ds \\
		&\hspace{2cm} +C \int_{t/2}^{t} \left( t-s \right)^{- \frac{1}{2}} s^{- \frac{n}{2} \left( 1- \frac{1}{q} \right) - \sigma - \frac{n}{2} \left( p-2 \right) - \frac{n}{2} - \sigma} ds \\
		&\hspace{1cm} \leq Ct^{- \frac{n}{2} \left( 1- \frac{1}{q} \right) - \frac{1}{2}} \int_{1}^{t/2} s^{- \frac{1}{2} -3 \sigma} ds+Ct^{- \frac{n}{2} \left( 1- \frac{1}{q} \right) - \frac{1}{2} -3 \sigma} \int_{t/2}^{t} \left( t-s \right)^{- \frac{1}{2}} ds \\
		&\hspace{1cm} \leq \begin{dcases}
			Ct^{- \frac{n}{2} \left( 1- \frac{1}{q} \right) -3 \sigma}, &\qquad 3 \sigma < \frac{1}{2}, \\
			Ct^{- \frac{n}{2} \left( 1- \frac{1}{q} \right) - \frac{1}{2}} \log t+Ct^{- \frac{n}{2} \left( 1- \frac{1}{q} \right) - 3 \sigma}, &\qquad 3 \sigma = \frac{1}{2}, \\
			Ct^{- \frac{n}{2} \left( 1- \frac{1}{q} \right) - \frac{1}{2}} +Ct^{- \frac{n}{2} \left( 1- \frac{1}{q} \right) - 3 \sigma}, &\qquad 3 \sigma > \frac{1}{2}
		\end{dcases}
	\end{align*}
	for $t>2$.
	Taking into account the fact that $p<1+3/ \left( 2n \right)$ if and only if $2 \sigma <1/2$, we can derive \eqref{eq:sub-1_1}.
	We next prove
	\begin{align}
		\label{eq:sub-1_2}
		\lim_{t \to + \infty} t^{\frac{n}{2} \left( 1- \frac{1}{q} \right) +2 \sigma} \norm{\int_{0}^{1} a \cdot \nabla e^{\left( t-s \right) \Delta} \left( \RR_{0, 1} \left( s \right) f' \left( \AA_{0} \left( s \right) \right) \right) ds}_{q} =0.
	\end{align}
	By Lemmas \ref{lem:heat_LpLq} and \ref{lem:P_asymp_sub_gene}, we obtain
	\begin{align*}
		\norm{\int_{0}^{1} a \cdot \nabla e^{\left( t-s \right) \Delta} \left( \RR_{0, 1} \left( s \right) f' \left( \AA_{0} \left( s \right) \right) \right) ds}_{q} &\leq C \int_{0}^{1} \norm{\nabla e^{\left( t-s \right) \Delta} \left( \RR_{0, 1} \left( s \right) f' \left( \AA_{0} \left( s \right) \right) \right)}_{q} ds \\
		&\leq C \int_{0}^{1} \left( t-s \right)^{- \frac{n}{2} \left( 1- \frac{1}{q} \right) - \frac{1}{2}} \norm{\RR_{0, 1} \left( s \right) f' \left( \AA_{0} \left( s \right) \right)}_{1} ds \\
		&\leq C \left( t-1 \right)^{- \frac{n}{2} \left( 1- \frac{1}{q} \right) - \frac{1}{2}} \int_{0}^{1} \norm{\RR_{0, 1} \left( s \right)}_{1} \norm{\AA_{0} \left( s \right)}_{\infty}^{p-1} ds \\
		&\leq Ct^{- \frac{n}{2} \left( 1- \frac{1}{q} \right) - \frac{1}{2}} \int_{0}^{1} s^{- \sigma - \frac{n}{2} \left( p-1 \right)} ds \\
		&=Ct^{- \frac{n}{2} \left( 1- \frac{1}{q} \right) - \frac{1}{2}} \int_{0}^{1} s^{- \frac{1}{2} -2 \sigma} ds \\
		&\leq Ct^{- \frac{n}{2} \left( 1- \frac{1}{q} \right) - \frac{1}{2}}
	\end{align*}
	for $t>2$, whence follows \eqref{eq:sub-1_2}.
	Since $f'$ is homogeneous of order $\left( p-1 \right)$, we have
	\begin{align}
		\label{eq:sub-1_3}
		&\int_{0}^{t} a \cdot \nabla e^{\left( t-s \right) \Delta} \left( \RR_{0, 1} \left( s \right) f' \left( \AA_{0} \left( s \right) \right) \right) ds \nonumber \\
		&\hspace{1cm} = \int_{0}^{t} a \cdot \nabla e^{\left( t-s \right) \Delta} \left( s^{- \sigma} \delta_{s} \RR_{0, 1} \left( 1 \right) f' \left( \delta_{s} \AA_{0} \left( 1 \right) \right) \right) ds \nonumber \\
		&\hspace{1cm} =t \int_{0}^{1} a \cdot \nabla e^{t \left( 1- \theta \right) \Delta} \left( \left( t \theta \right)^{- \sigma} \delta_{t \theta} \RR_{0, 1} \left( 1 \right) f' \left( \delta_{t \theta} \AA_{0} \left( 1 \right) \right) \right) d \theta \nonumber \\
		&\hspace{1cm} =t^{1- \sigma} \int_{0}^{1} a \cdot \nabla e^{t \left( 1- \theta \right) \Delta} \left( \theta^{- \sigma} \delta_{t} \delta_{\theta} \RR_{0, 1} \left( 1 \right) f' \left( \delta_{t} \delta_{\theta} \AA_{0} \left( 1 \right) \right) \right) d \theta \nonumber \\
		&\hspace{1cm} =t^{1- \sigma - \frac{n}{2} \left( p-1 \right) + \frac{n}{2}} \int_{0}^{1} a \cdot \nabla e^{t \left( 1- \theta \right) \Delta} \left( \delta_{t} \RR_{0, 1} \left( \theta \right) \delta_{t} \left( f' \left( \AA_{0} \left( \theta \right) \right) \right) \right) d \theta \nonumber \\
		&\hspace{1cm} =t^{\frac{1}{2} -2 \sigma} \int_{0}^{1} a \cdot \nabla e^{t \left( 1- \theta \right) \Delta} \delta_{t} \left( \RR_{0, 1} \left( \theta \right) f' \left( \AA_{0} \left( \theta \right) \right) \right) d \theta \nonumber \\
		&\hspace{1cm} =t^{-2 \sigma} \delta_{t} \left( \int_{0}^{1} a \cdot \nabla e^{\left( 1- \theta \right) \Delta} \left( \RR_{0, 1} \left( \theta \right) f' \left( \AA_{0} \left( \theta \right) \right) \right) d \theta \right) \nonumber \\
		&\hspace{1cm} =t^{-2 \sigma} \delta_{t} \RR_{0, 2}^{*}.
	\end{align}
	Combining Theorem \ref{th:P_asymp_sub_gene} with $k=2$, \eqref{eq:sub-1_1}, \eqref{eq:sub-1_2}, and \eqref{eq:sub-1_3} yields
	\begin{align*}
		&\limsup_{t \to + \infty} t^{\frac{n}{2} \left( 1- \frac{1}{q} \right) +2 \sigma} \norm{u \left( t \right) - \AA_{0, 1} \left( t \right)}_{q} \\
		&\hspace{1cm} \leq \limsup_{t \to + \infty} t^{\frac{n}{2} \left( 1- \frac{1}{q} \right) +2 \sigma} \norm{\AA_{0, 2} \left( t \right) - \AA_{0, 1} \left( t \right)}_{q} + \limsup_{t \to + \infty} t^{\frac{n}{2} \left( 1- \frac{1}{q} \right) +2 \sigma} \norm{u \left( t \right) - \AA_{0, 2} \left( t \right)}_{q} \\
		&\hspace{1cm} = \limsup_{t \to + \infty} t^{\frac{n}{2} \left( 1- \frac{1}{q} \right) +2 \sigma} \norm{\AA_{0, 2} \left( t \right) - \AA_{0, 1} \left( t \right)}_{q} \\
		&\hspace{1cm} \leq \limsup_{t \to + \infty} t^{\frac{n}{2} \left( 1- \frac{1}{q} \right) +2 \sigma} \norm{\int_{0}^{t} a \cdot \nabla e^{\left( t-s \right) \Delta} \left( \RR_{0, 1} \left( s \right) f' \left( \AA_{0} \left( s \right) \right) \right) ds}_{q} \\
		&\hspace{2cm} + \limsup_{t \to + \infty} t^{\frac{n}{2} \left( 1- \frac{1}{q} \right) +2 \sigma} \norm{\int_{0}^{1} a \cdot \nabla e^{\left( t-s \right) \Delta} \left( \RR_{0, 1} \left( s \right) f' \left( \AA_{0} \left( s \right) \right) \right) ds}_{q} \\
		&\hspace{2cm} + \limsup_{t \to + \infty} t^{\frac{n}{2} \left( 1- \frac{1}{q} \right) +2 \sigma} \norm{\AA_{0, 2} \left( t \right) - \AA_{0, 1} \left( t \right) - \int_{1}^{t} a \cdot \nabla e^{\left( t-s \right) \Delta} \left( \RR_{0, 1} \left( s \right) f' \left( \AA_{0} \left( s \right) \right) \right) ds}_{q} \\
		&\hspace{1cm} = \limsup_{t \to + \infty} t^{\frac{n}{2} \left( 1- \frac{1}{q} \right)} \norm{\delta_{t} \RR_{0, 2}^{*}}_{q} \\
		&\hspace{1cm} = \norm{\RR_{0, 2}^{*}}_{q}, \\
		&\liminf_{t \to + \infty} t^{\frac{n}{2} \left( 1- \frac{1}{q} \right) +2 \sigma} \norm{u \left( t \right) - \AA_{0, 1} \left( t \right)}_{q} \\
		&\hspace{1cm} \geq \liminf_{t \to + \infty} t^{\frac{n}{2} \left( 1- \frac{1}{q} \right) +2 \sigma} \norm{\AA_{0, 2} \left( t \right) - \AA_{0, 1} \left( t \right)}_{q} - \limsup_{t \to + \infty} t^{\frac{n}{2} \left( 1- \frac{1}{q} \right) +2 \sigma} \norm{u \left( t \right) - \AA_{0, 2} \left( t \right)}_{q} \\
		&\hspace{1cm} = \liminf_{t \to + \infty} t^{\frac{n}{2} \left( 1- \frac{1}{q} \right) +2 \sigma} \norm{\AA_{0, 2} \left( t \right) - \AA_{0, 1} \left( t \right)}_{q} \\
		&\hspace{1cm} \geq \liminf_{t \to + \infty} t^{\frac{n}{2} \left( 1- \frac{1}{q} \right) +2 \sigma} \norm{\int_{0}^{t} a \cdot \nabla e^{\left( t-s \right) \Delta} \left( \RR_{0, 1} \left( s \right) f' \left( \AA_{0} \left( s \right) \right) \right) ds}_{q} \\
		&\hspace{2cm} - \limsup_{t \to + \infty} t^{\frac{n}{2} \left( 1- \frac{1}{q} \right) +2 \sigma} \norm{\int_{0}^{1} a \cdot \nabla e^{\left( t-s \right) \Delta} \left( \RR_{0, 1} \left( s \right) f' \left( \AA_{0} \left( s \right) \right) \right) ds}_{q} \\
		&\hspace{2cm} - \limsup_{t \to + \infty} t^{\frac{n}{2} \left( 1- \frac{1}{q} \right) +2 \sigma} \norm{\AA_{0, 2} \left( t \right) - \AA_{0, 1} \left( t \right) - \int_{1}^{t} a \cdot \nabla e^{\left( t-s \right) \Delta} \left( \RR_{0, 1} \left( s \right) f' \left( \AA_{0} \left( s \right) \right) \right) ds}_{q} \\
		&\hspace{1cm} = \liminf_{t \to + \infty} t^{\frac{n}{2} \left( 1- \frac{1}{q} \right)} \norm{\delta_{t} \RR_{0, 2}^{*}}_{q} \\
		&\hspace{1cm} = \norm{\RR_{0, 2}^{*}}_{q}.
	\end{align*}
	As a consequence, we can deduce the desired result.
\end{proof}

\begin{rem} \label{rem:S_02}
	Let $1+1/n<p<1+3/ \left( 2n \right)$.
	We define a function $S_{0, 2}^{*} \colon \R^{n} \to \R$ by
	\begin{align*}
		S_{0, 2}^{*} \left( x \right) \coloneqq \left( \int_{0}^{1} a \cdot \nabla e^{\left( 1- \theta \right) \Delta} \left( G_{\theta}^{p-1} \int_{0}^{\theta} a \cdot \nabla e^{\left( \theta - \tau \right) \Delta} G_{\tau}^{p} d \tau \right) d \theta \right) \left( x \right), \qquad x \in \R^{n}.
	\end{align*}
	By a simple computation, we see that $S_{0, 2}^{*}$ is continuous and bounded.
	In addition, it is represented as
	\begin{align*}
		S_{0, 2}^{*} \left( x \right) =- \int_{0}^{1} \int_{\R^{n}} \frac{a \cdot \left( x-y \right)}{2 \left( 1- \theta \right)} G_{1- \theta} \left( x-y \right) G_{\theta} \left( y \right)^{p-1} S_{0, 1} \left( \theta, y \right) dyd \theta,
	\end{align*}
	where
	\begin{align*}
		S_{0, 1} \left( \theta, y \right) \coloneqq - \frac{1}{2} p^{- \frac{n}{2}} \left( a \cdot y \right) \int_{0}^{\theta} \left( 4 \pi \tau \right)^{- \frac{n}{2} \left( p-1 \right)} \frac{1}{\theta - \tau + \frac{\tau}{p}} G_{\theta - \tau + \tau /p} \left( y \right) d \tau.
	\end{align*}
	Substituting $x=0$ yields
	\begin{align*}
		S_{0, 2}^{*} \left( 0 \right) = \int_{0}^{1} \int_{\R^{n}} \frac{a \cdot y}{2 \left( 1- \theta \right)} G_{1- \theta} \left( y \right) G_{\theta} \left( y \right)^{p-1} S_{0, 1} \left( \theta, y \right) dyd \theta.
	\end{align*}
	Since
	\begin{align*}
		\left( a \cdot y \right) S_{0, 1} \left( \theta, y \right) =- \frac{1}{2} p^{- \frac{n}{2}} \left( a \cdot y \right)^{2} \int_{0}^{\theta} \left( 4 \pi \tau \right)^{- \frac{n}{2} \left( p-1 \right)} \frac{1}{\theta - \tau + \frac{\tau}{p}} G_{\theta - \tau + \tau /p} \left( y \right) d \tau \leq 0
	\end{align*}
	for any $\left( \theta, y \right) \in \left( 0, 1 \right) \times \R^{n}$, we have $S_{0, 2}^{*} \left( 0 \right) \leq 0$.
	Taking into the assumption that $a \neq 0$ and the fact that $\left( a \cdot y \right) S_{0, 1} \left( \theta, y \right) <0$ if and only if $a \cdot y \neq 0$, we obtain $S_{0, 2}^{*} \left( 0 \right) <0$, which in turn implies $S_{0, 2}^{*} \not\equiv 0$.
\end{rem}

We next give the proof of Theorem \ref{th:P_asymp_critical-1_optimal}.

\begin{proof}[Proof of Theorem \ref{th:P_asymp_critical-1_optimal}]
	We note that
	\begin{align*}
		\widetilde{\AA}_{0, 2} \left( t \right) - \AA_{0, 1} \left( t \right) =- \frac{1}{2} \left( \int_{1}^{t} \MM_{0} \left( f \left( \AA_{0, 1} \left( s \right) \right) - f \left( \AA_{0} \left( s \right) \right) \right) ds \right) t^{- \frac{1}{2}} \sum_{j=1}^{n} a_{j} \delta_{t} \left( x_{j} G_{1} \right).
	\end{align*}
	We first show that there exists $C>0$ such that
	\begin{align}
		\label{eq:critical}
		\sup_{t>1} \abs{\int_{1}^{t} \MM_{0} \left( f \left( \AA_{0, 1} \left( s \right) \right) - f \left( \AA_{0} \left( s \right) \right) \right) ds} \leq C.
	\end{align}
	Since $f'$ is homogeneous of order $\left( p-1 \right)$ and $p=1+3/ \left( 2n \right) \Leftrightarrow \sigma =1/4$, we have
	\begin{align*}
		\RR_{0, 1} \left( s \right) f' \left( \AA_{0} \left( s \right) \right)
		&=s^{- \frac{1}{4}} \delta_{s} \RR_{0, 1} \left( 1 \right) f' \left( \delta_{s} \AA_{0} \left( 1 \right) \right) \\
		&=s^{- \frac{1}{4} - \frac{n}{2} \left( p-1 \right) + \frac{n}{2}} \delta_{s} \RR_{0, 1} \left( 1 \right) \delta_{s} \left( f' \left( \AA_{0} \left( 1 \right) \right) \right) \\
		&=s^{- \frac{1}{4} - \frac{n}{2} \left( p-1 \right)} \delta_{s} \left[ \RR_{0, 1} \left( 1 \right) f' \left( \AA_{0} \left( 1 \right) \right) \right] \\
		&=s^{-1} \delta_{s} \left[ \RR_{0, 1} \left( 1 \right) f' \left( \AA_{0} \left( 1 \right) \right) \right]
	\end{align*}
	for any $s>1$.
	In addition, from the fact that the function $\RR_{0, 1} \left( 1 \right) f' \left( \AA_{0} \left( 1 \right) \right)$ on $\R^{n}$, represented as
	\begin{align*}
		\left( \RR_{0, 1} \left( 1 \right) f' \left( \AA_{0} \left( 1 \right) \right) \right) \left( x \right) =f \left( \MM_{0} \left( u_{0} \right) \right) f' \left( \MM_{0} \left( u_{0} \right) \right) G_{1} \left( x \right)^{p-1} \left( \int_{0}^{1} a \cdot \nabla e^{\left( 1- \theta \right) \Delta} G_{\theta}^{p} d \theta \right) \left( x \right),
	\end{align*}
	is odd, we see that $\MM_{0} \left( \RR_{0, 1} \left( 1 \right) f' \left( \AA_{0} \left( 1 \right) \right) \right) =0$.
	Hence, we obtain
	\begin{align*}
		&\int_{1}^{t} \MM_{0} \left( f \left( \AA_{0, 1} \left( s \right) \right) -f \left( \AA_{0} \left( s \right) \right) \right) ds \\
		&\hspace{1cm} = \int_{1}^{t} \MM_{0} \left( f \left( \AA_{0, 1} \left( s \right) \right) - f \left( \AA_{0} \left( s \right) \right) \right) ds- \left( \log t \right) \MM_{0} \left( \RR_{0, 1} \left( 1 \right) f' \left( \AA_{0} \left( 1 \right) \right) \right) \\
		&\hspace{1cm} = \int_{1}^{t} \left\{ \MM_{0} \left( f \left( \AA_{0, 1} \left( s \right) \right) - f \left( \AA_{0} \left( s \right) \right) \right) - \MM_{0} \left( s^{-1} \delta_{s} \left[ \RR_{0, 1} \left( 1 \right) f' \left( \AA_{0} \left( 1 \right) \right) \right] \right) \right\} ds \\
		&\hspace{1cm} = \int_{1}^{t} \MM_{0} \left( f \left( \AA_{0, 1} \left( s \right) \right) - f \left( \AA_{0} \left( s \right) \right) - \RR_{0, 1} \left( s \right) f' \left( \AA_{0} \left( s \right) \right) \right) ds \\
		&\hspace{1cm} = \int_{1}^{t} \MM_{0} \left( \RR_{0, 1} \left( s \right) \int_{0}^{1} \left( f' \left( \AA_{0} \left( s \right) + \theta \RR_{0, 1} \left( s \right) \right) -f' \left( \AA_{0} \left( s \right) \right) \right) d \theta \right) ds
	\end{align*}
	for any $t>1$.
	By Lemma \ref{lem:P_asymp_sub_gene}, we get
	\begin{align*}
		&\abs{\int_{1}^{t} \MM_{0} \left( f \left( \AA_{0, 1} \left( s \right) \right) - f \left( \AA_{0} \left( s \right) \right) \right) ds} \\
		&\hspace{1cm} \leq \int_{1}^{t} \norm{\RR_{0, 1} \left( s \right) \int_{0}^{1} \left( f' \left( \AA_{0} \left( s \right) + \theta \RR_{0, 1} \left( s \right) \right) -f' \left( \AA_{0} \left( s \right) \right) \right) d \theta}_{1} ds \\
		&\hspace{1cm} \leq \int_{1}^{t} \int_{0}^{1} \norm{\RR_{0, 1} \left( s \right)}_{1} \norm{f' \left( \AA_{0} \left( s \right) + \theta \RR_{0, 1} \left( s \right) \right) -f' \left( \AA_{0} \left( s \right) \right)}_{\infty} d \theta ds \\
		&\hspace{1cm} \leq \begin{dcases}
			C \int_{1}^{t} \int_{0}^{1} \norm{\RR_{0, 1} \left( s \right)}_{1} \left( \norm{\AA_{0} \left( s \right)}_{\infty}^{p-2} + \norm{\theta \RR_{0, 1} \left( s \right)}_{\infty}^{p-2} \right) \norm{\theta \RR_{0, 1} \left( s \right)}_{\infty} d \theta ds, &\quad p>2, \\
			C \int_{1}^{t} \int_{0}^{1} \norm{\RR_{0, 1} \left( s \right)}_{1} \norm{\theta \RR_{0, 1} \left( s \right)}_{\infty}^{p-1} d \theta ds, &\quad 1<p \leq 2
		\end{dcases} \\
		&\hspace{1cm} \leq \begin{dcases}
			C \int_{1}^{t} s^{- \frac{1}{4} - \frac{n}{2} \left( p-2 \right) - \frac{n}{2} - \frac{1}{4}} ds=C \int_{1}^{t} s^{-1- \frac{1}{4}} ds, &\qquad p>2, \\
			C \int_{1}^{t} s^{- \frac{1}{4} - \frac{n}{2} \left( p-1 \right) - \frac{1}{4} \left( p-1 \right)} ds=C \int_{1}^{t} s^{-1- \frac{1}{4} \left( p-1 \right)} ds, &\qquad 1<p \leq 2 \\
		\end{dcases} \\
		&\hspace{1cm} \leq C.
	\end{align*}
	Combining Theorem \ref{th:P_asymp_critical_gene} with $k=1$ and \eqref{eq:critical} yields
	\begin{align*}
		&\norm{u \left( t \right) - \AA_{0, 1} \left( t \right)}_{q} \\
		&\hspace{1cm} \leq \norm{u \left( t \right) - \widetilde{\AA}_{0, 2} \left( t \right)}_{q} + \norm{\widetilde{\AA}_{0, 2} \left( t \right) - \AA_{0, 1} \left( t \right)}_{q} \\
		&\hspace{1cm} \leq Ct^{- \frac{n}{2} \left( 1- \frac{1}{q} \right) - \frac{1}{2}} +Ct^{- \frac{1}{2}} \abs{\int_{1}^{t} \MM_{0} \left( f \left( \AA_{0, 1} \left( s \right) \right) - f \left( \AA_{0} \left( s \right) \right) \right) ds} \sum_{j=1}^{n} \norm{ \delta_{t} \left( x_{j} G_{1} \right)}_{q} \\
		&\hspace{1cm} \leq Ct^{- \frac{n}{2} \left( 1- \frac{1}{q} \right) - \frac{1}{2}}
	\end{align*}
	for any $q \in \left[ 1, + \infty \right]$ and $t>4$.
\end{proof}

%%%%%%%%%%%%%%%%%%%%%%%%%
\appendix
\section{Proofs of Propositions \ref{pro:P_asymp_F-critical} and \ref{pro:P_asymp_F-super}} \label{app:Zuazua}
%%%%%%%%%%%%%%%%%%%%%%%%%
\subsection{Proof of Proposition \ref{pro:P_asymp_F-critical}}
%%%%%%%%%%%%%%%%%%%%%%%%%
First of all, we remark that if $u \in X$ satisfies \eqref{eq:P_decay}, then we have
\begin{align}
	\label{eq:P_decay_reg}
	\sup_{q \in \left[ 1, + \infty \right]} \sup_{t>0} \left( 1+t \right)^{\frac{n}{2} \left( 1- \frac{1}{q} \right)} \norm{u \left( t \right)}_{q} <+ \infty.
\end{align}
Let $q \in \left[ 1, + \infty \right]$ and let $t>2$.
Using \eqref{I}, we decompose the remainder into seven parts:
\begin{align*}
	u \left( t \right) - \widetilde{\AA}_{0, 1} \left( t \right)
	&= \left( e^{t \Delta} u_{0} - \Lambda_{0, 0} \left( t; u_{0} \right) \right) + \int_{0}^{1} a \cdot \nabla e^{\left( t-s \right) \Delta} f \left( u \left( s \right) \right) ds+ \int_{t/2}^{t} a \cdot \nabla e^{\left( t-s \right) \Delta} f \left( u \left( s \right) \right) ds \\
	&\hspace{1cm} + \int_{1}^{t/2} a \cdot \nabla e^{\left( t-s \right) \Delta} \left( f \left( u \left( s \right) \right) -f \left( \AA_{0} \left( s \right) \right) \right) ds \\
	&\hspace{1cm} + \sum_{j=1}^{n} a_{j} \int_{1}^{t/2} \int_{0}^{1} \left( -s \Delta \right) \partial_{j} e^{\left( t-s \theta \right) \Delta} f \left( \AA_{0} \left( s \right) \right) d \theta ds \\
	&\hspace{1cm} + \sum_{j=1}^{n} a_{j} \int_{1}^{t/2} \left( \partial_{j} e^{t \Delta} f \left( \AA_{0} \left( s \right) \right) - \Lambda_{e_{j}, 0} \left( t; f \left( \AA_{0} \left( s \right) \right) \right) \right) ds \\
	&\hspace{1cm} - \sum_{j=1}^{n} a_{j} \int_{t/2}^{t} \Lambda_{e_{j}, 0} \left( t; f \left( \AA_{0} \left( s \right) \right) \right) ds \\
	&\eqqcolon \sum_{\ell =1}^{7} \widetilde{R}_{0, 1, \ell} \left( t \right).
\end{align*}
By Lemma \ref{lem:heat_LpLq} and \eqref{eq:P_decay_reg}, we have
\begin{align*}
	\norm{\widetilde{R}_{0, 1, 2} \left( t \right)}_{q}
	&\leq C \int_{0}^{1} \norm{\nabla e^{\left( t-s \right) \Delta} f \left( u \left( s \right) \right)}_{q} ds \\
	&\leq C \int_{0}^{1} \left( t-s \right)^{- \frac{n}{2} \left( 1- \frac{1}{q} \right) - \frac{1}{2}} \norm{u \left( s \right)}_{p}^{p} ds \\
	&\leq C \left( t-1 \right)^{- \frac{n}{2} \left( 1- \frac{1}{q} \right) - \frac{1}{2}} \int_{0}^{1} \left( 1+s \right)^{- \frac{n}{2} \left( p-1 \right)} ds \\
	&\leq Ct^{- \frac{n}{2} \left( 1- \frac{1}{q} \right) - \frac{1}{2}}, \\
	\norm{\widetilde{R}_{0, 1, 3} \left( t \right)}_{q}
	&\leq C \int_{t/2}^{t} \norm{\nabla e^{\left( t-s \right) \Delta} f \left( u \left( s \right) \right)}_{q} ds \\
	&\leq C \int_{t/2}^{t} \left( t-s \right)^{- \frac{1}{2}} \norm{u \left( s \right)}_{pq}^{p} ds \\
	&\leq C \int_{t/2}^{t} \left( t-s \right)^{- \frac{1}{2}} s^{- \frac{n}{2} \left( 1- \frac{1}{q} \right) - \frac{n}{2} \left( p-1 \right)} ds \\
	&\leq Ct^{- \frac{n}{2} \left( 1- \frac{1}{q} \right) - \frac{n}{2} \left( p-1 \right)} \int_{t/2}^{t} \left( t-s \right)^{- \frac{1}{2}} ds \\
	&\leq Ct^{- \frac{n}{2} \left( 1- \frac{1}{q} \right) - \frac{n}{2} \left( p-1 \right) + \frac{1}{2}} \\
	&=Ct^{- \frac{n}{2} \left( 1- \frac{1}{q} \right) - \frac{1}{2}}, \\
	\norm{\widetilde{R}_{0, 1, 5} \left( t \right)}_{q}
	&\leq C \sum_{j=1}^{n} \int_{1}^{t/2} \int_{0}^{1} s \norm{\partial_{j} \Delta e^{\left( t-s \theta \right) \Delta} f \left( \AA_{0} \left( s \right) \right)}_{q} d \theta ds \\
	&\leq C \int_{1}^{t/2} \int_{0}^{1} s \left( t-s \theta \right)^{- \frac{n}{2} \left( 1- \frac{1}{q} \right) - \frac{3}{2}} \norm{\AA_{0} \left( s \right)}_{p}^{p} d \theta ds \\
	&\leq Ct^{- \frac{n}{2} \left( 1- \frac{1}{q} \right) - \frac{3}{2}} \int_{1}^{t/2} s^{1- \frac{n}{2} \left( p-1 \right)} ds \\
	&=Ct^{- \frac{n}{2} \left( 1- \frac{1}{q} \right) - \frac{3}{2}} \int_{1}^{t/2} ds \\
	&\leq Ct^{- \frac{n}{2} \left( 1- \frac{1}{q} \right) - \frac{1}{2}}, \\
	\norm{\widetilde{R}_{0, 1, 7} \left( t \right)}_{q}
	&\leq C \sum_{j=1}^{n} \int_{t/2}^{t} \norm{\Lambda_{e_{j}, 0} \left( f \left( t; \AA_{0} \left( s \right) \right) \right)}_{q} ds \\
	&\leq Ct^{- \frac{1}{2}} \sum_{j=1}^{n} \int_{t/2}^{t} \norm{\AA_{0} \left( s \right)}_{p}^{p} \norm{\delta_{t} \left( x_{j} G_{1} \right)}_{q} ds \\
	&\leq Ct^{- \frac{n}{2} \left( 1- \frac{1}{q} \right) - \frac{1}{2}} \int_{t/2}^{t} s^{- \frac{n}{2} \left( p-1 \right)} ds \\
	&=Ct^{- \frac{n}{2} \left( 1- \frac{1}{q} \right) - \frac{1}{2}} \int_{t/2}^{t} s^{-1} ds \\
	&\leq Ct^{- \frac{n}{2} \left( 1- \frac{1}{q} \right) - \frac{1}{2}}.
\end{align*}
It follows from Proposition \ref{pro:heat_asymp} that
\begin{align*}
	\norm{\widetilde{R}_{0, 1, 1} \left( t \right)}_{q}
	&\leq Ct^{- \frac{n}{2} \left( 1- \frac{1}{q} \right) - \frac{1}{2}} \sum_{\abs{\alpha} =1} \norm{x^{\alpha} u_{0}}_{1}, \\
	\norm{\widetilde{R}_{0, 1, 6} \left( t \right)}_{q}
	&\leq Ct^{- \frac{n}{2} \left( 1- \frac{1}{q} \right) -1} \sum_{\abs{\alpha} =1} \int_{1}^{t/2} \norm{\AA_{0} \left( s \right)}_{\infty}^{p-1} \norm{x^{\alpha} \AA_{0} \left( s \right)}_{1} ds \\
	&\leq Ct^{- \frac{n}{2} \left( 1- \frac{1}{q} \right) -1} \int_{1}^{t/2} s^{- \frac{n}{2} \left( p-1 \right) + \frac{1}{2}} ds \\
	&=Ct^{- \frac{n}{2} \left( 1- \frac{1}{q} \right) -1} \int_{1}^{t/2} s^{- \frac{1}{2}} ds \\
	&\leq Ct^{- \frac{n}{2} \left( 1- \frac{1}{q} \right) - \frac{1}{2}}.
\end{align*}
By \eqref{eq:P_decay}, Proposition \ref{pro:P_asymp-0}, and Lemma \ref{lem:heat_LpLq}, we obtain
\begin{align*}
	\norm{\widetilde{R}_{0, 1, 4} \left( t \right)}_{q}
	&\leq C \int_{1}^{t/2} \norm{\nabla e^{\left( t-s \right) \Delta} \left( f \left( u \left( s \right) \right) -f \left( \AA_{0} \left( s \right) \right) \right)}_{q} ds \\
	&\leq C \int_{1}^{t/2} \left( t-s \right)^{- \frac{n}{2} \left( 1- \frac{1}{q} \right) - \frac{1}{2}} \norm{f \left( u \left( s \right) \right) -f \left( \AA_{0} \left( s \right) \right)}_{1} ds \\
	&\leq Ct^{- \frac{n}{2} \left( 1- \frac{1}{q} \right) - \frac{1}{2}} \int_{1}^{t/2} \left( \norm{u \left( s \right)}_{\infty}^{p-1} + \norm{\AA_{0} \left( s \right)}_{\infty}^{p-1} \right) \norm{u \left( s \right) - \AA_{0} \left( s \right)}_{1} ds \\
	&\leq Ct^{- \frac{n}{2} \left( 1- \frac{1}{q} \right) - \frac{1}{2}} \int_{1}^{t/2} s^{- \frac{n}{2} \left( p-1 \right) - \frac{1}{2}} \log \left( 2+s \right) ds \\
	&=Ct^{- \frac{n}{2} \left( 1- \frac{1}{q} \right) - \frac{1}{2}} \int_{1}^{t/2} s^{- \frac{3}{2}} \log \left( 2+s \right) ds \\
	&\leq Ct^{- \frac{n}{2} \left( 1- \frac{1}{q} \right) - \frac{1}{2}} \int_{1}^{t/2} s^{- \frac{5}{4}} ds \\
	&\leq Ct^{- \frac{n}{2} \left( 1- \frac{1}{q} \right) - \frac{1}{2}}.
\end{align*}
Here, we have used the fact that
\begin{align*}
	\sup_{s>1} s^{- \frac{1}{4}} \log \left( 2+s \right) <+ \infty.
\end{align*}
Combining these estimates, we arrive at the desired result. \qed

%%%%%%%%%%%%%%%%%%%%%%%%%
\subsection{Proof of Proposition \ref{pro:P_asymp_F-super}}
%%%%%%%%%%%%%%%%%%%%%%%%%

We first prove the following lemma which describes the asymptotic behavior of the Duhamel term in \eqref{I}.

\begin{lem} \label{lem:P_asymp_F-super}
	Under the same assumption as in Proposition \ref{pro:P_asymp_F-super},
	\begin{align*}
		\lim_{t \to + \infty} t^{\frac{n}{2} \left( 1- \frac{1}{q} \right) + \frac{1}{2}} \norm{\int_{0}^{t} a \cdot \nabla e^{\left( t-s \right) \Delta} f \left( u \left( s \right) \right) ds- a \cdot \nabla e^{t \Delta} \psi_{0, 0}}_{q} =0
	\end{align*}
	holds for any $q \in \left[ 1, + \infty \right]$.
\end{lem}

\begin{proof}
	It follows from \eqref{eq:P_decay_reg} that
	\begin{align*}
		\norm{\psi_{0, 0}}_{1} \leq C \int_{0}^{+ \infty} \norm{u \left( s \right)}_{p}^{p} ds \leq C \int_{0}^{+ \infty} \left( 1+s \right)^{- \frac{n}{2} \left( p-1 \right)} ds \leq C,
	\end{align*}
	which implies $\psi_{0, 0} \in L^{1} \left( \R^{n} \right)$.
	Let $q \in \left[ 1, + \infty \right]$ and let $t>1$.
	We decompose the remainder into three parts:
	\begin{align*}
		\int_{0}^{t} a \cdot \nabla e^{\left( t-s \right) \Delta} f \left( u \left( s \right) \right) ds- a \cdot \nabla e^{t \Delta} \psi_{0, 0}
		&= \sum_{j=1}^{n} a_{j} \int_{0}^{t/2} \int_{0}^{1} \left( -s \Delta \right) \partial_{j} e^{\left( t-s \theta \right) \Delta} f \left( u \left( s \right) \right) d \theta ds \\
		&\hspace{1cm} + \int_{t/2}^{t} a \cdot \nabla e^{\left( t-s \right) \Delta} f \left( u \left( s \right) \right) ds \\
		&\hspace{1cm} - \int_{t/2}^{+ \infty} a \cdot \nabla e^{t \Delta} f \left( u \left( s \right) \right) ds \\
		&\eqqcolon \sum_{\ell =1}^{3} R_{1, 0, \ell} \left( t \right).
	\end{align*}
	By Lemma \ref{lem:heat_LpLq} and \eqref{eq:P_decay_reg}, each term is estimated as
	\begin{align*}
		\norm{R_{1, 0, 1} \left( t \right)}_{q}
		&\leq C \sum_{j=1}^{n} \int_{0}^{t/2} \int_{0}^{1} s \norm{\partial_{j} \Delta e^{\left( t-s \theta \right) \Delta} f \left( u \left( s \right) \right)}_{q} d \theta ds \\
		&\leq C \int_{0}^{t/2} \int_{0}^{1} s \left( t-s \theta \right)^{- \frac{n}{2} \left( 1- \frac{1}{q} \right) - \frac{3}{2}} \norm{u \left( s \right)}_{p}^{p} d \theta ds \\
		&\leq Ct^{- \frac{n}{2} \left( 1- \frac{1}{q} \right) - \frac{3}{2}} \int_{0}^{t/2} \left( 1+s \right)^{1- \frac{n}{2} \left( p-1 \right)} ds \\
		&\leq Ct^{- \frac{n}{2} \left( 1- \frac{1}{q} \right) - \frac{3}{2}} \times \begin{dcases}
			\left( 1+t \right)^{2- \frac{n}{2} \left( p-1 \right)}, &\qquad p<1+ \frac{4}{n}, \\
			\log \left( 2+t \right), &\qquad p=1+ \frac{4}{n}, \\
			1, &\qquad p>1+ \frac{4}{n}
		\end{dcases} \\
		&\leq Ct^{- \frac{n}{2} \left( 1- \frac{1}{q} \right) - \frac{1}{2}} \times \begin{dcases}
			t^{- \left( \sigma - \frac{1}{2} \right)}, &\qquad p<1+ \frac{4}{n}, \\
			t^{-1} \log \left( 2+t \right), &\qquad p=1+ \frac{4}{n}, \\
			t^{-1}, &\qquad p>1+ \frac{4}{n},
		\end{dcases} \\
		\norm{R_{1, 0, 2} \left( t \right)}_{q}
		&\leq C \int_{t/2}^{t} \norm{\nabla e^{\left( t-s \right) \Delta} f \left( u \left( s \right) \right)}_{q} ds \\
		&\leq C \int_{t/2}^{t} \left( t-s \right)^{- \frac{1}{2}} \norm{u \left( s \right)}_{pq}^{p} ds \\
		&\leq C \int_{t/2}^{t} \left( t-s \right)^{- \frac{1}{2}} s^{- \frac{n}{2} \left( 1- \frac{1}{q} \right) - \frac{n}{2} \left( p-1 \right)} ds \\
		&\leq Ct^{- \frac{n}{2} \left( 1- \frac{1}{q} \right) - \frac{n}{2} \left( p-1 \right)} \int_{t/2}^{t} \left( t-s \right)^{- \frac{1}{2}} ds \\
		&\leq Ct^{- \frac{n}{2} \left( 1- \frac{1}{q} \right) - \frac{n}{2} \left( p-1 \right) + \frac{1}{2}} \\
		&=Ct^{- \frac{n}{2} \left( 1- \frac{1}{q} \right) - \sigma}, \\
		\norm{R_{1, 0, 3} \left( t \right)}_{q}
		&\leq C \int_{t/2}^{+ \infty} \norm{\nabla e^{t \Delta} f \left( u \left( s \right) \right)}_{q} ds \\
		&\leq Ct^{- \frac{n}{2} \left( 1- \frac{1}{q} \right) - \frac{1}{2}} \int_{t/2}^{+ \infty} \norm{u \left( s \right)}_{p}^{p} ds \\
		&\leq Ct^{- \frac{n}{2} \left( 1- \frac{1}{q} \right) - \frac{1}{2}} \int_{t/2}^{+ \infty} s^{- \frac{n}{2} \left( p-1 \right)} ds \\
		&\leq Ct^{- \frac{n}{2} \left( 1- \frac{1}{q} \right) + \frac{1}{2} - \frac{n}{2} \left( p-1 \right)} \\
		&=Ct^{- \frac{n}{2} \left( 1- \frac{1}{q} \right) - \sigma}.
	\end{align*}
	Since $p>1+2/n \Leftrightarrow \sigma >1/2$, we conclude that
	\begin{align*}
		\lim_{t \to + \infty} t^{\frac{n}{2} \left( 1- \frac{1}{q} \right) + \frac{1}{2}} \norm{\int_{0}^{t} a \cdot \nabla e^{\left( t-s \right) \Delta} f \left( u \left( s \right) \right) ds- a \cdot \nabla e^{t \Delta} \psi_{0, 0}}_{q} =0.
	\end{align*}
\end{proof}

We next give the proof of Proposition \ref{pro:P_asymp_F-super}.

\begin{proof}[Proof of Proposition \ref{pro:P_asymp_F-super}]
	By virtue of \eqref{I}, Proposition \ref{pro:heat_asymp_lim}, and Lemma \ref{lem:P_asymp_F-super}, we have
	\begin{align*}
		t^{\frac{n}{2} \left( 1- \frac{1}{q} \right) + \frac{1}{2}} \norm{u \left( t \right) - \AA_{1, 0} \left( t \right)}_{q} &\leq t^{\frac{n}{2} \left( 1- \frac{1}{q} \right) + \frac{1}{2}} \norm{e^{t \Delta} u_{0} - \Lambda_{0, 1} \left( t; u_{0} \right)}_{q} \\
		&\hspace{1cm} +t^{\frac{n}{2} \left( 1- \frac{1}{q} \right) + \frac{1}{2}} \norm{\int_{0}^{t} a \cdot \nabla e^{\left( t-s \right) \Delta} f \left( u \left( s \right) \right) ds- a \cdot \nabla e^{t \Delta} \psi_{0, 0}}_{q} \\
		&\hspace{1cm} + \sum_{j=1}^{n} \abs{a_{j}} t^{\frac{n}{2} \left( 1- \frac{1}{q} \right) + \frac{1}{2}} \norm{\partial_{j} e^{t \Delta} \psi_{0, 0} - \Lambda_{e_{j}, 0} \left( t; \psi_{0, 0} \right)}_{q} \\
		&\to 0
	\end{align*}
	as $t \to + \infty$.
\end{proof}

%%%%%%%%%%%%%%%%%%%%%%%%%
\section{Optimal decay rate of the 0th order asymptotic expansion} \label{app:optimal}
%%%%%%%%%%%%%%%%%%%%%%%%%

This section is devoted to proving the optimality for the decay rates of the remainder $u \left( t \right) - \AA_{0} \left( t \right)$ given in Proposition \ref{pro:P_asymp-0}.
The self-similarity of the asymptotic profiles given by Propositions \ref{pro:P_asymp_F-sub}, \ref{pro:P_asymp_F-critical}, and \ref{pro:P_asymp_F-super} plays a crucial role in the proof of the optimality.
To obtain this property in the Fujita-subcritical and Fujita-critical cases, we need to assume that $f$ is homogeneous of order $p$, namely,
\begin{align*}
	f \left( \lambda \xi \right) = \lambda^{p} f \left( \xi \right)
\end{align*}
for any $\lambda >0$ and $\xi \in \R$.

%%%%%%%%%%%%%%%%%%%%%%%%%
\subsection{Fujita-subcritical case}
%%%%%%%%%%%%%%%%%%%%%%%%%

\begin{prop} \label{pro:P_asymp_F-sub_optimal}
	In addition to the assumption in Proposition \ref{pro:P_asymp_F-sub}, we suppose that $f$ is homogeneous of order $p$.
	Then,
	\begin{align*}
		\lim_{t \to + \infty} t^{\frac{n}{2} \left( 1- \frac{1}{q} \right) + \sigma} \norm{u \left( t \right) - \AA_{0} \left( t \right)}_{q} &= \norm{\AA_{0, 1} \left( 1 \right) - \AA_{0} \left( 1 \right)}_{q}
	\end{align*}
	holds for any $q \in \left[ 1, + \infty \right]$.
	In particular, if $f \left( \MM_{0} \left( u_{0} \right) \right) \neq 0$, then $\AA_{0, 1} \left( 1 \right) - \AA_{0} \left( 1 \right) \not\equiv 0$, and therefore
	\begin{align*}
		t^{\frac{n}{2} \left( 1- \frac{1}{q} \right)} \norm{u \left( t \right) - \AA_{0} \left( t \right)}_{q} =t^{- \sigma} \norm{\AA_{0, 1} \left( 1 \right) - \AA_{0} \left( 1 \right)}_{q} \left( 1+o \left( 1 \right) \right)
	\end{align*}
	for all $q \in \left[ 1, + \infty \right]$ as $t \to + \infty$.
\end{prop}

\begin{lem} \label{lem:P_asymp_F-sub_structure}
	Under the same assumption as in Proposition \ref{pro:P_asymp_F-sub_optimal}, the identity
	\begin{align*}
		\AA_{0, 1} \left( t \right) - \AA_{0} \left( t \right) =t^{- \sigma} \delta_{t} \left( \AA_{0, 1} \left( 1 \right) - \AA_{0} \left( 1 \right) \right)
	\end{align*}
	holds for any $t>0$.
	Furthermore, $\AA_{0, 1} \left( 1 \right) - \AA_{0} \left( 1 \right)$ is represented as
	\begin{align*}
		\AA_{0, 1} \left( 1 \right) - \AA_{0} \left( 1 \right) &= \int_{0}^{1} a \cdot \nabla e^{\left( 1- \theta \right) \Delta} f \left( \AA_{0} \left( \theta \right) \right) d \theta \\
		&=f \left( \MM_{0} \left( u_{0} \right) \right) \int_{0}^{1} a \cdot \nabla e^{\left( 1- \theta \right) \Delta} G_{\theta}^{p} d \theta.
	\end{align*}
	In particular, $\AA_{0, 1} \left( 1 \right) - \AA_{0} \left( 1 \right) \not\equiv 0$ if and only if $f \left( \MM_{0} \left( u_{0} \right) \right) \neq 0$.
\end{lem}

\begin{rem}
	Let $1+1/n<p<1+2/n$.
	We define a function $S_{0, 1} \colon \R^{n} \to \R$ by
	\begin{align*}
		S_{0, 1} \left( x \right) \coloneqq \left( \int_{0}^{1} a \cdot \nabla e^{\left( 1- \theta \right) \Delta} G_{\theta}^{p} d \theta \right) \left( x \right), \qquad x \in \R^{n}.
	\end{align*}
	By a simple computation, we see that $S_{0, 1}$ is continuous and bounded.
	Moreover, it is represented as
	\begin{align*}
		S_{0, 1} \left( x \right) &=a \cdot \nabla \left( p^{- \frac{n}{2}} \int_{0}^{1} \left( 4 \pi \theta \right)^{- \frac{n}{2} \left( p-1 \right)} G_{1- \theta + \theta /p} \left( x \right) d \theta \right) \\
		&=- \frac{1}{2} p^{- \frac{n}{2}} \left( a \cdot x \right) \int_{0}^{1} \left( 4 \pi \theta \right)^{- \frac{n}{2} \left( p-1 \right)} \frac{1}{1- \theta + \frac{\theta}{p}} G_{1- \theta + \theta /p} \left( x \right) d \theta.
	\end{align*}
	Taking $x=a$, we obtain
	\begin{align*}
		S_{0, 1} \left( a \right) =- \frac{1}{2} p^{- \frac{n}{2}} \abs{a}^{2} \int_{0}^{1} \left( 4 \pi \theta \right)^{- \frac{n}{2} \left( p-1 \right)} \frac{1}{1- \theta + \frac{\theta}{p}} G_{1- \theta + \theta /p} \left( a \right) d \theta.
	\end{align*}
	Since $a \neq 0$ and the integrand on the right hand side of the above identity is strictly positive, we have $S_{0, 1} \left( a \right) <0$, which in turn implies $S_{0, 1} \not\equiv 0$.
\end{rem}

\begin{proof}[Proof of Lemma \ref{lem:P_asymp_F-sub_structure}]
	For any $t>0$, we have
	\begin{align*}
		\AA_{0, 1} \left( t \right) - \AA_{0} \left( t \right)
		&= \int_{0}^{t} a \cdot \nabla e^{\left( t-s \right) \Delta} f \left( \AA_{0} \left( s \right) \right) ds \\
		&=a \cdot \nabla \int_{0}^{t} \left( \delta_{t-s} G_{1} \right) \ast f \left( \delta_{s} \AA_{0} \left( 1 \right) \right) ds \\
		&=ta \cdot \nabla \int_{0}^{1} \left( \delta_{t \left( 1- \theta \right)} G_{1} \right) \ast f \left( \delta_{t \theta} \AA_{0} \left( 1 \right) \right) d \theta \\
		&=ta \cdot \nabla \int_{0}^{1} \left( \delta_{t} \delta_{1- \theta} G_{1} \right) \ast f \left( \delta_{t} \delta_{\theta} \AA_{0} \left( 1 \right) \right) d \theta \\
		&=ta \cdot \nabla \int_{0}^{1} \left( \delta_{t} G_{1- \theta} \right) \ast \left( t^{- \frac{n}{2} p+ \frac{n}{2}} \delta_{t} \left( f \left( \AA_{0} \left( \theta \right) \right) \right) \right) d \theta \\
		&=t^{1- \frac{n}{2} \left( p-1 \right)} a \cdot \nabla \int_{0}^{1} \delta_{t} \left( G_{1- \theta} \ast f \left( \AA_{0} \left( \theta \right) \right) \right) d \theta \\
		&=t^{\frac{1}{2} - \frac{n}{2} \left( p-1 \right)} \delta_{t} \left( \int_{0}^{1} a \cdot \nabla \left( G_{1- \theta} \ast f \left( \AA_{0} \left( \theta \right) \right) \right) d \theta \right) \\
		&=t^{- \sigma} \delta_{t} \left( \int_{0}^{1} a \cdot \nabla e^{\left( 1- \theta \right) \Delta} f \left( \AA_{0} \left( \theta \right) \right) d \theta \right) \\
		&=t^{- \sigma} \delta_{t} \left( \int_{0}^{1} a \cdot \nabla e^{\left( 1- \theta \right) \Delta} f \left( \MM_{0} \left( u_{0} \right) G_{\theta} \right) d \theta \right) \\
		&=f \left( \MM_{0} \left( u_{0} \right) \right) t^{- \sigma} \delta_{t} \left( \int_{0}^{1} a \cdot \nabla e^{\left( 1- \theta \right) \Delta} G_{\theta}^{p} d \theta \right).
	\end{align*}
	Substituting $t=1$, we obtain the explicit formula of $\AA_{0, 1} \left( 1 \right) - \AA_{0} \left( 1 \right)$.
\end{proof}

\begin{proof}[Proof of Proposition \ref{pro:P_asymp_F-sub_optimal}]
	By Lemma \ref{lem:P_asymp_F-sub_structure}, we have
	\begin{align*}
		u \left( t \right) - \AA_{0} \left( t \right) &= u \left( t \right) - \AA_{0, 1} \left( t \right) + \left( \AA_{0, 1} \left( t \right) - \AA_{0} \left( t \right) \right) \\
		&=u \left( t \right) - \AA_{0, 1} \left( t \right) +t^{- \sigma} \delta_{t} \left( \AA_{0, 1} \left( 1 \right) - \AA_{0} \left( 1 \right) \right)
	\end{align*}
	for any $t>0$.
	Proposition \ref{pro:P_asymp_F-sub} implies
	\begin{align*}
		\limsup_{t \to + \infty} t^{\frac{n}{2} \left( 1- \frac{1}{q} \right) + \sigma} \norm{u \left( t \right) - \AA_{0} \left( t \right)}_{q}
		&\leq \limsup_{t \to + \infty} t^{\frac{n}{2} \left( 1- \frac{1}{q} \right)} \norm{\delta_{t} \left( \AA_{0, 1} \left( 1 \right) - \AA_{0} \left( 1 \right) \right)}_{q} \\
		&\hspace{1cm} + \limsup_{t \to + \infty} t^{\frac{n}{2} \left( 1- \frac{1}{q} \right) + \sigma} \norm{u \left( t \right) - \AA_{0, 1} \left( t \right)}_{q} \\
		&= \norm{\AA_{0, 1} \left( 1 \right) - \AA_{0} \left( 1 \right)}_{q}, \\
		\liminf_{t \to + \infty} t^{\frac{n}{2} \left( 1- \frac{1}{q} \right) + \sigma} \norm{u \left( t \right) - \AA_{0} \left( t \right)}_{q}
		&\geq \liminf_{t \to + \infty} t^{\frac{n}{2} \left( 1- \frac{1}{q} \right)} \norm{\delta_{t} \left( \AA_{0, 1} \left( 1 \right) - \AA_{0} \left( 1 \right) \right)}_{q} \\
		&\hspace{1cm} - \limsup_{t \to + \infty} t^{\frac{n}{2} \left( 1- \frac{1}{q} \right) + \sigma} \norm{u \left( t \right) - \AA_{0, 1} \left( t \right)}_{q} \\
		&= \norm{\AA_{0, 1} \left( 1 \right) - \AA_{0} \left( 1 \right)}_{q}
	\end{align*}
	for any $q \in \left[ 1, + \infty \right]$, whence follows
	\begin{align*}
		\lim_{t \to + \infty} t^{\frac{n}{2} \left( 1- \frac{1}{q} \right) + \sigma} \norm{u \left( t \right) - \AA_{0} \left( t \right)}_{q} = \norm{\AA_{0, 1} \left( 1 \right) - \AA_{0} \left( 1 \right)}_{q}.
	\end{align*}
\end{proof}

%%%%%%%%%%%%%%%%%%%%%%%%%
\subsection{Fujita-critical case}
%%%%%%%%%%%%%%%%%%%%%%%%%

\begin{prop} \label{pro:P_asymp_F-critical_optimal}
	In addition to the assumption in Proposition \ref{pro:P_asymp_F-critical}, we suppose that $f$ is homogeneous of order $p$.
	Then,
	\begin{align*}
		\lim_{t \to + \infty} \frac{t^{\frac{n}{2} \left( 1- \frac{1}{q} \right) + \frac{1}{2}}}{\log t} \norm{u \left( t \right) - \AA_{0} \left( t \right)}_{q} = \bigl\lVert \widetilde{\RR}_{0, 1}^{*} \bigr\rVert_{q}
	\end{align*}
	holds for any $q \in \left[ 1, + \infty \right]$, where
	\begin{align*}
		\widetilde{\RR}_{0, 1}^{*} &\coloneqq \sum_{j=1}^{n} a_{j} \Lambda_{e_{j}, 0} \left( 1; f \left( \AA_{0} \left( 1 \right) \right) \right).
	\end{align*}
	In particular, if $f \left( \MM_{0} \left( u_{0} \right) \right) \neq 0$, then $\widetilde{\RR}_{0, 1}^{*} \not\equiv 0$, and hence
	\begin{align*}
		t^{\frac{n}{2} \left( 1- \frac{1}{q} \right)} \norm{u \left( t \right) - \AA_{0} \left( t \right)}_{q} =t^{- \frac{1}{2}} \left( \log t \right) \bigl\lVert \widetilde{\RR}_{0, 1}^{*} \bigr\rVert_{q} \left( 1+o \left( 1 \right) \right)
	\end{align*}
	for all $q \in \left[ 1, + \infty \right]$ as $t \to + \infty$.
\end{prop}

\begin{lem} \label{lem:P_asymp_F-critical_structure}
	Under the same assumption as in Proposition \ref{pro:P_asymp_F-critical_optimal}, the identity
	\begin{align*}
		\widetilde{\AA}_{0, 1} \left( t \right) - \AA_{0} \left( t \right) =t^{- \frac{1}{2}} \left( \log t \right) \delta_{t} \widetilde{\RR}_{0, 1}^{*}
	\end{align*}
	holds for any $t>1$.
	Moreover, $\widetilde{\RR}_{0, 1}^{*}$ is represented as
	\begin{align*}
		\widetilde{\RR}_{0, 1}^{*} =- \frac{1}{8 \pi} \left( 1+ \frac{2}{n} \right)^{- \frac{n}{2}} f \left( \MM_{0} \left( u_{0} \right) \right) \sum_{j=1}^{n} a_{j} x_{j} G_{1}.
	\end{align*}
	In particular, $\widetilde{\RR}_{0, 1}^{*} \not\equiv 0$ if and only if $f \left( \MM_{0} \left( u_{0} \right) \right) \neq 0$.
\end{lem}

\begin{proof}
	For any $t>1$, we have
	\begin{align*}
		\int_{1}^{t} \MM_{0} \left( f \left( \AA_{0} \left( s \right) \right) \right) ds
		&= \int_{1}^{t} \MM_{0} \left( f \left( \delta_{s} \AA_{0} \left( 1 \right) \right) \right) ds \\
		&= \int_{1}^{t} s^{- \frac{n}{2} p+ \frac{n}{2}} \MM_{0} \left( \delta_{s} \left( f \left( \AA_{0} \left( 1 \right) \right) \right) \right) ds \\
		&= \MM_{0} \left( f \left( \AA_{0} \left( 1 \right) \right) \right) \int_{1}^{t} s^{-1} ds \\
		&= \MM_{0} \left( f \left( \AA_{0} \left( 1 \right) \right) \right) \log t,
	\end{align*}
	whence follows
	\begin{align*}
		\widetilde{\AA}_{0, 1} \left( t \right) - \AA_{0} \left( t \right) &= \sum_{j=1}^{n} a_{j} \int_{1}^{t} \Lambda_{e_{j}, 0} \left( t; f \left( \AA_{0} \left( s \right) \right) \right) ds \\
		&= \sum_{j=1}^{n} a_{j} \left[ - \frac{1}{2} t^{- \frac{1}{2}} \left( \int_{1}^{t} \MM_{0} \left( f \left( \AA_{0} \left( s \right) \right) \right) ds \right) \delta_{t} \left( x_{j} G_{1} \right) \right] \\
		&=t^{- \frac{1}{2}} \left( \log t \right) \sum_{j=1}^{n} a_{j} \left[ - \frac{1}{2} \MM_{0} \left( f \left( \AA_{0} \left( 1 \right) \right) \right) \delta_{t} \left( x_{j} G_{1} \right) \right] \\
		&=t^{- \frac{1}{2}} \left( \log t \right) \delta_{t} \left( \sum_{j=1}^{n} a_{j} \left[ - \frac{1}{2} \MM_{0} \left( f \left( \AA_{0} \left( 1 \right) \right) \right) x_{j} G_{1} \right] \right) \\
		&=t^{- \frac{1}{2}} \left( \log t \right) \delta_{t} \left( \sum_{j=1}^{n} a_{j} \Lambda_{e_{j}, 0} \left( 1; f \left( \AA_{0} \left( 1 \right) \right) \right) \right) \\
		&=t^{- \frac{1}{2}} \left( \log t \right) \delta_{t} \widetilde{\RR}_{0, 1}^{*}.
	\end{align*}
	Moreover, we obtain
	\begin{align*}
		\widetilde{\RR}_{0, 1}^{*} &=- \frac{1}{2} \MM_{0} \left( f \left( \MM_{0} \left( u_{0} \right) G_{1} \right) \right) \sum_{j=1}^{n} a_{j} x_{j} G_{1} \\
		&=- \frac{1}{2} \MM_{0} \left( G_{1}^{p} f \left( \MM_{0} \left( u_{0} \right) \right) \right) \sum_{j=1}^{n} a_{j} x_{j} G_{1} \\
		&=- \frac{1}{2} \norm{G_{1}}_{p}^{p} f \left( \MM_{0} \left( u_{0} \right) \right) \sum_{j=1}^{n} a_{j} x_{j} G_{1} \\
		&=- \frac{1}{8 \pi} \left( 1+ \frac{2}{n} \right)^{- \frac{n}{2}} f \left( \MM_{0} \left( u_{0} \right) \right) \sum_{j=1}^{n} a_{j} x_{j} G_{1}.
	\end{align*}
\end{proof}

\begin{proof}[Proof of Proposition \ref{pro:P_asymp_F-critical_optimal}]
	By Lemma \ref{lem:P_asymp_F-critical_structure}, we obtain
	\begin{align*}
		u \left( t \right) - \AA_{0} \left( t \right) &=u \left( t \right) - \widetilde{\AA}_{0, 1} \left( t \right) + \left( \widetilde{\AA}_{0, 1} \left( t \right) - \AA_{0} \left( t \right) \right) \\
		&=u \left( t \right) - \widetilde{\AA}_{0, 1} \left( t \right) +t^{- \frac{1}{2}} \left( \log t \right) \delta_{t} \widetilde{\RR}_{0, 1}^{*}
	\end{align*}
	for any $t>1$.
	It follows from Proposition \ref{pro:P_asymp_F-critical} that
	\begin{align*}
		\limsup_{t \to + \infty} \frac{t^{\frac{n}{2} \left( 1- \frac{1}{q} \right) + \frac{1}{2}}}{\log t} \norm{u \left( t \right) - \AA_{0} \left( t \right)}_{q}
		&\leq \limsup_{t \to + \infty} t^{\frac{n}{2} \left( 1- \frac{1}{q} \right)} \bigl\lVert \delta_{t} \widetilde{\RR}_{0, 1}^{*} \bigr\rVert_{q} \\
		&\hspace{1cm} + \limsup_{t \to + \infty} \frac{t^{\frac{n}{2} \left( 1- \frac{1}{q} \right) + \frac{1}{2}}}{\log t} \norm{u \left( t \right) - \widetilde{\AA}_{0, 1} \left( t \right)}_{q} \\
		&= \bigl\lVert \widetilde{\RR}_{0, 1}^{*} \bigr\rVert_{q}, \\
		\liminf_{t \to + \infty} \frac{t^{\frac{n}{2} \left( 1- \frac{1}{q} \right) + \frac{1}{2}}}{\log t} \norm{u \left( t \right) - \AA_{0} \left( t \right)}_{q}
		&\geq \liminf_{t \to + \infty} t^{\frac{n}{2} \left( 1- \frac{1}{q} \right)} \bigl\lVert \delta_{t} \widetilde{\RR}_{0, 1}^{*} \bigr\rVert_{q} \\
		&\hspace{1cm} - \limsup_{t \to + \infty} \frac{t^{\frac{n}{2} \left( 1- \frac{1}{q} \right) + \frac{1}{2}}}{\log t} \norm{u \left( t \right) - \widetilde{\AA}_{0, 1} \left( t \right)}_{q} \\
		&= \bigl\lVert \widetilde{\RR}_{0, 1}^{*} \bigr\rVert_{q}
	\end{align*}
	for any $q \in \left[ 1, + \infty \right]$, which implies
	\begin{align*}
		\lim_{t \to + \infty} \frac{t^{\frac{n}{2} \left( 1- \frac{1}{q} \right) + \frac{1}{2}}}{\log t} \norm{u \left( t \right) - \AA_{0} \left( t \right)}_{q} = \bigl\lVert \widetilde{\RR}_{0, 1}^{*} \bigr\rVert_{q}.
	\end{align*}
\end{proof}

%%%%%%%%%%%%%%%%%%%%%%%%%
\subsection{Fujita-supercritical case}
%%%%%%%%%%%%%%%%%%%%%%%%%

\begin{prop} \label{pro:P_asymp_F-super_optimal}
	Under the same assumption as in Proposition \ref{pro:P_asymp_F-super},
	\begin{align*}
		\lim_{t \to + \infty} t^{\frac{n}{2} \left( 1- \frac{1}{q} \right) + \frac{1}{2}} \norm{u \left( t \right) - \AA_{0} \left( t \right)}_{q} = \norm{\AA_{1, 0} \left( 1 \right) - \AA_{0} \left( 1 \right)}_{q}
	\end{align*}
	holds for any $q \in \left[ 1, + \infty \right]$.
	In particular, if there exists $j \in \left\{ 1, \ldots, n \right\}$ such that
	\begin{align*}
		\MM_{e_{j}} \left( u_{0} \right) -a_{j} \MM_{0} \left( \psi_{0, 0} \right) \neq 0,
	\end{align*}
	then $\AA_{1, 0} \left( 1 \right) - \AA_{0} \left( 1 \right) \not\equiv 0$, and therefore
	\begin{align*}
		t^{\frac{n}{2} \left( 1- \frac{1}{q} \right)} \norm{u \left( t \right) - \AA_{0} \left( t \right)}_{q} =t^{- \frac{1}{2}} \norm{\AA_{1, 0} \left( 1 \right) - \AA_{0} \left( 1 \right)}_{q} \left( 1+o \left( 1 \right) \right)
	\end{align*}
	for all $q \in \left[ 1, + \infty \right]$ as $t \to + \infty$.
\end{prop}

We can prove this proposition with the same argument as in the proof of Theorem \ref{th:P_asymp_super_optimal}.
We remark that under the same assumption as in Proposition \ref{pro:P_asymp_F-super_optimal}, the identity
\begin{align*}
	\AA_{1, 0} \left( t \right) - \AA_{0} \left( t \right) =t^{- \frac{1}{2}} \delta_{t} \left( \AA_{1, 0} \left( 1 \right) - \AA_{0} \left( 1 \right) \right)
\end{align*}
holds for any $t>0$ and that $\AA_{1, 0} \left( 1 \right) - \AA_{0} \left( 1 \right)$ is represented as
\begin{align*}
	\AA_{1, 0} \left( 1 \right) - \AA_{0} \left( 1 \right) = \frac{1}{2} \sum_{j=1}^{n} \left( \MM_{e_{j}} \left( u_{0} \right) -a_{j} \MM_{0} \left( \psi_{0, 0} \right) \right) x_{j} G_{1}.
\end{align*}
In particular, $\AA_{1, 0} \left( 1 \right) - \AA_{0} \left( 1 \right) \not\equiv 0$ if and only if there exists $j \in \left\{ 1, \ldots, n \right\}$ such that
\begin{align*}
	\MM_{e_{j}} \left( u_{0} \right) -a_{j} \MM_{0} \left( \psi_{0, 0} \right) \neq 0.
\end{align*}

%%%%%%%%%%%%%%%%%%%%%%%%%
\section*{Acknowledgments}
%%%%%%%%%%%%%%%%%%%%%%%%%

This work was supported by Grant-in-Aid for JSPS Fellows \# JP24KJ2084.
The author is deeply grateful to Professor Vladimir Georgiev for his valuable comments and suggestions on this paper.
The author would also like to thank the anonymous referees for their helpful comments.

%%%%%%%%%%%%%%%%%%%%%%%%%

%%%%%%%%%%%%%%%%%%%%%%%%%
\end{document}